\documentclass[12pt]{amsart}

\usepackage[cp1251]{inputenc}
\usepackage[T2A]{fontenc}
\usepackage{amsfonts}
\usepackage{amsbsy}
\usepackage{amssymb}
\usepackage[english]{babel}
\usepackage{wrapfig}
\usepackage{longtable}
\usepackage{hyperref}

\renewcommand{\abovecaptionskip}{0pt}
\renewcommand{\belowcaptionskip}{6pt}
\makeatletter

\renewcommand{\@makecaption}[2]{
\vspace{\abovecaptionskip}%
\sbox{\@tempboxa}{#1 #2}%
\global\@minipagefalse \hbox to \hsize {{\scshape \hfil #1 #2\hfil}}
\vspace{\belowcaptionskip}}

\renewcommand{\LT@makecaption}[3]{%
  \LT@mcol\LT@cols c{\hbox to\z@{\hss\parbox[t]\LTcapwidth{%
    \sbox\@tempboxa{#1{{\normalsize \scshape#2 }}{\normalsize #3}}%
    \ifdim\wd\@tempboxa>\hsize
      #1{#2: }#3%
    \else
      \hbox to\hsize{\hfil\box\@tempboxa\hfil}%
    \fi
    \endgraf\vskip\belowcaptionskip}%
  \hss}}}

\makeatother

\newcommand{\mb}{\mathbb}
\newcommand{\mf}{\mathfrak}
\newcommand{\ms}{\mathsf}
\newcommand{\mr}{\mathrm}

\newcommand{\reg}{\mr{reg}}

\renewcommand{\ge}{\geqslant}
\renewcommand{\le}{\leqslant}

\newcommand{\al}{\alpha}
\newcommand{\be}{\beta}
\newcommand{\ga}{\gamma}
\newcommand{\de}{\delta}
\newcommand{\la}{\lambda}
\newcommand{\vf}{\varphi}

\newcommand{\De}{\Delta}
\newcommand{\si}{\sigma}

\newcommand{\Ker}{\operatorname{Ker}}

\newcommand{\rk}{\operatorname{rk}}

\newcommand{\doubleprec}{\mathrel{\mbox{$\prec$\hspace{-0.6em}$\prec$}}}

\DeclareMathOperator{\otimesZ}{\otimes\hspace{1pt}\rule[-3pt]{0pt}{0pt}_{\mathbb{Z}}}
\DeclareMathOperator{\Ad}{Ad} \DeclareMathOperator{\Supp}{Supp}
\DeclareMathOperator{\hgt}{ht}

\newtheorem{theorem}{Theorem}

\newtheorem{proposition}{Proposition}
\newtheorem{lemma}{Lemma}
\newtheorem{corollary}{Corollary}

\newtheorem*{question*}{Question}

\theoremstyle{definition}

\newtheorem{dfn}{Definition}
\newtheorem*{dfn*}{Definition}

\theoremstyle{remark}

\newtheorem{remark}{Remark}

\begin{document}

\renewcommand{\proofname}{Proof}
\renewcommand{\abstractname}{Abstract}
\renewcommand{\refname}{References}
\renewcommand{\figurename}{Figure}
\renewcommand{\tablename}{Table}

\title[On solvable spherical subgroups]{On solvable spherical subgroups\\of semisimple algebraic groups}

\author{Roman Avdeev}

\thanks{Partially supported by Russian Foundation for Basic Research, grant no. 09-01-00648}

\address{Chair of Higher Algebra, Department of Mechanics and Mathematics,
Moscow State University, 1, Leninskie Gory, Moscow, 119992, Russia}

\email{suselr@yandex.ru}




\begin{abstract}
We develop a structure theory of connected solvable spherical
subgroups in semisimple algebraic groups. Based on this theory, we
obtain an explicit classification of all such subgroups up to
conjugation.
\end{abstract}

\maketitle

\sloppy

\section{Introduction}
\label{introduction}

\subsection{}\label{intro}

Let $G$ be a connected semisimple complex algebraic group. A closed
subgroup $H \subset G$ (resp. a homogeneous space $G/H$) is said to
be \emph{spherical} if one of the following three equivalent
conditions is satisfied:

(1) a Borel subgroup $B \subset G$ has an open orbit in $G/H$;

(2) for every irreducible $G$-variety $X$ containing $G/H$ as an
open orbit the number of $G$-orbits in $X$ is finite;

(3) for every irreducible finite-dimensional $G$-module $V$ and
every character $\chi$ of $H$ the dimension of the subspace $\{v \in
V \mid hv = \chi(h)v \; \forall h \in H\} \subset V$ is at most one.

There are other characterizations of spherical subgroups, but the
three mentioned are most often used while studying these subgroups.

Spherical homogeneous spaces have been intensively studied during
last three decades. However, the problem of classification of these
spaces or, equivalently, the problem of classification of spherical
subgroups in semisimple algebraic groups still remains of
importance. Let us give a short historical reference on this
question. The first considerable result in this direction was
obtained by Kr\"amer in~1979~\cite{Kr}. He classified all reductive
spherical subgroups in simple groups. Then Mikityuk in
1986~\cite{Mi} and, independently, Brion in~1987~\cite{Br}
classified all reductive spherical subgroups in arbitrary semisimple
groups (see also~\cite{Yak} for a more accurate formulation). The
next step towards a classification of spherical homogeneous spaces
was performed by Luna in his preprint~\cite{Lu93} of~1993 where he
considered solvable spherical subgroups in semisimple groups. In
this preprint, under some restrictions, all such subgroups are
described in the following sense: to each subgroup one assigns a set
of combinatorial data that uniquely determines this subgroup, and
then one classifies all sets that may appear in that way. In~2001
Luna created a theory of spherical systems and, using this theory,
described (in the same sense) all spherical subgroups in semisimple
groups of type~$\ms A$~\cite{Lu01}. During the following several
years Luna's approach was successfully applied by Bravi and Pezzini
for several other types of semisimple groups including all classical
groups (for details see~\cite{BP} and references therein). At last,
in~2009 a new approach to the problem was suggested by Cupit-Foutou
who completed the proof of the so-called Luna conjecture and thereby
obtained a description of all spherical subgroups in arbitrary
semisimple groups~\cite{CF}. Thus, by this moment there is a
description in combinatorial terms of all spherical subgroups in
semisimple groups. However this description has the following
disadvantage: it does not provide a simple way of constructing a
spherical subgroup corresponding to a given set of invariants that
uniquely determines this subgroup, even in the case of solvable
spherical subgroups. In other words, the existing description is
\emph{implicit}. In this connection the problem of obtaining an
\emph{explicit} classification of all spherical subgroups in
semisimple groups still remains of interest.

The present paper contains a new approach to classification of
connected solvable spherical subgroups in semisimple algebraic
groups. This approach is completely different from Luna's approach
of~1993~\cite{Lu93} and provides an explicit classification. We note
that in this paper the above-mentioned results of Luna and the
others are not used.

\subsection{}

Throughout the paper the ground field is the field $\mb C$ of
complex numbers. All topological terms relate to the Zarisky
topology. All groups are assumed to be algebraic and their subgroups
closed. The tangent algebras of groups denoted by capital Latin
letters are denoted by the corresponding small German letters.
Weights of tori are identified with their differentials.

Until the end of the paper we fix the following notation:

$G$ is an arbitrary connected semisimple algebraic group;

$B \subset G$ is a fixed Borel subgroup of~$G$;

$T \subset B$ is a fixed maximal torus of~$G$;

$U \subset B$ is the maximal unipotent subgroup of $G$ contained
in~$B$;

$N_G(T)$ is the normalizer of $T$ in~$G$;

$W = N_G(T)/T$ is the Weyl group of $G$ with respect to~$T$;

$\mf X(T)$ is the character lattice (weight lattice) of~$T$;

$Q = \mf X(T)\otimesZ \mb Q$ is the rational vector space generated
by~$\mf X(T)$;

$(\cdot\,, \cdot)$ is a fixed inner product on~$Q$ invariant with
respect to~$W$;

$\De \subset \mf X(T)$ is the root system of $G$ with respect
to~$T$;

$\De_+ \subset \De$ is the subset of positive roots with respect
to~$B$;

$\Pi \subset \De_+$ is the set of simple roots;

$r_\al \in W$ is the simple reflection corresponding to a root~$\al
\in \Pi$;

$\overline w \in N_G(T)$ is a fixed representative of an element~$w
\in W$;

$\mf g_\al \subset \mf g$ is the root subspace corresponding to a
root~$\al \in \De$;

$e_\al \in \mf g_\al$ is a fixed non-zero element.

Let $H \subset B$ be a connected solvable subgroup and $N \subset U$
its unipotent radical. We say that $H$ is \emph{standardly embedded
in} $B$ (with respect to $T$) if the subgroup $S = H \cap T \subset
T$ is a maximal torus in~$H$. Clearly, in this situation we have $H
= S \rightthreetimes N$. It is well known that every connected
solvable subgroup in $G$ is conjugate to a subgroup that is
standardly embedded in~$B$.

\subsection{}

We now discuss the structure of this paper and its main ideas.

In \S\,\ref{main_theorem} we prove a convenient criterion of
sphericity for a connected solvable subgroup in terms of its tangent
algebra (Theorem~\ref{solvable_spherical}). This criterion serves as
a basis of the whole paper. Then, using this criterion, we prove
Theorem~\ref{S_and_Psi}, which may be regarded as a first
approximation to a classification of connected solvable spherical
subgroups. Theorem~\ref{S_and_Psi} claims that a connected solvable
spherical subgroup $H$ standardly embedded in $B$ is uniquely
determined by its maximal torus $S = H \cap T$ and the set $\Psi =
{\{\al \in \De_+ \mid \mf g_\al \not\subset \mf h\} \subset \De_+}$.

In \S\,\ref{active_root_theory} we investigate what kind of set the
set~$\Psi$ may be. For roots in $\Psi$ we introduce the term `active
roots'. Having studied properties of a single active root in
relation to the others we list all positive roots that may be
elements of $\Psi$ depending on the root system $\De$
(Theorem~\ref{active_roots}). As a result of the subsequent
investigation of active roots, to each connected solvable spherical
subgroup~$H$ standardly embedded in~$B$ we assign a set of
combinatorial data $\Upsilon(H) = (S,\mr M, \pi, \sim)$, where $S =
H \cap T$ is a maximal torus in~$H$, $\mr M \subset \Psi$ is the set
of so-called maximal active roots, $\pi \colon \mr M \to \Pi$ is a
map, $\sim$ is an equivalence relation on $\mr M$. Then we determine
a series of conditions that are fulfilled by~$\Upsilon(H)$. The
section is ended by the uniqueness theorem
(Theorem~\ref{uniqueness_theorem}): every connected solvable
spherical subgroup $H$ standardly embedded in $B$ is uniquely
determined by its set of combinatorial data~$\Upsilon(H)$.

In \S\,\ref{section_existence} we prove the existence theorem
(Theorem~\ref{existence_theorem}): for every set of combinatorial
data $(S, \mr M, \pi, \sim)$ satisfying the conditions listed in the
uniqueness theorem, there exists a connected solvable spherical
subgroup $H$ standardly embedded in $B$ with this set of
combinatorial data. The proof of the existence theorem contains an
algorithm that allows one to construct a subgroup $H$ corresponding
to a set $(S, \mr M, \pi, \sim)$.

In \S\,\ref{up_to_conjugacy} we investigate when two connected
solvable spherical subgroups standardly embedded in $B$ are
conjugate in~$G$. For this purpose we introduce the notion of an
elementary transformation. An elementary transformation is a
transformation of the form $H_1 \mapsto H_2$, where $H_1, H_2$ are
connected solvable spherical subgroups standardly embedded in $B$
and $H_2 = \si_\al H_1 \si_\al^{-1}$ for some representative
$\si_\al \in N_G(T)$ of the simple reflection~$r_\al$. The answer to
the question under consideration is given by
Theorem~\ref{elementary_transformations}: two connected solvable
spherical subgroups standardly embedded in $B$ are conjugate in $G$
if and only if there is a sequence of elementary transformations
taking one of these subgroups to the other.
Theorems~\ref{uniqueness_theorem}, \ref{existence_theorem},
and~\ref{elementary_transformations} already give a complete
classification of connected solvable spherical subgroups in
semisimple groups. Next, in the context of the general theory we
consider in more detail an important particular case of connected
solvable spherical subgroups, namely, the case of subgroups having
finite index in their normalizer. Compared with the general case,
the classification of such subgroups is reformulated in a simpler
form.

In \S\,\ref{simplification} we show that every conjugacy class of
connected solvable spherical subgroups contains a subgroup $H$
standardly embedded in $B$ such that the set $\Upsilon(H)$ satisfies
stronger conditions than those appearing in the uniqueness theorem
(Theorem~\ref{simplified}). We call such sets $\Upsilon(H)$
`reduced'. Then we prove that, for every two connected solvable
spherical subgroups standardly embedded in $B$ and conjugate in $G$
such that their sets of combinatorial data are reduced, there is a
sequence of elementary transformations taking one of these subgroups
to the other and such that the set of combinatorial data of every
intermediate subgroup is reduced
(Theorem~\ref{simplified_conjugate}).

At last, \S\,\ref{examples} contains some applications of the theory
developed in this paper. Namely, in this section for all simple
groups $G$ of rank at most $4$ we list, up to conjugation, all
connected solvable spherical subgroups having finite index in their
normalizer. In order to simplify this procedure we essentially use
the results of~\S\,\ref{simplification}.

The main results of this paper were announced at the workshop
`Algeb\-raic groups' held on April 18-24, 2010 in Oberwolfach,
Germany (see~\cite{Avd}).

\subsection{Some notation and conventions.}~

$e$ is the identity element of any group;

$|X|$ is the cardinality of a finite set~$X$;

$\langle A \rangle$ is the linear span in~$Q$ of a subset $A \subset
\mf X(T)$;

$V^*$ is the space of linear functions on a vector space~$V$;

$Z_L(K)$ is the centralizer of a subgroup $K$ in a group~$L$;

$N_L(K)$ is the normalizer of a subgroup $K$ in a group~$L$;

$L^0$ is the connected component of the identity of a group~$L$;

$\mf X(L)$ is the group of characters (in additive notation) of a
group~$L$;

$\rk L$ is the rank of a reductive group $L$, that is, the dimension
of a maximal torus in~$L$;

$\Sigma(\widetilde \Pi)$ is the Dynkin diagram of a
subset~$\widetilde \Pi \subset \Pi$.

For every root $\al = \sum\limits_{\ga \in \Pi}k_\ga \ga \in \De_+$,
we define its \emph{support} $\Supp \al = \{\ga \mid k_\ga
> 0\}$ and \emph{height} $\hgt \al = \sum\limits_{\ga \in \Pi}k_\ga$.
If $\al \in \De_+$, then we put $\De(\al) = \De \cap \langle \Supp
\al \rangle$ and $\De_+(\al) = \De_+ \cap \langle \Supp\al \rangle$.
The set $\De(\al)$ is an indecomposable root system whose set of
simple roots is $\Supp \al$. The set of positive roots of $\De(\al)$
coincides with~$\De_+(\al)$.

Let $L$ be a group and let $L_1, L_2$ be subgroups of it. We write
$L = L_1 \rightthreetimes L_2$ if $L$ is a semidirect product of
$L_1, L_2$, that is, $L = L_1L_2$, $L_1 \cap L_2 = \{e\}$, and $L_2$
is a normal subgroup in~$L$.

By abuse of language, we identify roots in $\Pi$ and the
corresponding nodes of the Dynkin diagram of~$\Pi$.

By saying that two nodes of a Dynkin diagram are joined by an edge,
we mean that the edge may be multiple.

For connected Dynkin diagrams, the numeration of simple roots is the
same as in~\cite{VO}.

\section{Criterion of sphericity and some applications}
\label{main_theorem}

\subsection{}

Suppose that a connected solvable subgroup $H \subset G$ standardly
embedded in $B$ is fixed. Let $S = H \cap T$ and $N = H \cap U$ be a
maximal torus and the unipotent radical of $H$, respectively. We
denote by $\tau \colon \mf X(T) \to \mf X(S)$ the character
restriction map from $T$ to~$S$. Let $\Phi = \tau(\Delta_+) \subset
\mf X(S)$ be the weight system of the natural action of $S$ on $\mf
u$. We have $\mf u = \bigoplus \limits_{\la \in \Phi} \mf u_\la$,
where $\mf u_\la \subset \mf u$ is the weight subspace of weight
$\la$ with respect to~$S$. Let $\mf n = \bigoplus \limits_{\la \in
\Phi} \mf n_\la$ be the decomposition of the space $\mf n$ into a
direct sum of weight subspaces with respect to~$S$. At that, $\mf
n_\la \subset \mf u_\la$ for all $\la \in \Phi$ and some of the
subspaces $\mf n_\la$ may be zero. For every $\la \in \Phi$ we
denote by $c_\la$ the codimension of $\mf n_\la$ in~$\mf u_\la$.

The following theorem provides a convenient criterion of sphericity
for connected solvable subgroups.

\begin{theorem}\label{solvable_spherical}
Let $H \subset G$ be a connected solvable subgroup standardly
embedded in~$B$. Then the following conditions are equivalent:

\textup{(1)} $H$ is spherical in~$G$;

\textup{(2)} $c_\la \le 1$ for every $\la \in \Phi$, and the weights
$\la$ with $c_\la = 1$ are linearly independent in~$\mf X(S)$.
\end{theorem}

\begin{proof}
According to~\cite[Proposition~I.1,~3)]{Br} the sphericity of $H$ is
equivalent to the condition that $S$ has an open orbit in $U/N$
under the action $(s,uN) \mapsto sus^{-1}N$.
By~\cite[Lemma~1.4]{Mon} this condition is equivalent to the
existence of an open orbit under the natural action of $S$ on $\mf
u/\mf n$. It remains to prove that $S$ has an open orbit in $\mf
u/\mf n$ if and only if condition~(2) is fulfilled.

For each $\la \in \Phi$ with $c_\la > 0$ choose a subspace $\mf
p_\la\subset\mf u_\la$ such that $\mf u_\la = \mf n_\la \oplus \mf
p_\la$. Put $\mf p = \bigoplus \limits_{\la \in \Phi :\; c_\la>0}\mf
p_\la$ so that $\mf u = \mf n \oplus \mf p$. Then there is an
$S$-equivariant isomorphism $\mf u/\mf n \simeq \mf p$. Let us show
that condition~(2) is equivalent to the existence of an open
$S$-orbit in $\mf p$. Indeed, assume that condition~(2) is
satisfied. Choose a non-zero element in each subspace $\mf p_\la$
with $c_\la=1$. Then all chosen elements form a basis in~$\mf p$.
Clearly, the open $S$-orbit in $\mf p$ consists of elements such
that all their coordinates with respect to this basis are non-zero.
Now assume that condition~(2) does not hold. Choose a basis in each
subspace~$\mf p_\la$. The union of all these bases is a basis
in~$\mf p$. If $c_\la\ge2$ for some $\la\in\Phi$, then for every two
different basis elements in $\mf p_\lambda$ the ratio of the
corresponding coordinate functions is a non-constant $S$-invariant
rational function on $\mf p$, whence there is no open $S$-orbit
in~$\mf p$. Now assume that $c_\la\le1$ for all $\la\in\Phi$ but
there are elements $\la_1,\ldots,\la_k\in\Phi$ such that
$c_{\la_1}=\ldots=c_{\la_k}=1$ and $p_1\la_1+\ldots+p_k\la_k=0$ for
some non-zero tuple $(p_1,\ldots,p_k)\in\mb Z^k$. Let
$y_1,\ldots,y_k$ be the coordinate functions corresponding to the
basis elements of subspaces $\mf p_{\la_1},\ldots,\mf p_{\la_k}$,
respectively. Then it is easy to see that the non-constant rational
function $y_1^{p_1}\cdot\ldots\cdot y_k^{p_k}$ on $\mf p$ is
$S$-invariant, therefore $\mf p$ contains no open $S$-orbit.
\end{proof}

\subsection{} \label{simple_applications}

In this subsection we deduce several consequences from
Theorem~\ref{solvable_spherical}. These consequences will play a
crucial role in the subsequent exposition.

First of all, we recall the following well-known lemma from linear
algebra.

\begin{lemma}\label{vectors_in_euclid_space}
Suppose that vectors $v_1,\ldots,v_n$ of a finite-dimensional Euclid
space $V$ lie in the same half-space, and the angles between them
are pairwise non-acute. Then these vectors are linearly independent.
\end{lemma}

Let $H \subset G$ be a connected solvable spherical subgroup
standardly embedded in~$B$. We put $S = H \cap T$ and $N = H \cap U$
so that $H = S \rightthreetimes N$. It follows from the sphericity
of $H$ that condition~(2) of Theorem~\ref{solvable_spherical} holds.
We denote all weights $\la \in \Phi$ with $c_\la = 1$ by $\vf_1,
\ldots, \vf_K$. These weights are linearly independent in $\mf
X(S)$, in particular, each of them is non-zero. For every $i = 1,
\ldots, K$ we denote by $\Psi_i$ the set of roots $\al \in \De_+$
such that $\tau(\al) = \vf_i$ and $\mf g_\al \not\subset \mf n$. We
put $\mf u_i = \bigoplus \limits_{\al \in \Psi_i}\mf g_\al$.
Evidently, $\mf u_i \subset \mf u_{\vf_i}$ for all $i = 1, \ldots,
K$. Next, for every $i = 1, \ldots, K$ the subspace $\mf n\cap\mf
u_i \subset \mf u_i$ is the kernel of a linear function $\xi_i \in
\mf u_i^*$, which is determined up to proportionality. Clearly, if
$\al \in \Psi_i$ for some $i \in \{1, \ldots, K\}$, then the
restriction of $\xi_i$ to $\mf g_\al$ is non-zero. We also put $\Psi
= \Psi(H) = \Psi_1 \cup \ldots \cup \Psi_K$. Note that $\Psi = \{\al
\in \De_+ \mid \mf g_\al \not\subset \mf n\}$.

\begin{lemma}\label{three_roots}
Suppose that $\al,\be \in \Psi$ and $\ga = \be - \al \in \De_+$.
Then $\ga \notin \Psi$.
\end{lemma}

\begin{proof}
We have $\tau(\ga) = \tau(\be) - \tau(\al)$. If $\tau(\al) =
\tau(\be)$, then $\tau(\ga) = 0$, which is impossible for $\ga \in
\Psi$. If $\tau(\al) \ne \tau(\be)$, then the weights $\tau(\al)$,
$\tau(\be)$ are linearly independent and therefore both are
different from~$\tau(\ga)$. We have obtained that the weights
$\tau(\al)$, $\tau(\be)$, $\tau(\ga)$ are pairwise different and
linearly dependent, which is also impossible for $\ga \in \Psi$.
\end{proof}

\begin{proposition}\label{crucial}
Suppose that $1 \le i,j \le K$ \textup{(}not necessarily $i \ne
j$\textup{)} and roots $\al \in \Psi_i$, $\be \in \Psi_j$ are
different. Assume that $\ga = \be - \al \in \De_+$. Then $\Psi_i +
\ga \subset \Psi_j$. In particular, $|\Psi_i| \le |\Psi_j|$.
\end{proposition}

\begin{proof}
It follows from Lemma~\ref{three_roots} that $\ga \notin \Psi$ and
$\mf g_\ga \subset \mf n$. Assume that $\al' + \ga \notin \Psi_j$
for some element $\al' \in \Psi_i$. Consider the one-dimensional
subspace $(\mf g_\al \oplus \mf g_{\al'}) \cap \mf n$ and choose a
non-zero element $x = pe_\al + p'e_{\al'}$ in it, where $p, p' \in
\mb C$. Note that $p \ne 0$, $p' \ne 0$, and $[x, e_\ga] \in \mf n$.
Fix $q \ne 0$ such that $[e_\al, e_\ga] = qe_\be$. Then $[x, e_\ga]
= pqe_\be + p'[e_{\al'}, e_\ga]$. If $\al' + \ga\in\De$, then the
conditions $\tau(\al' + \ga) = \vf_j$ and $\al' + \ga \notin \Psi_j$
imply $[e_{\al'}, e_\ga] \in \mf g_{\al'+\ga}\subset \mf n$, whence
$e_\be \in \mf n$. If $\al' + \ga \notin \De$, then $[e_{\al'},
e_\ga]=0$, and again $e_\be \in \mf n$. Hence we have obtained that
$\mf g_\be \subset \mf n$, which contradicts the condition $\be \in
\Psi_j$.
\end{proof}

\begin{corollary}\label{non-sharp_angles}
For each $j = 1, \ldots, K$ the angles between the roots in $\Psi_j$
are pairwise non-acute, and these roots are linearly independent.
\end{corollary}

\begin{proof}
For $|\Psi_j| = 1$ there is nothing to prove. For $|\Psi_j| \ge 2$,
suppose that two different roots $\al, \be \in \Psi_j$ satisfy
$(\al,\be)>0$. Then the vector $\gamma = \be - \al$ is a root.
Without loss of generality we may assume that $\ga \in \De_+$. Then
by Proposition~\ref{crucial} we have $\Psi_j + \gamma \subset
\Psi_j$, which is false. Therefore for any two different roots $\al,
\be \in \Psi_j$ we have $(\al, \be) \le 0$. Now, the linear
independence of all roots in $\Psi_j$ follows from
Lemma~\ref{vectors_in_euclid_space}.
\end{proof}

Proposition~\ref{crucial} enables one to introduce a partial order
on the set $\widetilde \Psi = \{\Psi_1, \ldots, \Psi_K\}$ as
follows. For $i \ne j$ we write $\Psi_i \doubleprec \Psi_j$ if
$\Psi_i + \ga \subset\Psi_j$ for some root $\ga\in\De_+$. We write
$\Psi_i \prec \Psi_j$ if $i = j$ or there is a chain $\Psi_i =
\Psi_{k_1}, \Psi_{k_2}, \ldots, \Psi_{k_{m-1}}, \Psi_{k_m} = \Psi_j$
such that $\Psi_{k_p} \doubleprec \Psi_{k_{p+1}}$ for all $p = 1,
\ldots, m - 1$. In particular, $\Psi_i \prec \Psi_j$ if $\Psi_i
\doubleprec \Psi_j$. Clearly, the relation $\prec$ is transitive.
Further, to each set $\Psi_i$ we assign the number $\rho(\Psi_i) =
\sum \limits_{\al \in \Psi_i} \hgt \al$. Then for $\Psi_i
\doubleprec \Psi_j$ we have $\rho(\Psi_i) < \rho(\Psi_j)$. Hence for
$i \ne j$ the relations $\Psi_i \prec \Psi_j$ and $\Psi_j \prec
\Psi_i$ cannot hold simultaneously. Thus the relation $\prec$ is
indeed a partial order on~$\widetilde \Psi$.

For $i = 1, \ldots, K$ we say that a root $\al \in \Psi_i$ is
\emph{maximal} if the set $\Psi_i$ is maximal in $\widetilde \Psi$
with respect to the partial order~$\prec$.

\begin{lemma} \label{Psi_maximal_elements}
Let $\Psi_{i_1}, \ldots, \Psi_{i_m}$ be all maximal elements of the
partially ordered set $\widetilde \Psi$. Then the angles between the
roots in the set $\Psi_{i_1} \cup \ldots \cup \Psi_{i_m}$
\textup{(}that is, the set of all maximal roots\textup{)} are
pairwise non-acute, and these roots are linearly independent.
\end{lemma}

\begin{proof}
In view of Lemma~\ref{vectors_in_euclid_space} and
Corollary~\ref{non-sharp_angles} it suffices to show that for $p \ne
q$ the angle between any two roots $\al \in \Psi_{i_p}$ and $\be \in
\Psi_{i_q}$ is non-acute. Assume the converse. Then $\ga = \be -
\al$ is a root. Without loss of generality it may be assumed that
$\ga \in \De_+$. By Proposition~\ref{crucial} we get $\Psi_{i_p} +
\ga \subset \Psi_{i_q}$, whence $\Psi_{i_p} \prec \Psi_{i_q}$. The
latter relation contradicts the maximality of the set $\Psi_{i_p}$
in~$\widetilde \Psi$.
\end{proof}

\begin{proposition}\label{linear_functions_1}
Suppose that $1 \le i, j \le K$, $i \ne j$, $\Psi_i \doubleprec
\Psi_j$ and $\Psi_i + \ga \subset\Psi_j$ for a root $\ga \in \De_+$.
Then, up to proportionality, the linear function $\xi_i \in \mf
u_i^*$ is uniquely determined by the linear function~$\xi_j \in \mf
u_j^*$. More precisely, there is a constant $c_{ij} \ne 0$ such that
$\xi_i(x) = c_{ij} \xi_j([x, e_\ga])$ for all $x \in \mf u_i$.
\end{proposition}

\begin{proof}
Taking into account Lemma~\ref{three_roots}, we obtain $\ga \notin
\Psi$, whence $\mf g_\ga \subset \mf n$. From the condition $\Psi_i
+ \ga \subset\Psi_j$ it follows that the linear map $l\colon \mf u_i
\to \mf u_j$, $x \mapsto [x, e_\ga]$, is injective. Consider the
linear function $\xi'_i \in \mf u_i^*$ such that $\xi'_i(x) =
\xi_j(l(x))$ for $x \in \mf u_i$. As $\xi_j(e_\al) \ne 0$ for every
$\al \in \Psi_j$, we have $\xi'_i \ne 0$. Since $l(\mf n \cap \mf
u_i)\subset \mf n \cap \mf u_j$, then $\xi'_i(x) = 0$ for every $x
\in \mf n \cap \mf u_i$. From this it immediately follows that
$\xi_i = c_{ij}\xi'_i$ for some $c_{ij} \ne 0$, that is, $\xi_i(x) =
c_{ij} \xi_j([x, e_\ga])$ for all $x \in \mf u_i$.
\end{proof}

\begin{theorem}\label{S_and_Psi}
Up to conjugation by elements of\, $T$, a connected solvable
spherical subgroup $H \subset G$ standardly embedded in $B$ is
uniquely determined by its maximal torus $S \subset T$ and the set
$\Psi \subset \De_+$.
\end{theorem}

\begin{proof}
The set of weights $\{\vf_1, \ldots, \vf_K\}$ is uniquely determined
as the image of the set $\Psi$ under the map~$\tau$. For every $i =
1, \ldots, K$ the set $\Psi_i$ is uniquely determined as the set
$\{\al \in \Psi \mid \tau(\al) = \vf_i\}$. Further, by
Proposition~\ref{linear_functions_1} from the condition $\Psi_i
\prec \Psi_j$ it follows that, up to proportionality, the linear
function $\xi_i$ is uniquely determined by the linear function
$\xi_j$, therefore, up to proportionality, the whole set of linear
functions $\xi_1, \ldots ,\xi_K$ is uniquely determined by the
linear functions $\xi_j$ corresponding to the maximal elements
$\Psi_j$ of~$\widetilde\Psi$.

Conjugation by an element $t\in T$ takes the algebra $\mf h$ to an
isomorphic one and acts on each space $\mf g_\al$, $\al \in \De_+$,
as the multiplication by~$\al(t)$. By
Lemma~\ref{Psi_maximal_elements} all maximal roots in the set $\Psi$
are linearly independent. Therefore, under an appropriate choice of
$t\in T$, all linear functions $\xi_i$ corresponding to maximal
elements $\Psi_i$ of $\widetilde\Psi$ can be simultaneously reduced
to a prescribed form. For example, we may require each $\xi_i$ to be
the sum of all coordinates in the basis $\{e_\al \mid \al \in
\Psi_i\}$. The latter is possible because $\left. \xi_i \right|_{\mf
g_\al} \ne 0$ for all $\al \in \Psi_i$.
\end{proof}

\section{Active root theory}
\label{active_root_theory}

As we have seen in \S\,\ref{simple_applications} (see
Theorem~\ref{S_and_Psi}), up to conjugation by elements of $T$, a
connected solvable spherical subgroup $H \subset G$ standardly
embedded in $B$ is uniquely determined by its maximal torus $S
\subset T$ and the set $\Psi \subset \De_+$. This section is devoted
to study of roots contained in $\Psi$ (in
\S\,\ref{active_roots_subsection} these roots will be called
`active'), as well as the set $\Psi$ on the whole.

During this section we suppose a connected solvable spherical
subgroup $H = S \rightthreetimes N \subset G$ standardly embedded in
$B$ to be fixed. (Here $S = H \cap T$, $N = H \cap U$.) Also, we
preserve all notation introduced in \S\,\ref{main_theorem}.

\subsection{} \label{active_roots_subsection}

In this subsection we introduce the notion of an active root,
establish basic properties of active roots and find out which
positive roots may be active in dependence on the root system~$\De$.

\begin{dfn}
A root $\al \in \De_+$ is called \emph{active} if $\mf g_\al \not
\subset \mf n$.
\end{dfn}

Evidently, a root $\al$ is active if and only if $\al \in \Psi$.

\begin{lemma}\label{sum_of_two_roots}
Let $\al$ be an active root and suppose that $\al=\be+\ga$, where
$\be, \ga \in \De_+$. Then exactly one of the two roots $\be, \ga$
is active.
\end{lemma}

\begin{proof}
If neither of the roots $\be,\ga$ is active, then $\mf g_\be,\mf
g_\ga\subset\mf n$, whence $\mf g_\al=[\mf g_\be,\mf g_\ga]\subset
\mf n$, which is false. Therefore at least one of the two roots
$\be, \ga$ is active. By Lemma~\ref{three_roots} these two roots
cannot be active simultaneously.
\end{proof}

\begin{dfn}
We say that an active root $\be$ is \emph{subordinate} to an active
root $\al$ if $\al = \be + \ga$ for some $\ga \in \De_+$.
\end{dfn}

\begin{dfn}
An active root $\al$ is called \emph{maximal} if it is not
subordinate to any other active root.
\end{dfn}

We note that the notion of maximality of an active root introduced
in this definition coincides with the notion of maximality
considered in \S\,\ref{simple_applications}. In particular, if $\al$
is a maximal active root, then every active root $\be$ with
$\tau(\al) = \tau(\be)$ is also maximal.

\begin{dfn}
If $\al$ is an active root, then the set consisting of $\al$ and all
roots subordinate to $\al$ is called a \emph{family of active roots
generated by the active root $\al$}. We denote this set by~$F(\al)$.
\end{dfn}

For each root $\al \in \De_+$ let $s(\al)$ denote the number of
representations of $\al$ as a sum of two positive roots. Then by
Lemma~\ref{sum_of_two_roots} for an active root $\al$ the number of
its subordinates equals $s(\al)$, that is, $s(\al) = |F(\al)| - 1$.

\begin{lemma}\label{subordinates_diff_weights}
Let $\al$ be an active root. Then:

\textup{(a)} if $\be \in F(\al) \backslash \{\al\}$, then $\tau(\al)
\ne \tau(\be)$;

\textup{(b)} if $\be, \ga \in F(\al) \backslash \{\al\}$ and $\be
\ne \ga$, then $\tau(\be) \ne \tau(\ga)$.
\end{lemma}

\begin{proof}
(a) Put $\ga = \al - \be \in \De_+$. Assume that $\tau(\al) =
\tau(\be)$. Then $\al, \be \in \Psi_i$ for some $i \in \{1, \ldots,
K\}$. By Proposition~\ref{crucial} we have $\Psi_i + \ga \subset
\Psi_i$, which is false. Hence $\tau(\al) \ne \tau(\be)$.

(b) Suppose that $\al = \be + \be' = \ga + \ga'$, where $\be',\ga'
\in \De_+$ and $\be' \ne \ga'$. Fix $i \in \{1, \ldots, K\}$ such
that $\al \in \Psi_i$. Assume that $\tau(\be) = \tau(\ga)$. Then
$\tau(\be') = \tau(\ga')$. Further, Proposition~\ref{crucial} yields
$\be + \ga',\ga + \be' \in \Psi_i$, whence $\tau(\be + \ga') =
\tau(\ga + \be') = \tau(\al)$. Note that in view of the condition
$\be \ne \ga$ the roots $\al, \be + \ga', \ga + \be'$ are different.
By Corollary~\ref{non-sharp_angles} these three roots are linearly
independent. On the other hand, there is a linear dependence $2\al =
(\be + \ga') + (\ga + \be')$, a contradiction. Thus $\tau(\be) \ne
\tau(\ga)$.
\end{proof}

\begin{corollary} \label{family_lin_ind}
If $\al$ is an active root, then all roots in $F(\al)$ are linearly
independent.
\end{corollary}

\begin{proof}
From Lemma~\ref{subordinates_diff_weights} it follows that all
weights $\tau(\be)$, where $\be \in F(\al)$, are different. By
Theorem~\ref{solvable_spherical} these weights are linearly
independent. Hence, all roots in $F(\al)$ are also linearly
independent.
\end{proof}

\begin{lemma} \label{s_of_alpha}
Suppose that $\al \in \De_+$. Then:

\textup{(a)} if $\De(\al)$ is a root system of type $\ms A$, $\ms
D$, or $\ms E$, then $s(\al) = \hgt \al - 1$;

\textup{(b)} in the general case, $s(\al) \ge |\Supp \al| - 1$;
\end{lemma}

\begin{proof}
Without loss of generality we may assume that $\De = \De(\al)$.

Let us prove~(a). Since the root system $\De$ is of type $\ms A$,
$\ms D$, or $\ms E$, it follows that all roots have the same length,
therefore:

(1) a sum of two roots is a root if and only if the angle between
them equals $2\pi/3$;

(2) a difference of two roots is a root if and only if the angle
between them equals $\pi/3$;

(3) for every $\be \in \De$ and $\be_0 \in \Pi$ the root
$r_{\be_0}(\be)$ equals either of $\be - \be_0$, $\be$, $\be +
\be_0$.

Further we use induction on~$\hgt \al$. For $\hgt \al = 1$ the
assertion is true. Assume that $\hgt \al = k$ and the assertion is
true for all roots $\al' \in \De_+$ with $\hgt \al' < k$. Consider
an arbitrary simple root $\al_0$ such that $\be = \al - \al_0 \in
\De_+$. Then the angle between $\al_0$ and~$\be$ is~$2\pi/3$, whence
$\al = r_{\al_0}(\be)$. We have $\hgt \be = \hgt \al - 1$, therefore
$s(\be) = \hgt \al - 2$ by the induction hypothesis. Suppose that
$\be = \be_1 + \be_2$, where $\be_1, \be_2 \in \De_+$. Note that
neither of the sets $\Supp \be_1$, $\Supp \be_2$ coincides with
$\{\al_i\}$. Indeed, otherwise one of the roots $\be_1, \be_2$ would
coincide with $\al_i$, which is impossible since $\be - \al_i$ is
not a root. Hence $r_{\al_0}(\be_1), r_{\al_0}(\be_2) \in \De_+
\backslash \{\al_0\}$ and $\al = r_{\al_0}(\be_1) +
r_{\al_0}(\be_2)$ is a representation of $\al$ as a sum of two
positive roots. Conversely, if $\al = \al_1 + \al_2$, where $\al_1,
\al_2 \in \De_+ \backslash \{\al_0\}$, then ${\Supp \al_1 \ne
\{\al_0\}}$ and ${\Supp \al_2 \ne \{\al_0\}}$. Hence
$r_{\al_0}(\al_1), r_{\al_0}(\al_2) \in \De_+$ and $\be =
{r_{\al_0}(\al_1) + r_{\al_0}(\al_2)}$ is a representation of $\be$
as a sum of two positive roots. Thus we have established a
one-two-one correspondence between representations of $\be$ as a sum
of two positive roots and representations of $\al$ as a sum of two
positive roots different from~$\al_0$. Taking into account the
representation $\al = \al_0 + \be$, we obtain $s(\al) = s(\be) + 1 =
\hgt \al - 1$.

We now prove~(b). Again we use induction on~$\hgt \al$. For $\hgt
\al = 1$ the assertion is true. Assume that $\hgt \al = k$ and the
assertion is proved for all roots $\al' \in \De_+$ with $\hgt \al' <
k$. In view of Lemma~\ref{vectors_in_euclid_space} there is a simple
root $\al_0$ such that $(\al, \al_0) > 0$. Then $\be = \al - \al_0
\in \De_+$. Put $\ga = r_{\al_0}(\al)$. Since $(\al, \al_0)
> 0$, we have $\hgt \ga < \hgt \al$, therefore the root $\ga$ satisfies
the induction hypothesis. Namely, $s(\ga) \ge |\Supp \al| - 2$ for
$\al_0 \notin \Supp \ga$ and $s(\ga) \ge |\Supp \al| - 1$ for $\al_0
\in \Supp \ga$. In any case, the number of representations of the
form $\ga = \ga_1 + \ga_2$, where $\ga_1, \ga_2 \in \De_+
\backslash\{\al_0\}$, is at least $|\Supp \al| - 2$. For every such
a representation we have $r_{\al_0}(\ga_1), r_{\al_0}(\ga_2) \in
\De_+ \backslash \{\al_0\}$, therefore $\al = r_{\al_0}(\ga_1) +
r_{\al_0}(\ga_2)$ is a representation of $\al$ as a sum of two
positive roots. Taking into account the representation $\al = \be +
\al_0$, we obtain $s(\al) \ge |\Supp \al| - 1$.
\end{proof}

\begin{lemma}\label{number_of_subordinates}
Let $\al$ be an active root. Then:

\textup{(a)} $|F(\al)| = |\Supp \al|$;

\textup{(b)} the weights $\tau(\be)$, where $\be \in \Supp \al$, are
linearly independent;

\textup{(c)} $\langle F(\al) \rangle = \langle \Supp \al \rangle$;

\textup{(d)} if $\be \in \De_+$, $\Supp \be \subset \Supp \al$, and
$\al - \be \notin \De_+$, then the root $\be$ is not active.
\end{lemma}

\begin{proof}
By Lemma~\ref{s_of_alpha}(b) we have $|F(\al)| \ge |\Supp \al|$. In
view of Lemma~\ref{subordinates_diff_weights} the weights
$\tau(\ga)$, where $\ga \in F(\al)$, are different and, by
Theorem~\ref{solvable_spherical}, linearly independent. But these
weights lie in the subspace $\tau(\langle \Supp \al \rangle) \subset
\mf X(S) \otimesZ \mb Q$ of dimension at most $|\Supp \al|$,
therefore $|F(\al)| \le |\Supp \al|$. Hence we get~(a),~(b),
and~(c).

Let us prove~(d). Suppose that $\be \in \De_+$, $\Supp \be \subset
\Supp \al$, and $\al - \be \notin \De_+$. From~(b) it follows that
$\tau(\be) \ne \tau(\ga)$ for all $\ga \in F(\al)$. If $\be$ were an
active root, by Theorem~\ref{solvable_spherical} all weights in the
set $\{\tau(\be)\} \cup \{\tau(\ga) \mid \ga \in F(\al)\}$ would be
linearly independent, which is impossible in view of~(c). Hence the
root $\be$ is not active.
\end{proof}

\begin{corollary}\label{subfamilies}
Let $\al$ be an active root. Then:

\textup{(a)} if $\be \in \Psi$ and $\Supp \be \subset \Supp \al$,
then $\be \in F(\al)$;

\textup{(b)} if $\be\in F(\al)\backslash\{\al\}$, then
$F(\be)\subset F(\al)$;

\textup{(c)} if $\al$ is maximal, then $\Supp \al \backslash \Supp
\be \ne \varnothing$ for every maximal active root $\be \ne \al$.
\end{corollary}

\begin{proof}
In the hypothesis of~(a), by Lemma~\ref{number_of_subordinates}(d)
we get $\al - \be \in \De_+$, whence $\be \in F(\al)$. Obviously,
assertions~(b) and~(c) follow from~(a).
\end{proof}

\begin{corollary}\label{property_of_ordering}
Suppose that $\Psi_i \prec \Psi_j$ for some $i,j \in \{1, \ldots,
K\}$, $i \ne j$. Then $\Psi_i \doubleprec \Psi_j$.
\end{corollary}

\begin{proof}
It suffices to prove that for any $p, q, r$ such that $\Psi_p
\doubleprec \Psi_q$, $\Psi_q \doubleprec \Psi_r$ we have $\Psi_p
\doubleprec \Psi_r$. By definition of the partial order on
$\widetilde \Psi$ there are roots $\ga_{pq},\ga_{qr} \in \De_+$ such
that $\Psi_p + \ga_{pq} \subset \Psi_q$ and $\Psi_q + \ga_{qr}
\subset \Psi_r$. Consider an arbitrary root $\al \in \Psi_p$. Then
$\al + \ga_{pq} \in \Psi$, $\al + \ga_{pq} + \ga_{qr} \in \Psi$,
$\al \in F(\al + \ga_{pq})$, and $\al + \ga_{pq} \in F(\al +
\ga_{pq} + \ga_{qr})$. By Corollary~\ref{subfamilies} we obtain $\al
\in F(\al + \ga_{pq} + \ga_{qr})$. Therefore $\ga_{pq} + \ga_{qr}
\in \De_+$, whence by Proposition~\ref{crucial} we have $\Psi_p +
(\ga_{pq} + \ga_{qr}) \subset \Psi_r$, that is, $\Psi_p \doubleprec
\Psi_r$.
\end{proof}

\begin{proposition}\label{associated_simple_root}
For every active root $\al$ there exists a unique simple root
$\pi(\al) \in \Supp\al$ with the following property: if $\al = \al_1
+ \al_2$ for some roots $\al_1, \al_2 \in \De_+$, then $\al_1$
\textup{(}resp.~$\alpha_2$\textup{)} is active if and only if
$\pi(\al) \notin \Supp \al_1$ \textup{(}resp. $\pi(\al) \notin \Supp
\al_2$\textup{)}.
\end{proposition}

\begin{proof}[Proof \textup{is by induction on $\hgt \al$.}]
If $\hgt \al = 1$, then $\al \in \Pi$ and we may put $\pi(\al) =
\al$. Now suppose that $\hgt \al = k$ and the assertion is proved
for all active roots of height at most~$k - 1$. Assume that the
required root $\pi(\al)$ does not exist. To each simple root $\ga
\in \Supp \al$ we assign an active root $\ga' \in F(\al) \backslash
\{\al\}$ such that $\ga \in \Supp \ga'$ and $\hgt \ga'$ is minimal.
Then $\ga = \pi(\ga')$ in view of the choice of $\ga'$ and the
induction hypothesis. Since the root $\pi(\ga')$ is unique, we
obtain that for different roots $\ga_1, \ga_2 \in \Supp \al$ the
corresponding roots $\ga'_1, \ga'_2 \in F(\al)$ are also different.
Thus $|F(\al)| \ge |\Supp \al| + 1$, which contradicts
Lemma~\ref{number_of_subordinates}(a). Hence there exists a root
with required properties. If there is another such simple root
$\pi'(\al) \ne \pi(\al)$, then the set $F(\al) \backslash \{\al\}$,
which is linearly independent by Corollary~\ref{family_lin_ind} and
consists of $|\Supp\al| - 1$ elements, is contained in the subspace
$\langle (\Supp\al) \backslash \{\pi(\al), \pi'(\al)\} \rangle$ of
dimension $|\Supp\al| - 2$, a contradiction. Thus, the root
$\pi(\al)$ is uniquely determined.
\end{proof}

\begin{corollary} \label{determ_family}
For every active root $\al$ the family $F(\al)$ is uniquely
determined by~$\pi(\al)$.
\end{corollary}

\begin{corollary}\label{one_to_one}
If $\al$ is an active root, then the map $\pi \colon F(\al) \to
\Supp \al$ is a bijection.
\end{corollary}

\begin{proof}
To each simple root $\be \in \Supp \al$ we assign a root $\rho(\be)
\in F(\al)$ of minimal height such that $\be \in \Supp \rho(\be)$.
(If there are several such roots, we choose any of them.) Then by
Proposition~\ref{associated_simple_root} applied to $\rho(\be)$ we
obtain $\be = \pi(\rho(\be))$, whence $\pi$ is surjective. Since
$|F(\al)|=|\Supp \al|$ (see Lemma~\ref{number_of_subordinates}(a)),
it follows that $\pi$ is a bijection.
\end{proof}

\begin{dfn}
If $\al$ is an active root, then the root $\pi(\al) \in \Pi$
appearing in Proposition~\ref{associated_simple_root} is called the
\emph{simple root associated with the active root~$\al$}.
\end{dfn}

\begin{theorem}\label{active_roots}
Suppose that $\al$ is an active root and $\pi(\al)$ is the simple
root associated with~it. Then the pair $(\al,\pi(\al))$ is contained
in Table~\textup{\ref{table_active_roots}}.
\end{theorem}

\begin{table}[h]

\caption{}\label{table_active_roots}

\begin{center}

\begin{tabular}{|c|c|c|c|}

\hline

No. & Type of $\De(\al)$ & $\al$ & $\pi(\al)$ \\

\hline

1 & any of rank $n$ & $\al_1 + \al_2 + \ldots + \al_n$ &
$\al_1, \al_2, \ldots, \al_n$\\

\hline

2 & $\ms B_n$ & $\al_1 + \al_2 + \ldots + \al_{n-1} + 2\al_n$ &
$\al_1, \al_2, \ldots, \al_{n-1}$\\

\hline

3 & $\ms C_n$ & $2\al_1 + 2\al_2 + \ldots + 2\al_{n-1} + \al_n$
& $\al_n$\\

\hline

4 & $\ms F_4$ & $2\al_1 + 2\al_2 + \al_3 + \al_4$ &
$\al_3, \al_4$\\

\hline

5 & $\ms G_2$ & $2\al_1 + \al_2$ & $\al_2$\\

\hline

6 & $\ms G_2$ & $3\al_1 + \al_2$ & $\al_2$\\

\hline
\end{tabular}

\end{center}

\end{table}

This theorem being proved below, we now explain the notation in
Table~\ref{table_active_roots}. In the column `$\al$' the expression
of $\al$ as the sum of simple roots in $\Supp \al$ is given. At
that, $j$th simple root in the diagram $\Sigma(\Supp\al)$ is denoted
by~$\al_j$. In the column `$\pi(\al)$' all possibilities for
$\pi(\al)$ for a given active root $\al$ are listed.

\begin{proof}[Proof of Theorem~\textup{\ref{active_roots}}]
If $\De(\al)$ is a root system of type $\ms A$, $\ms D$, or $\ms E$,
then, by Lemmas~\ref{s_of_alpha} and~\ref{number_of_subordinates},
$\al$ equals the sum of all simple roots in its support.

If $\De(\al)$ is of type $\ms B, \ms C, \ms F, \ms G$, then by
Lemma~\ref{number_of_subordinates}(a) we obtain
$s(\al)=|\Supp(\al)|-1$. Using case-by-case considerations, it is
not hard to find out that this equality holds for exactly two roots
with complete support\footnote{For $\Delta$ indecomposable, a root
$\al \in \De_+$ has \textit{complete support} (with respect to
$\Pi$) if $\Supp \al = \Pi$.} in root systems $\ms B_n, \ms C_n, \ms
F_4$ and exactly three roots with complete support in root
system~$\ms G_2$. All these roots are contained in
Table~\ref{table_active_roots}. For each of rows 2--5 of this table,
regard the root $\be = \sum \limits_{\ga \in \Supp \al}\ga \in
\De_+$. We have $\al - \be \in \De_+$, therefore $\al = \be + (\al -
\be)$ is a representation of $\al$ as a sum of two positive roots.
Hence $\pi(\al) \notin \Supp (\al - \be)$. For the root $\al$ in
row~6 of Table~\ref{table_active_roots} there is the representation
$\al = \al_1 + (2\al_1 + \al_2)$ as a sum of two positive roots,
whence $\pi(\al) \ne \al_1$. Thus, for every root $\al$ in rows~2--6
of Table~\ref{table_active_roots} we have obtained a subset of the
set $\Supp \al$ that does not contain~$\pi(\al)$. In each case, all
remaining possibilities for $\pi(\al)$ are listed in the
column~`$\pi(\al)$'.
\end{proof}

\begin{remark}
From the existence theorem proved in~\S\,\ref{section_existence}
below it follows that all the possibilities listed in
Table~\ref{table_active_roots} are actually realized.
\end{remark}

In order to formulate some consequences of
Theorem~\ref{active_roots}, we need to introduce the following
notion.

\begin{dfn}
Let $\al$ be an active root. A simple root $\al'\in\Supp \al$ is
called \emph{terminal with respect to $\Supp\al$} if in the diagram
$\Sigma(\Supp \al)$ the node $\al'$ is joined by an edge with
exactly one other node.
\end{dfn}

Simple case-by-case considerations of all possibilities in
Table~\ref{table_active_roots} yield the following three statements.

\begin{corollary}\label{supp&active}
If $\al$ is an active root, $|\Supp \al| \ge 2$, and $\al' \in \Supp
\al \cap F(\al)$, then the root $\al'$ is terminal with respect
to~$\Supp \al$.
\end{corollary}

\begin{corollary}\label{extreme_root}
If $\al$ is an active root and a simple root $\al' \in \Supp \al$ is
terminal with respect to $\Supp \al$, then either $\al' = \pi(\al)$
or $\al' \in F(\al)$.
\end{corollary}

\begin{corollary}\label{neighbour}
Let $\al, \al'$ be active roots such that $\al' \in F(\al)$. Suppose
that the simple root $\pi(\al)$ is terminal with respect to $\Supp
\al$ and in the diagram $\Sigma(\Supp \al)$ the node $\pi(\al')$ is
joined by an edge with the node~$\pi(\al)$. Then $\Supp \al =
\{\pi(\al)\} \cup \Supp \al'$.
\end{corollary}

\subsection{}

In this subsection we investigate how the supports of two different
active roots may intersect. The main results of the subsection are
Propositions~\ref{intersection_diff_weights}
and~\ref{intersection_equal_weights}.

\begin{lemma} \label{two_families_diff_weights}
Let $\al, \be$ be different maximal active roots such that
$\tau(\al) \ne \tau(\be)$. Suppose that $\al' \in F(\al)$, $\be' \in
F(\be)$, and $\al' \ne \be'$. Then $\tau(\al') \ne \tau(\be')$.
\end{lemma}

\begin{proof}
If $\al' = \al$ and $\be' = \be$, then there is nothing to prove.
Hence without loss of generality we may assume that $\al'\ne\al$ and
$\al = \al' + \al''$ for some $\al'' \in \De_+$. Assume that
$\tau(\al') = \tau(\be')$. If $\be' = \be$, then by
Proposition~\ref{crucial} we obtain that $\be' + \al'' = \be +
\al''$ is an active root, which contradicts the maximality of~$\be$.
Further we assume that $\be' \ne \be$ and $\be = \be' + \be''$ for
some $\be'' \in \De_+$. Again by Proposition~\ref{crucial} we obtain
that $\al' + \be''$ and $\be' + \al''$ are active roots such that
$\tau(\al' + \be'') = \tau(\be)$ and $\tau(\be' + \al'') =
\tau(\al)$. Hence by Lemma~\ref{Psi_maximal_elements} the angles
between (different) roots $\al, \be' + \al'', \be, \al' + \be''$ are
pairwise non-acute, and these roots are linearly independent. On the
other hand, there is the linear dependence $\al + \be = \hbox{$(\be'
+ \al'')$} + \hbox{$(\al' + \be'')$}$. This contradiction proves
that $\tau(\al') \ne \tau(\be')$.
\end{proof}

\begin{lemma}\label{two_families_diff_weights_2}
Let $\al, \be$ be different maximal active roots such that
$\tau(\al) \ne \tau(\be)$. Then neither of the simple roots
$\pi(\al), \pi(\be)$ lies in the set $\Supp \al \cap \Supp \be$.
\end{lemma}

\begin{proof}
It suffices to show that $\pi(\al) \notin \Supp \al \cap \Supp \be$.
Assume the converse. Put $a = |\Supp \al|$, $b = |\Supp \be|$, $c =
|\Supp \al \cap \Supp \be|$. By Lemma~\ref{one_to_one} the set
$\{\ga \in F(\al) \mid \pi(\ga) \in \Supp \al \backslash \Supp
\be\}$ contains at least $a - c$ roots. Clearly, none of these
roots, nor the root~$\al$, is not contained in the set~$F(\be)$.
Hence there are at least $a - c + 1 + b$ pairwise different roots in
the set $F(\al) \cup F(\be)$. By
Lemmas~\ref{subordinates_diff_weights}
and~\ref{two_families_diff_weights}, the $S$-weights of all roots in
$F(\al)\cup F(\be)$ are different and, by
Theorem~\ref{solvable_spherical}, linearly independent. Hence the
dimension of the space $\langle F(\al) \cup F(\be) \rangle$ is at
least $a + b - c + 1$. On the other hand, this space is contained in
the space $\langle \Supp \al \cup \Supp \be \rangle$ of dimension $a
+ b - c$, a contradiction.
\end{proof}

\begin{corollary} \label{max_assoc_int}
Let $\al, \be$ be different active roots such that $\pi(\al) \in
\Supp \al \cap \Supp \be$. Then $\tau(\al) = \tau(\be)$.
\end{corollary}

Below we give a list of some conditions on a pair of two active
roots $\al, \be$. These conditions will be used later when we
formulate Propositions~\ref{intersection_diff_weights}
and~\ref{intersection_equal_weights}.

$(\mr D0)$ $\Supp \al \cap \Supp \be = \varnothing$;

\begin{wrapfigure}[11]{R}{150\unitlength}

\begin{picture}(150,115)
\put(70,75){\circle{4}} \put(65,83){$\gamma_0$}

\put(70,73){\line(0,-1){18}} \put(70,53){\circle{4}}
\put(76,51){$\gamma_1$} \put(70,51){\line(0,-1){11}}
\put(70,36){\circle*{0.5}} \put(70,32){\circle*{0.5}}
\put(70,28){\circle*{0.5}} \put(70,25){\line(0,-1){11}}
\put(70,12){\circle{4}} \put(76,10){$\gamma_r$}

\put(72,76){\line(2,1){16}} \put(68,76){\line(-2,1){16}}

\put(90,85){\circle{4}} \put(88,74){$\beta_1$}
\put(92,86){\line(2,1){10}} \put(106,93){\circle*{0.5}}
\put(110,95){\circle*{0.5}} \put(114,97){\circle*{0.5}}
\put(118,99){\line(2,1){10}} \put(130,105){\circle{4}}
\put(128,94){$\beta_q$}

\put(50,85){\circle{4}} \put(45,74){$\alpha_1$}
\put(48,86){\line(-2,1){10}} \put(34,93){\circle*{0.5}}
\put(30,95){\circle*{0.5}} \put(26,97){\circle*{0.5}}
\put(22,99){\line(-2,1){10}} \put(10,105){\circle{4}}
\put(05,94){$\alpha_p$}

\end{picture}

{

\makeatletter

\renewcommand{\abovecaptionskip}{0pt}
\renewcommand{\belowcaptionskip}{0pt}

\renewcommand{\@makecaption}[2]{
\vspace{\abovecaptionskip}%
\sbox{\@tempboxa}{#1 #2}%
\global\@minipagefalse \hbox to \hsize {{\scshape \hfil #1 #2\hfil}}
\vspace{\belowcaptionskip}}

\makeatother

\caption{}\label{diagram_difficult}

}

\end{wrapfigure}

$(\mr D1)$ $\Supp \al \cap \Supp \be = \{\de\}$, where $\pi(\al) \ne
\de$, $\pi(\be) \ne \de$, and $\de$ is terminal with respect to both
$\Supp \al$ and $\Supp \be$;

$(\mr E1)$ $\Supp \al \cap \Supp \be = \{\de\}$, where $\de =
\pi(\al) = \pi(\be)$, $\al - \de \in \De_+$, $\be - \de \in \De_+$,
and $\de$ is terminal with respect to both $\Supp \al$ and $\Supp
\be$;

$(\mr D2)$ the diagram $\Sigma(\Supp \al \cup \Supp \be)$ has the
form shown on Figure~\ref{diagram_difficult} (for some $p, q, r\ge
1$), $\al = \al_1 + \ldots + \al_p + \ga_0 + \ga_1 + \ldots +
\ga_r$, $\be = \be_1 + \ldots + \be_q + \ga_0 + \ga_1 + \ldots +
\ga_r$, $\pi(\al) \not\in \Supp \al \cap \Supp \be$, and $\pi(\be)
\notin \Supp \al \cap \Supp \be$;

$(\mr E2)$ the diagram $\Sigma(\Supp\al\cup\Supp \be)$ has the form
shown on Figure~\ref{diagram_difficult} (for some $p,q,r\ge 1$),
$\al = \al_1 + \ldots + \al_p + \ga_0 + \ga_1 + \ldots + \ga_r$,
$\be = \be_1 + \ldots + \be_q + \ga_0 + \ga_1 + \ldots + \ga_r$, and
$\pi(\al) = \pi(\be) \in \Supp \al \cap \Supp \be$.

We note that in view of Corollary~\ref{extreme_root} the root $\de$
appearing in~$(\mr D1)$ is active.

\begin{proposition}\label{intersection_diff_weights}
Let $\al,\be$ be different maximal active roots such that $\tau(\al)
\ne \tau(\be)$. Then one of possibilities $(\mr D0)$, $(\mr D1)$, or
$(\mr D2)$ is realized.
\end{proposition}

\begin{proof}
Without loss of generality we may assume that $\De = \De \cap
\langle \Supp\al \cup \Supp \be \rangle$. Put $I = \Supp \al \cap
\Supp \be$. Assume that possibility $(\mr D0)$ is not realized, that
is, $I \ne \varnothing$. Then in the diagram $\Sigma(\Pi)$ there is
a node $\de \in I$ joined by an edge with a node contained in $\Supp
\al \backslash I$. (The latter set is nonempty in view of the
maximality of $\al$ and Corollary~\ref{subfamilies}(a).) Next we
consider two possible cases.

\textit{Case}~1. The root $\de$ is terminal with respect to $\Supp
\al$ or $\Supp \be$. Then, by Corollary~\ref{extreme_root} and
Lemma~\ref{two_families_diff_weights_2}, $\de$ is an active root. In
view of Corollary~\ref{supp&active} we obtain that $\de$ is terminal
with respect to both $\Supp \al$ and $\Supp \be$, therefore $I =
\{\de\}$ and~$(\mr D1)$ is realized.

\textit{Case}~2. The root $\de$ is terminal with respect to neither
$\Supp \al$ nor $\Supp \be$. Because of the symmetry under the
interchange of $\al$ and $\be$ we may assume that the following
additional condition is satisfied: every node in $I$ joined by an
edge with a node in $(\Supp \al \cup \Supp \be) \backslash I$ is
terminal with respect to neither $\Supp \al$ nor $\Supp \be$. From
this condition it follows that the degree of $\de$ in the diagram
$\Sigma(\Pi)$ is $3$, and the diagram itself has the form shown on
Figure~\ref{diagram_difficult} (for some $p, q, r \ge 1$). Moreover,
$\al = \al_1 + \ldots + \al_p + \ga_0 + \ga_1 + \ldots + \ga_r$,
$\be = \be_1 + \ldots + \be_q + \ga_0 + \ga_1 + \ldots + \ga_r$,
$\de = \ga_0$. By Lemma~\ref{two_families_diff_weights_2} neither of
the roots $\pi(\al), \pi(\be)$ lies in~$I$, therefore the condition
$(\mr D2)$ holds.
\end{proof}

\begin{lemma} \label{coincide_assoc}
Let $\al, \be$ be different active roots such that $\pi(\al) =
\pi(\be)$. Then $\tau(\al) = \tau(\be)$.
\end{lemma}

\begin{proof}
In view of Corollary~\ref{one_to_one} we have $\al \notin F(\be)$
and $\be \notin F(\al)$. Put $\de = \pi(\al) = \pi(\be)$, $a =
|\Supp \al|$, $b = |\Supp \be|$, $c = |\Supp \al \cap \Supp \be|$.
Assume that $\tau(\al) \ne \tau(\be)$. Consider the set $A = F(\al)
\cup \{\be\} \cup \{\ga \in F(\be) \mid \pi(\ga) \in \Supp \be
\backslash \Supp \al\}$. This set contains exactly $a + b - c + 1$
different elements. If $\dim \tau(\langle \Supp \al \cup \Supp \be
\rangle) = a + b - c$, then the $S$-weights of all elements in $A$
are different and therefore linearly independent
(Theorem~\ref{solvable_spherical}). The latter is impossible since
$A \subset \langle \Supp \al \cup \Supp \be \rangle$. Hence, $\dim
\tau(\langle \Supp \al \cup \Supp \be \rangle) \le a + b - c - 1$
and there are at least two pairs of elements in $A$ such that the
$S$-weights inside one pair are the same. Further we consider two
cases.

\emph{Case}~1. There are roots $\al' \in F(\al) \backslash\{\al\}$
and $\be' \in A \cap F(\be) \backslash \{\be\}$ such that $\tau(\al)
= \tau(\be')$ and $\tau(\be) = \tau(\al')$. Fix $i, j \in \{1,
\ldots, K\}$ such that $\al \in \Psi_i$ and $\be \in \Psi_j$. Then
$\Psi_i \doubleprec \Psi_j$ and $\Psi_j \doubleprec \Psi_i$, whence
$\Psi_i = \Psi_j$ and $\tau(\al)= \tau(\be)$, a contradiction.

\emph{Case}~2. There are roots $\al' \in F(\al) \backslash\{\al\}$
and $\be' \in A \cap F(\be) \backslash \{\be\}$ such that
$\tau(\al') = \tau(\be')$. Note that $\al' \ne \be'$. Let $\al'',
\be'' \in \De_+$ be the roots such that $\al = \al' + \al''$ and
$\be = \be' + \be''$. By Proposition~\ref{crucial} we have $\al' +
\be'', \be' + \al'' \in \Psi$, at that, $\tau(\al' + \be'') =
\tau(\be)$, $\tau(\be' + \al'') = \tau(\al)$. Since $\al' \ne \be'$
and $\tau(\al) \ne \tau(\be)$, it follows that the four roots $\al,
\be, \al' + \be'', \be' + \al''$ are pairwise different. By
Corollary~\ref{non-sharp_angles} the angle between the roots $\al' +
\be''$ and $\be$, as well as the angle between the roots $\be' +
\al''$ and $\al$, is non-acute. Further, the angle between $\al$ and
$\be$ is also non-acute since otherwise $\al - \be$ would be a root
and we would have $\al \in F(\be)$ or $\be \in F(\al)$, whence
$\pi(\al) \ne \pi(\be)$ (Corollary~\ref{one_to_one}), which is not
the case. Assume that the angle between either $\al' + \be''$ and
$\al$ or $\be' + \al''$ and $\be$ is acute. Interchanging $\al$ and
$\be$, if necessary, we may assume that the angle between $\al' +
\be''$ and $\al$ is acute. Then $\rho = \al - (\al' + \be'') = \al''
- \be'' \in \De$. Again, in view of the symmetry under the
interchange of $\al$ and~$\be$ we may assume that $\rho \in \De_+$.
Then $\al = (\al' + \be'') + \rho$, whence $\al' + \be'' \in
F(\al)$, $\de = \pi(\al) \notin \Supp (\al' + \be'')$, and $\de
\notin \Supp \be''$. On the other hand, $\be = \be' + \be''$,
therefore $\de = \pi(\be) \in \Supp \be''$ (see
Proposition~\ref{associated_simple_root}). This contradiction shows
that the angle between $\al' + \be''$ and~$\al$, as well as the
angle between $\be' + \al''$ and~$\be$, is non-acute. We now prove
that the angle between $\al' + \be''$ and $\be' + \al''$ is
non-acute. If this is not the case, then $\rho = \al' + \be'' - \be'
- \al'' \in \De$. Again we may assume that $\rho \in \De_+$. By
Proposition~\ref{crucial} we have $\al + \rho = 2(\al' + \be'') -
\be \in \Psi$, at that, $\tau(\al + \rho) = \tau(\be) = \tau(\al' +
\be'')$. The roots $\al + \rho$, $\be$, and $\al' + \be''$ are
pairwise different and linearly dependent, which contradicts
Corollary~\ref{non-sharp_angles}. As a result of the preceding
argument we have obtained that the four roots $\al, \be, \al' +
\be'', \be' + \al''$ are pairwise different, and the angles between
them are pairwise non-acute. Hence by
Lemma~\ref{vectors_in_euclid_space} these roots are linearly
independent. On the other hand, there is the linear dependence $\al
+ \be = (\al' + \be'') + (\be' + \al'') $. This contradiction
completes the proof.
\end{proof}

\begin{corollary}\label{assoc_intersection}
Let $\al, \be$ be different active roots such that $\tau(\al) =
\tau(\be)$ and $\pi(\al) \in \Supp \al \cap \Supp \be$. Then
$\pi(\al) = \pi(\be)$.
\end{corollary}

\begin{proof}
By Corollary~\ref{one_to_one} there is a root $\be' \in F(\be)$ such
that $\pi(\be') = \pi(\al)$. Then in view of
Lemma~\ref{coincide_assoc} we have $\tau(\be') = \tau(\al) =
\tau(\be)$. By Lemma~\ref{subordinates_diff_weights} we obtain $\be'
= \be$.
\end{proof}

\begin{lemma} \label{two_families_equal_weights}
Let $\al,\be$ be different active roots such that $\tau(\al) =
\tau(\be)$. Suppose that $\al' \in F(\al)$, $\be' \in F(\be)$, $\al'
\ne \be'$, and $\al - \al' \ne \be - \be'$. Then $\tau(\al') \ne
\tau(\be')$.
\end{lemma}

\begin{proof}
If $\al' = \al$ or $\be' = \be$, then the assertion follows from
Lemma~\ref{subordinates_diff_weights}. Further we assume that $\al'
\ne \al$ and $\be' \ne \be$. We have $\al = \al' + \al''$, $\be =
\be' + \be''$ for some $\al'', \be'' \in \De_+$, at that, $\al'' \ne
\be''$ by the hypothesis. Assume that $\tau(\al') = \tau(\be')$.
Then by Proposition~\ref{crucial} we obtain that $\al' + \be''$ and
$\be' + \al''$ are active roots such that $\tau(\al) = \tau(\be) =
\tau(\al' + \be'') = \tau(\be' + \al'')$. Moreover, from the
conditions $\al \ne \be$, $\al' \ne \be'$, and $\al'' \ne \be''$ it
follows that any two roots among $\al, \be, \al' + \be'', \be' +
\al''$ are different except for, possibly, roots $\al' + \be''$ and
$\be' + \al''$. In any case by Corollary~\ref{non-sharp_angles} all
(different) roots in the set $\{\al, \be, \al' + \be'', \be' +
\al''\}$, which consists of three or four elements, are linearly
independent. On the other hand, there is the linear dependence $\al
+ \be = \hbox{$(\be' + \al'')$} + \hbox{$(\al' + \be'')$}$, a
contradiction. Hence $\tau(\al') \ne \tau(\be')$.
\end{proof}

\begin{corollary}\label{two_families_equal_weights_crl}
Let $\al,\be$ be different active roots such that $\tau(\al) =
\tau(\be)$. Suppose that $\al' \in F(\al) \backslash \{\al\}$, $\be'
\in F(\be) \backslash \{\be\}$, $\al' \ne \be'$, and $\tau(\al') =
\tau(\be')$. Then there is a root $\ga \in \De_+$ such that $\al =
\al' + \ga$ and $\be = \be' + \ga$.
\end{corollary}

\begin{lemma}\label{intermediate}
Let $\al, \be$ be different active roots such that $\tau(\al) =
\tau(\be)$. Suppose that $\Supp \al \cap \Supp \be = \{\ga\}$, where
$\ga = \pi(\al) = \pi(\be)$. Then $\al - \ga \in F(\al)$, $\be - \ga
\in F(\be)$.
\end{lemma}

\begin{proof}
The hypothesis of the lemma implies that $F(\al)\cap
F(\be)=\varnothing$. Assume that ${\tau(\al') \ne \tau(\be')}$ for
any two roots $\al'\in F(\al)\backslash\{\al\}$, $\be'\in
F(\be)\backslash\{\be\}$. Then in view of
Lemma~\ref{subordinates_diff_weights} the restriction to $S$ of the
roots in $F(\al)\cup F(\be)$ yields exactly $|F(\al)| + |F(\be)|-1$
different weights. By Theorem~\ref{solvable_spherical} these weights
are linearly independent and span a subspace $\Omega\subset \mf
X(S)\otimesZ \mb Q$ of dimension $|F(\al)| + |F(\be)| - 1$, which is
equal to ${|\Supp \al| + |\Supp \be| - 1}$ by
Lemma~\ref{number_of_subordinates}(a). On the other hand, $\Omega$
is contained in the subspace $\Omega'$ spanned by the restrictions
to $S$ of roots in $\Supp \al \cup \Supp \be$. Further, $\Omega'$ is
spanned by $|\Supp \al| + |\Supp \be| - 1$ elements satisfying the
non-trivial relation $\tau(\al)=\tau(\be)$. Hence $\dim \Omega'\le
|\Supp \al| + |\Supp \be| - 2$, a contradiction. Therefore there are
roots $\al'\in F(\al)\backslash\{\al\}$, $\be'\in
F(\be)\backslash\{\be\}$ such that $\tau(\al')=\tau(\be')$. Then by
Corollary~\ref{two_families_equal_weights_crl} there is a root
$\de\in \De_+$ such that $\al=\al'+\de$ and $\be=\be'+\de$. The
latter equalities yield $\Supp \de\subset \Supp \al \cap \Supp \be =
\{\ga\}$. Hence $\de=\ga$, $\al-\ga = \al' \in F(\al)$, $\be - \ga =
\be' \in F(\be)$.
\end{proof}

\begin{proposition}\label{intersection_equal_weights}
Let $\al, \be$ be different active roots such that $\tau(\al) =
\tau(\be)$. Then one of possibilities $(\mr D0)$, $(\mr D1)$, $(\mr
E1)$, $(\mr D2)$, $(\mr E2)$ is realized.
\end{proposition}

\begin{proof}
Without loss of generality we may assume that $\De = \De \cap
\langle \Supp\al \cup \Supp \be \rangle$. Put $I = \Supp \al \cap
\Supp \be$. Assume that possibility $(\mr D0)$ is not realized, that
is, $I \ne \varnothing$. Then there is a node $\de \in I$ joined by
an edge with a node contained in $\Supp \al \backslash I$. (The
latter set is nonempty by Corollary~\ref{subfamilies}(a) and
Lemma~\ref{subordinates_diff_weights}(a).) Further we consider three
cases.

\textit{Case}~1. The root $\de$ is terminal with respect to $\Supp
\al$. Then $I = \{\de\}$. By Corollary~\ref{extreme_root}, $\de$ is
either an active root or the root associated with~$\al$. If $\de$ is
an active root, then, by Corollary~\ref{supp&active}, $\de$ is
terminal with respect to $\Supp \be$ and we have~$(\mr D1)$. If $\de
= \pi(\al)$, then by Corollary~\ref{assoc_intersection} and
Lemma~\ref{intermediate} we obtain $\de = \pi(\be)$, $\al - \de \in
F(\al)$ and $\be - \de \in F(\be)$. We now show that $\de$ is
terminal with respect to $\Supp \be$. Regard the degree $d$ of the
node $\de$ in the diagram $\Sigma(\Pi)$. If $d = 2$, then $\de$ is
automatically terminal. If $d = 3$, then $\De(\be)$ is of type~$\ms
A$ and, by Theorem~\ref{active_roots}, $\be$ equals the sum of all
roots in~$\Supp \be$. It follows that the support of the root $\be -
\de$ is disconnected, which is impossible. Therefore $d = 2$, $\de$
is terminal with respect to $\Supp \be$, and possibility~$(\mr E1)$
is realized.

\textit{Case}~2. The root $\de$ is not terminal with respect to
$\Supp \al$ but is terminal with respect to $\Supp \be$. If $I =
\{\de\}$, then we may interchange $\al$ and $\be$ and consider
\textit{Case}~1. Therefore we assume that $I \ne \{\de\}$. Denote by
$\de'$ the node in the diagram $\Sigma(I)$ joined by an edge
with~$\de$. In view of Corollaries~\ref{supp&active},
\ref{extreme_root}, and~\ref{assoc_intersection} we have $\de =
\pi(\be) = \pi(\al)$. Let $\al' \in F(\al)\backslash\{\al\}$ and
$\be' \in F(\be) \backslash \{\be\}$ be such that $\pi(\al') =
\pi(\be') = \de'$. By Lemma~\ref{coincide_assoc} we obtain
$\tau(\al') = \tau(\be')$. If $\al' \ne \be'$, then in view of
Corollary~\ref{two_families_equal_weights_crl} we have $\al - \al' =
\be - \be' \in \De_+$ and, in particular, $\Supp(\al - \al') \subset
I$. Since $\de \notin \Supp \al'$ and $\de' \in \Supp \al'$, the set
$\Supp(\al - \al')$ contains the node of $\Supp \al \backslash I$
joined by an edge with~$\de$, a contradiction. If $\al' = \be'$,
then in view of Corollary~\ref{neighbour} we get $\Supp \be \subset
I$, which is impossible.

\textit{Case}~3. The root $\de$ is terminal with respect to neither
$\Supp \al$ nor~$\Supp \be$. Arguing just as in \textit{Case}~2 of
the proof of Proposition~\ref{intersection_diff_weights} we obtain
that the diagram $\Sigma(\Pi)$ has the form shown on
Figure~\ref{diagram_difficult} (for some $p, q, r \ge 1$), $\al =
\al_1 + \ldots + \al_p + \ga_0 + \ga_1 + \ldots + \ga_r$, $\be =
\be_1 + \ldots + \be_q + \ga_0 + \ga_1 + \ldots + \ga_r$, $\de =
\ga_0$. In this situation, taking into account
Corollary~\ref{assoc_intersection} we obtain that one of
possibilities~$(\mr D2)$ or $(\mr E2)$ is realized.
\end{proof}

\subsection{}

The main goal of this subsection is to prove the following
proposition.

\begin{proposition} \label{covering_of_support}
Let $\al$ be a maximal active root. Then there exists a simple root
${\widetilde \al \in \Supp \al}$ such that $\widetilde \al \notin
\Supp \be$ for every maximal active root $\be \ne \al$.

In other words, the support of a maximal active root is not covered
by the supports of other maximal active roots.
\end{proposition}

Before we prove this proposition, let us prove several auxiliary
lemmas.

\begin{lemma}\label{eta}
Let $\al, \be$ be different active roots such that $\tau(\al) =
\tau(\be)$ and $\pi(\al) = \pi(\be)$. Then:

\textup{(a)} there is a unique node $\eta(\al, \be) \in \Supp \al
\backslash \Supp \be$ of the diagram $\Sigma(\Pi)$ joined by an edge
with a node in $\Supp \al \cap \Supp \be$;

\textup{(b)} if a root $\al' \in F(\al)$ satisfies $\pi(\al') =
\eta(\al, \be)$, then there is a root $\be' \in F(\be)$ such that
$\tau(\al') = \tau(\be')$.
\end{lemma}

\begin{proof}
It follows from Proposition~\ref{intersection_equal_weights} that
for $\al, \be$ exactly one of possibilities $(\mr E1)$ or $(\mr E2)$
is realized. It is easy to see that in both cases assertion~(a)
holds. To prove~(b), we consider both possibilities separately.

\textit{Case}~1. Possibility $(\mr E1)$ is realized. In view of
Lemma~\ref{intermediate}, $\al$ belongs to either row~1 or row~2 of
Table~\ref{table_active_roots}. Denote by $\de$ the unique simple
root contained in $\Supp \al \cap \Supp \be$. Then the root $\al' =
\al - \de$ is a desired one.

\textit{Case}~2. Possibility $(\mr E2)$ is realized. Then the
diagram $\Sigma(\Supp \al \cup \Supp \be)$ has the form shown on
Figure~\ref{diagram_difficult} (for some $p, q, r \ge 1$), $\al =
\al_1 + \ldots + \al_p + \ga_0 + \ga_1 + \ldots + \ga_r$, $\be =
\be_1 + \ldots + \be_q + \ga_0 + \ga_1 + \ldots + \ga_r$, and
$\pi(\al) = \pi(\be) = \ga_s$, where $0 \le s \le r$. Then,
evidently, $\al' = \al_1 + \ldots + \al_p$, $\be' = \be_1 + \ldots +
\be_q$.
\end{proof}

\begin{lemma}\label{three_active_roots}
Let $\al, \be, \ga$ be pairwise different roots such that $\tau(\al)
= \tau(\be) = \tau(\ga)$ and $\pi(\al) = \pi(\be) = \pi(\ga)$. Then
either $\Supp \al \cap \Supp \be \subset \Supp \al \cap \Supp \ga$
or $\Supp \al \cap \Supp \ga \subset \Supp \al \cap \Supp \be$.
\end{lemma}

\begin{proof}
In view of Proposition~\ref{intersection_equal_weights} for each
pair of roots $\al, \be$ and $\al, \ga$ one of possibilities $(\mr
E1)$ or $(\mr E2)$ is realized. If $(\mr E1)$ is realized for one of
these pairs, then the assertion is true. It remains to observe that
$(\mr E2)$ cannot be realized for both pairs simultaneously.
\end{proof}

\begin{lemma}\label{eta_superstar}
Let $\al$ be a maximal active root. Suppose that $\be \ne \al$ is a
maximal active root such that $\pi(\al) \in \Supp \al \cap \Supp
\be$ and the set $\Supp \al \cap \Supp \be$ is maximal with respect
to inclusion. Then for every maximal active root $\ga \ne \al$ we
have $\eta(\al,\be) \notin \Supp \ga$.
\end{lemma}

\begin{proof}
In view of Corollaries~\ref{max_assoc_int}
and~\ref{assoc_intersection} we have $\tau(\al) = \tau(\be)$ and
$\pi(\al) = \pi(\be)$. Regard the root $\eta = \eta(\al, \be)$. By
definition, we have $\eta \notin \Supp \be$. Let $\al' \in F(\al)
\backslash \{\al\}$ be a root such that $\pi(\al') = \eta$. Then by
Lemma~\ref{eta}(b) there is a root $\be' \in F(\be) \backslash
\{\be\}$ with $\tau(\al') = \tau(\be')$. In view of
Corollary~\ref{two_families_equal_weights_crl} there is a root $\de
\in \De_+$ such that $\de = \al - \al' = \be - \be'$.

Assume that $\eta$ is contained in the support of a maximal active
root $\ga$ different from $\al$ and $\be$. Since $\eta \in \Supp \al
\cap \Supp \ga$ and $\pi(\al) \ne \eta$, we have $\pi(\ga) \ne
\eta$. Regard the root $\ga' \in F(\ga) \backslash \{\ga\}$ such
that $\pi(\ga') = \eta$ and put $\ga'' = \ga - \ga' \in \De_+$. By
Lemma~\ref{coincide_assoc} we have ${\tau(\ga') = \tau(\al')}$. If
$\ga'' = \de$, then $\tau(\ga) = \tau(\al) = \tau(\be)$, whence
$\pi(\ga) \in \Supp \de \subset \Supp \al \cap \Supp \ga$ and, by
Corollary~\ref{assoc_intersection}, $\pi(\ga) = \pi(\al)$. Since
$\eta \in (\Supp \al \cap \Supp \ga) \backslash \Supp \be$, by
Lemma~\ref{three_active_roots} we obtain that $\Supp \al \cap \Supp
\ga \supsetneqq \Supp \al \cap \Supp \be$, a contradiction with the
choice of~$\be$. Hence $\ga'' \ne \de$. Further, by
Proposition~\ref{crucial}, $\al' + \ga''$, $\be' + \ga''$, and $\ga'
+ \de$ are maximal active roots. In view of
Lemma~\ref{Psi_maximal_elements} all different roots in the set
$\{\al, \be, \ga, {\al' + \ga''}, {\be' + \ga''}, \ga' + \de\}$ are
linearly independent. But there is the relation $\al + \be + 2\ga =
(\al' + \ga'') + (\be' + \ga'') + 2(\ga' + \de)$. Since $\ga$
coincides with none of $\al, \be, \be' + \ga'', \ga' + \de$, the
relation is non-trivial. This contradiction proves the lemma.
\end{proof}

\begin{proof}[Proof of Proposition~\textup{\ref{covering_of_support}}]
If $\pi(\al) \notin \Supp \be$ for every maximal active root $\be
\ne \al$, then one may take $\widetilde \al = \pi(\al)$. Otherwise
$\pi(\al) \in \Supp \be$ for some maximal active root $\be \ne \al$.
Without loss of generality one may assume that the set $\Supp \al
\cap \Supp \be$ is maximal with respect to inclusion. Then by
Lemma~\ref{eta_superstar} one may take $\widetilde \al = \eta(\al,
\be)$.
\end{proof}

\subsection{}\label{last_condition}

In this subsection we indicate a condition relating the torus $S$
with the set~$\Psi$. The main result of the subsection is
Proposition~\ref{condition_T}.

We recall (see Corollary~\ref{one_to_one}) that for every active
root $\al$ the map $\pi \colon F(\al) \to \Supp \al$ is bijective.

\begin{lemma}\label{alpha_minus_beta}
Let $\al, \be$ be different maximal active roots. Put $J = \Supp \al
\backslash \Supp \be$. Then:

\textup{(a)} if $\tau(\al) \ne \tau(\be)$, then for every root $\al'
\in F(\al)$ with $\pi(\al') \in J$ and every root $\be' \in F(\be)$
we have $\tau(\al') \ne \tau(\be')$;

\textup{(b)} if $\tau(\al) = \tau(\be)$ and $\pi(\al) \in J$, then
for every root $\al' \in F(\al)\backslash \{\al\}$ with $\pi(\al')
\in J$ and every root $\be' \in F(\be)$ we have $\tau(\al') \ne
\tau(\be')$;

\textup{(c)} if $\tau(\al) = \tau(\be)$ and $\pi(\al) \in \Supp \al
\cap \Supp \be$, then for every root $\al' \in F(\al)$ with
$\pi(\al') \in J \backslash\{\eta(\al, \be)\}$ and every root $\be'
\in F(\be)$ we have $\tau(\al') \ne \tau(\be')$.
\end{lemma}

\begin{proof}
Assertion~(a) is a direct consequence of
Lemma~\ref{two_families_diff_weights}. Let us prove~(b). Let ${\al'
\in F(\al) \backslash \{\al\}}$ and $\be' \in F(\be)$ be such that
$\pi(\al') \in J$ and $\tau(\al') = \tau(\be')$. Then by
Lemma~\ref{two_families_equal_weights} we obtain $\de = \al - \al' =
\be - \be' \in \De_+$. Hence $\pi(\al) \in \Supp \de \subset \Supp
\al \backslash J$, a contradiction. In the hypothesis of~(c), by
Proposition~\ref{intersection_equal_weights} for $\al, \be$ one of
possibilities $(\mr E1)$ or $(\mr E2)$ is realized. In both cases,
as is easy to see, every root $\al' \in F(\al)$ with $\pi(\al') \in
J \backslash \{\eta(\al, \be)\}$ is subordinate to the root $\al''
\in F(\al)$ such that $\pi(\al'') = \eta(\al, \be)$. Assume that
$\tau(\al') = \tau(\be')$ for some root $\be' \in F(\be)$. By
Lemma~\ref{two_families_equal_weights} we obtain that $\de = \al -
\al' = \be - \be' \in \De_+$. Then we have $\eta(\al, \be) \in \Supp
\de$, which is impossible in view of the condition $\Supp \de
\subset \Supp \al \cap \Supp \be$.
\end{proof}

Let us denote by $\mr M = \mr M(H)$ the set of maximal active roots
of~$H$.

\begin{lemma}\label{estimations}
Let $\mr M' \subset \mr M$ be an arbitrary subset. Put $l =
|\bigcup\limits_{\de \in \mr M'}\Supp \de|$, $k =
|\tau(\bigcup\limits_{\de \in \mr M'}F(\de))|$. Then:

\textup{(a)} $\dim \, \langle \mu - \nu \mid \mu, \nu \in \mr M',
\tau(\mu) = \tau(\nu) \rangle = |\mr M'| - |\tau(\mr M')|$;

\textup{(b)} $l = k + |\mr M'| - |\tau(\mr M')|$.
\end{lemma}

\begin{proof}
Let us prove both assertions simultaneously by induction on~$|\mr
M'|$.

For $|\mr M'| = 1$ we have $|\mr M'| = |\tau(\mr M')|$. Obviously,
assertion~(a) is true. Assertion~(b) is also true in view of
Lemmas~\ref{subordinates_diff_weights}
and~\ref{number_of_subordinates}(a).

Now assume that assertions~(a) and~(b) are true for all proper
subsets of $\mr M'$. Let us prove them for $\mr M'$. Suppose that
$\mr M' = \widetilde {\mr M}' \cup \{\al\}$, where $\al \notin
\widetilde {\mr M}'$. Put $J = (\Supp \al) \backslash
(\bigcup\limits_{\de \in \widetilde {\mr M}'}\Supp \de)$. Put also
$\widetilde l = |\bigcup\limits_{\de \in \widetilde {\mr M}'}\Supp
\de|$, $\widetilde k = |\tau(\bigcup\limits_{\de \in \widetilde {\mr
M}'}F(\de))|$. Clearly, $|\mr M'| = |\widetilde {\mr M}'| + 1$ and
$l = \widetilde l + |J|$. Note the following two properties
of~$\al$:

(1) if $\al', \al'' \in F(\al)$ are different roots, then
$\tau(\al') \ne \tau(\al'')$ (see
Lemma~\ref{subordinates_diff_weights});

(2) if for a root $\al' \in F(\al)$ it turns out that $\pi(\al') \in
\Supp \be$ for some root $\be \in \widetilde {\mr M}'$, then there
is a root $\be' \in F(\be)$ such that $\tau(\al') = \tau(\be')$
(this follows from Lemma~\ref{coincide_assoc}).

Further we consider two cases.

\emph{Case}~1. For every root $\de \in \widetilde {\mr M}'$ we have
$\tau(\al) \ne \tau(\de)$. Then $|\tau(\mr M')| = |\tau(\widetilde
{\mr M}')| + 1$ and the subspace $\langle \mu - \nu \mid \mu, \nu
\in \mr M', \tau(\mu) = \tau(\nu) \rangle$ coincides with the
subspace $\langle \mu - \nu \mid {\mu, \nu \in \widetilde {\mr M}'},
\tau(\mu) = \tau(\nu) \rangle$ whose dimension equals $|\widetilde
{\mr M}'| - |\tau(\widetilde {\mr M}')| = |\mr M'| - |\tau(\mr M')|$
by the induction hypothesis. Thus~(a) is proved. In order to
prove~(b), in view of the induction hypothesis it suffices to check
that $|J| = k - \widetilde k$. By Lemma~\ref{alpha_minus_beta}(a)
for every root $\al' \in F(\al)$ with $\pi(\al') \in J$ and every
root $\be \in \bigcup\limits_{\de \in \widetilde {\mr M}'}F(\de)$ we
have $\tau(\al') \ne \tau(\be)$. Hence, taking into account
properties~(1) and~(2), we get $|J| = k - \widetilde k$.

\emph{Case}~2. There is a root $\al_0 \in \widetilde {\mr M}'$ such
that $\tau(\al) = \tau(\al_0)$. Then we have ${|\tau(\mr M')| =
|\tau(\widetilde {\mr M}')|}$. By
Proposition~\ref{covering_of_support} there is a simple root $\rho
\in \Supp \al$ such that $\rho \in J$, whence $\al - \al_0$ does not
lie in the subspace $\langle \mu - \nu \mid \mu,\nu \in \widetilde
{\mr M}', \tau(\mu) = \tau(\nu) \rangle$. It is easy to see that the
subspace $\langle \mu - \nu \mid \mu,\nu \in \mr M', \tau(\mu) =
\tau(\nu)\rangle$ coincides with the subspace $\langle \mu - \nu
\mid \mu, \nu \in \widetilde {\mr M}', \tau(\mu) = \tau(\nu) \rangle
\oplus \langle \al - \al_0 \rangle$ whose dimension equals
$|\widetilde {\mr M}'| - |\tau(\widetilde {\mr M}')| + 1 = |\mr M'|
- |\tau(\mr M')|$ in view of the induction hypothesis. Assertion~(a)
is proved. In order to prove~(b), in view of the induction
hypothesis it suffices to check that $|J| = k - \widetilde k + 1$.
We consider two subcases.

\emph{Subcase}~2.1. $\pi(\al) \in J$. By
Lemma~\ref{alpha_minus_beta}(a,b) for every root $\al' \in F(\al)$
with ${\pi(\al') \in J \backslash \{\pi(\al)\}}$ and every root $\be
\in \bigcup\limits_{\de \in \widetilde {\mr M}'}F(\de)$ we have
$\tau(\al') \ne \tau(\be)$. Hence in view of properties~(1) and~(2)
we get $|J| = k - \widetilde k + 1$.

\emph{Subcase}~2.2. $\pi(\al) \notin J$. In this situation there is
a maximal active root $\be \ne \al$ such that $\pi(\al) \in \Supp
\be$. Without loss of generality we may assume that the set $\Supp
\al \cap \Supp \be$ is maximal with respect to inclusion. Then by
Lemma~\ref{eta_superstar} we have $\eta(\al,\be) \in J$. Let $\al'
\in F(\al)$ be the root such that $\pi(\al') = \eta(\al,\be)$. From
Lemma~\ref{eta}(b) it follows that there is a root $\be' \in F(\be)$
with $\tau(\al') = \tau(\be')$. Assume that for some root $\al'' \in
F(\al)$ with $\pi(\al'') \in J \backslash \{\eta(\al, \be)\}$ there
are roots $\ga \in \widetilde {\mr M}'$ and $\ga' \in F(\ga)$ such
that $\tau(\al'') = \tau(\ga')$. Clearly, $\al'' \ne \ga'$. Put
$\eta' = \pi(\al'')$. Applying
Lemmas~\ref{two_families_diff_weights}
and~\ref{two_families_equal_weights} we obtain that $\tau(\al) =
\tau(\ga)$ and $\al - \al'' = \ga - \ga' \in \De_+$, whence
$\pi(\al) \in \Supp \al \cap \Supp \ga$. Then by
Lemma~\ref{alpha_minus_beta}(c) we get $\eta' = \eta(\al, \ga)$.
Hence in the diagram $\Sigma(\Pi)$ the node $\eta'$ is joined by an
edge with some node of the set $\Supp \al \cap \Supp \ga$. Further,
by Corollary~\ref{assoc_intersection} we have $\pi(\al) = \pi(\ga)$.
In view of the choice of $\be$ and Lemma~\ref{three_active_roots}
there is the inclusion $\Supp \al \cap \Supp \ga \subset \Supp \al
\cap \Supp \be$. Hence we obtain that in the diagram $\Sigma(\Pi)$
the node $\eta' \in \Supp \al \backslash \Supp \be$ is joined by an
edge with some node of the set $\Supp \al \cap \Supp \be$. Then by
Lemma~\ref{eta}(a) we have $\eta' = \eta (\al, \be)$, which is not
the case. Thus for every root $\al'' \in F(\al)$ with $\pi(\al') \in
J \backslash \{\eta(\al, \be)\}$ and every root $\ga' \in
\bigcup\limits_{\de \in \widetilde {\mr M}'}F(\de)$ we have
$\tau(\al'') \ne \tau(\ga')$. Hence in view of properties~(1)
and~(2) we obtain $|J| = k - \widetilde k + 1$.

Assertion~(b) is proved.
\end{proof}

\begin{proposition}\label{condition_T}
The kernel of the map $\tau \colon \langle \bigcup\limits_{\de \in
\mr M}\Supp \de \rangle \to \mf X(S) \otimesZ \mb Q$ coincides with
the subspace $\langle \mu - \nu \mid \mu, \nu \in \mr M, \tau(\mu) =
\tau(\nu) \rangle$.
\end{proposition}

\begin{proof}
Put $R = \langle \bigcup\limits_{\de \in \mr M}\Supp \de \rangle
\subset Q$. By Theorem~\ref{solvable_spherical} the dimension of
$\tau(R)$ is at least~$K$. Further, in view of the inclusion
$\langle \mu - \nu \mid \mu,\nu \in \mr M, \tau(\mu) =
\tau(\nu)\rangle \subset \Ker \left.\tau\right|_R$
Lemma~\ref{estimations}(a) yields $\dim \Ker \left.\tau\right|_R \ge
|\mr M| - |\tau(\mr M)|$. Applying Lemma~\ref{estimations}(b) we
obtain $\dim \Ker \left.\tau\right|_R = |\mr M| - |\tau(\mr M)|$,
which implies the required result.
\end{proof}

\subsection{}

In this subsection we sum up the results obtained in this section
and prove the uniqueness theorem (see
Theorem~\ref{uniqueness_theorem}).

We recall that in \S\,\ref{last_condition} we introduced the
notation $\mr M = \mr M(H)$ for the set of maximal active roots
of~$H$. We now introduce a relation $\sim$ on $\mr M$ as follows.
For any two roots $\al, \be \in \mr M$ we write $\al \sim \be$ if
and only if $\tau(\al) = \tau(\be)$. Evidently, this relation is an
equivalence relation.

To each connected solvable spherical subgroup $H \subset G$
standardly embedded in $B$ we assign the set of combinatorial data
$\Upsilon(H) = (S, \mr M, \pi, \sim)$. We also put $\Upsilon_0(H) =
(\mr M, \pi, \sim)$. (In both of the sets $\Upsilon(H)$ and
$\Upsilon_0(H)$, $\pi$ is considered as a map from $\mr M$
to~$\Pi$.)

\begin{theorem}[Uniqueness]\label{uniqueness_theorem} Let
$H \subset G$ be a connected solvable spherical subgroup standardly
embedded in~$B$. Then, up to conjugation by elements of $T$, $H$ is
uniquely determined by its set $\Upsilon(H) = (S, \mr M, \pi,
\sim)$, and this set satisfies the following conditions:

$(\mr A)$ $\pi(\al) \in \Supp \al$ for every $\al \in \mr M$, and
the pair $(\al, \pi(\al))$ is contained in
Table~\textup{\ref{table_active_roots}};

$(\mr D)$ if $\al, \be \in \mr M$ and $\al \nsim \be$, then for
$\al, \be$ one of possibilities $(\mr D0)$, $(\mr D1)$, $(\mr D2)$
is realized;

$(\mr E)$ if $\al, \be \in \mr M$ and $\al \sim \be$, then for $\al,
\be$ one of possibilities $(\mr D0)$, $(\mr D1)$, $(\mr E1)$, $(\mr
D2)$, $(\mr E2)$ is realized;

$(\mr C)$ if $\al \in \mr M$, then $\Supp \al \not\subset
\bigcup\limits_{\de \in \mathrm M \backslash \{\al\}}\Supp \de$;

$(\mr T)$ $\left.\Ker \tau \right|_R = \langle \mu - \nu \mid \mu,
\nu \in \mr M, \mu \sim \nu \rangle$, where $R = \langle
\bigcup\limits_{\de \in \mr M} \Supp \de \rangle$.
\end{theorem}

\begin{proof}
In view of Corollary~\ref{determ_family} the set $\Psi$ is uniquely
determined by the pair $(\mr M, \pi)$. Then in view of
Theorem~\ref{S_and_Psi}, up to conjugation by elements of~$T$, $H$
is uniquely determined by the triple $(S, \mr M, \pi)$.

Condition $(\mr A)$ follows from the definition of $\pi(\al)$ and
Theorem~\ref{active_roots}. Conditions~$(\mr D)$ and~$(\mr E)$
follow from Propositions~\ref{intersection_diff_weights}
and~\ref{intersection_equal_weights}, respectively. Condition~$(\mr
C)$ is established in Proposition~\ref{covering_of_support}. At
last, condition~$(\mr T)$ is proved in
Proposition~\ref{condition_T}.
\end{proof}

\begin{remark}
The set of combinatorial data $(S, \mr M, \pi, \sim)$ is redundant
in the sense that the relation $\sim$ is uniquely determined by $S$
and $\mr M$. However, the advantage of this set is that, as we shall
see in~\S\,\ref{section_existence}, the unipotent radical $N$ of $H$
can be constructed using only the subset $(\mr M, \pi, \sim)$ with
no need of~$S$ (see Remark~\ref{remark_unipotent_radical}).
\end{remark}

\begin{remark}
If two connected solvable spherical subgroups $H_1, H_2 \subset G$
standardly embedded in $B$ are conjugate in~$G$, then, generally
speaking, the sets of combinatorial data $(S, \mr M, \pi, \sim)$
corresponding to them are different. Therefore, generally speaking,
the set $(S, \mr M, \pi, \sim)$ is not an invariant of conjugacy
classes of connected solvable spherical subgroups. We shall come
back to this question in~\S\,\ref{up_to_conjugacy}.
\end{remark}

\section{Existence theorem}\label{section_existence}

In this section we show that, given a set of combinatorial data
indicated in Theorem~\ref{uniqueness_theorem}, one can construct a
connected solvable spherical subgroup in $G$ standardly embedded
in~$B$ with this set of combinatorial data. Namely, we prove the
following theorem.

\begin{theorem}[Existence] \label{existence_theorem}
Suppose that a subtorus $S\subset T$, a subset $\mr M \subset
\De_+$, a~map $\pi \colon \mr M \to \Pi$, and an equivalence
relation $\sim$ on $\mr M$ satisfy conditions $(\mr A)$, $(\mr D)$,
$(\mr E)$, $(\mr C)$, and $(\mr T)$. Then there exists a connected
solvable spherical subgroup $H \subset G$ standardly embedded in~$B$
such that $\Upsilon(H) = (S, \mr M, \pi, \sim)$.
\end{theorem}

In~\S\,\ref{existence_first} we collect some facts that will be
needed in the proof of this theorem. The proof itself is carried out
in~\S\S\,\ref{existence_proof_Psi}-\ref{existence_proof_end}.

\subsection{}\label{existence_first}

Let a pair $(\al, \al_0)$, where $\al \in \De_+$, $\al_0 \in \Supp
\al$, be such that $\al$ is contained in the column `$\al$' of
Table~\ref{table_active_roots} and $\al_0$ is contained in the same
row in column `$\pi(\al)$' of this table. Put
$$
F(\al)=\{\al\} \cup \{\al'\in \De_+ \mid \al - \al' \in \De_+, \al_0
\notin \Supp \al'\}.
$$
Then using simple case-by-case considerations one can establish the
following properties:

\emph{\textup{(1)} if $\be \in F(\al)$, then $\be$ is contained in
Table~\textup{\ref{table_active_roots}};}

\emph{\textup{(2)} if $\be \in F(\al)$ and $\be = \be_1 + \be_2$ for
some roots $\be_1, \be_2 \in \De_+$, then exactly one of the two
roots $\be_1, \be_2$ lies in~$F(\al)$;}

\emph{\textup{(3)} for every $\be \in F(\al)$ we have $|\{\be\} \cup
\{\be' \in F(\al) \mid \be - \be' \in \De_+\}| = |\Supp \be|$; in
particular, $|F(\al)| = |\Supp \al|$;}

\emph{\textup{(4)} all roots in $F(\al)$ are linearly independent
\textup{(}which in view of condition~\textup{(3)} is equivalent to
$\langle F(\al) \rangle = \langle \Supp \al \rangle$\textup{)}.}

\subsection{} \label{existence_proof_Psi}

We proceed to the proof of Theorem~\ref{existence_theorem}. Suppose
that a set of combinatorial data $(S, \mr M, \pi, \sim)$, where
$S\subset T$ is a subtorus, $\mr M \subset \De_+$ is a subset, $\pi
\colon \mr M \to \Pi$ is a map, and $\sim$ is an equivalence
relation on~$\mr M$, satisfies conditions $(\mr A)$, $(\mr D)$,
$(\mr E)$, $(\mr C)$, and $(\mr T)$.

For each pair $(\al, \pi(\al))$, where $\al \in \mr M$, we construct
the set $F(\al)$ as indicated in~\S\,\ref{existence_first} and put
$\Psi = \bigcup\limits_{\al\in \mr M}F(\al)$.

In this subsection we derive basic properties of the set $\Psi$ that
are necessary for the proof of Theorem~\ref{existence_theorem}.

\begin{lemma}\label{beta_subset_alpha}
Let roots $\al \in \mr M$ and $\be \in \Psi$ be such that $\Supp \be
\subset \Supp \al$. Then $\be \in F(\al)$.
\end{lemma}

\begin{proof}
Regard a root $\widetilde \be \in \mr M$ such that $\be \in
F(\widetilde \be)$. If $\widetilde \be = \al$, then there is nothing
to prove, therefore we assume that $\widetilde \be \ne \al$. In view
of conditions~$(\mr D)$ and $(\mr E)$ for the roots $\al, \widetilde
\be$ one of possibilities $(\mr D1)$, $(\mr E1)$, $(\mr D2)$,
or~$(\mr E2)$ is realized. A direct check in each case shows that
the assertion is true.
\end{proof}

\begin{lemma} \label{sum_of_two_roots_2}
Suppose that $\al \in \Psi$ and $\al = \al_1 + \al_2$ for some roots
$\al_1, \al_2 \in \De_+$. Then exactly one of the two roots $\al_1,
\al_2$ lies in~$\Psi$.
\end{lemma}

\begin{proof}
Let a root $\widetilde \al \in \mr M$ (which, possibly, coincides
with $\al$) be such that $\al \in F(\widetilde \al)$. Then by
property~(2) exactly one of the two roots $\al_1, \al_2$ lies
in~$F(\widetilde \al)$. We may assume that $\al_1 \in F(\widetilde
\al)$. If $\al_2 \in \Psi$, then by Lemma~\ref{beta_subset_alpha} we
obtain $\al_2 \in F(\widetilde \al)$, which is not the case.
\end{proof}

We now define the set $F(\al)$ for an arbitrary root $\al \in \Psi$:
$F(\al) = \{\al\} \cup \{\al' \in \Psi \mid \al - \al' \in \De_+\}$.
For roots $\al \in \mr M$ this definition coincides with the one
given above.

\begin{corollary}\label{family_Psi}
Let $\al \in \Psi$ be an arbitrary root. Then:

\textup{(a)} $|F(\al)| = |\Supp \al|$;

\textup{(b)} all roots in $F(\al)$ are linearly independent
\textup{(}which in view of~\textup{(}a\textup{)} is equivalent to
$\langle F(\al) \rangle = \langle\Supp \al \rangle$\textup{)}.
\end{corollary}

\begin{proof}
Assertion~(a) follows from condition~(3) and
Lemma~\ref{sum_of_two_roots_2}, assertion~(b) follows from
condition~(4).
\end{proof}

\begin{proposition}\label{associated_simple_root_2}
\textup{(a)} Suppose that $\al \in \Psi$. Then there exists a unique
simple root ${\pi(\al) \in \Supp\al}$ with the following property:
if $\al = \al_1 + \al_2$ for some roots $\al_1, \al_2 \in \De_+$,
then $\alpha_1$ \textup{(}resp.~$\alpha_2$\textup{)} belongs to
$\Psi$ if and only if $\pi(\alpha) \notin \Supp \al_1$
\textup{(}resp. $\pi(\alpha) \notin \Supp \alpha_2$\textup{)}.

\textup{(b)} For every $\al \in \Psi$ the map $\pi: F(\al) \to \Supp
\al$ is a bijection.
\end{proposition}

\begin{proof}
Assertion~(a) (resp.~(b)) is proved by the same argument that is
used in the proof of Proposition~\ref{associated_simple_root} (resp.
Corollary~\ref{one_to_one}), with replacing reference to
Lemma~\ref{number_of_subordinates}(a) (resp.
Corollary~\ref{family_lin_ind}) by reference to
Corollary~\ref{family_Psi}(a) (resp.~\ref{family_Psi}(b)).
\end{proof}

Thus we have defined the map $\pi$ on the whole set~$\Psi$. We note
that on the set $\mr M$ this map coincides with the given map $\pi
\colon \mr M \to \Pi$.

The next step is to extend the equivalence relation $\sim$ to the
whole set~$\Psi$. Suppose that $\al',\be'\in \Psi \backslash \mr M$.
We write $\al' \sim \be'$ if and only if there are roots $\al, \be
\in \mr M$ and $\de \in \De_+$ such that $\al' \in F(\al)$, $\be'
\in F(\be)$, $\al' + \de = \al$ and $\be' + \de = \be$. Below we
shall prove (see Proposition~\ref{equivalence_relation}) that this
relation is an equivalence relation on the set~$\Psi\backslash \mr
M$. We now note two simple properties of this relation.

\begin{lemma}\label{equivalence}
Suppose that $\al',\be'\in \Psi \backslash \mr M$, $\al'\sim\be'$
and roots $\al,\be\in \mr M$, $\de\in\De_+$ are such that $\al'\in
F(\al)$, $\be'\in F(\be)$, $\al' + \de = \al$, and $\be' + \de =
\be$. Then $\al \sim \be$.
\end{lemma}

\begin{proof}
Since both roots $\pi(\al),\pi(\be)$ are contained in $\Supp \de$,
they are contained in $\Supp\al \cap \Supp \be$, which is impossible
for $\al \nsim \be$ in view of condition~$(\mr D)$.
\end{proof}

\begin{lemma}\label{two_equiv_roots_one_family}
Suppose that $\al',\be'\in \Psi\backslash \mr M$, $\al'\ne \be'$,
and $\al'\sim\be'$. Then there is exactly one root $\al\in \mr M$
with~$\al'\in F(\al)$.
\end{lemma}

\begin{proof}
Choose roots $\al, \be \in \mr M$, $\de \in \De_+$ such that $\al'
\in F(\al)$, $\be'\in F(\be)$, $\al' + \de = \al$, and $\be' + \de =
\be$. Then by the hypothesis we have $\al\ne\be$. Assume that there
is a root $\widetilde \al \in \mr M$ such that $\widetilde \al \ne
\al$ and $\al'\in F(\widetilde \al)$. Then we have $\Supp \al'
\subset \Supp \widetilde \al$, $\Supp \de \subset \Supp \be$, whence
$\Supp \al \subset \Supp \widetilde \al \cup \Supp \be$, which
contradicts condition~$(\mr C)$.
\end{proof}

\begin{proposition}\label{equivalence_relation}
The relation $\sim$ is an equivalence relation on~$\Psi\backslash
\mr M$.
\end{proposition}

\begin{proof}
Reflexivity and symmetry of $\sim$ are obvious, therefore it
suffices to prove transitivity. Let $\al',\be',\ga' \in \Psi
\backslash \mr M$ be different roots such that $\al' \sim \be'$,
$\al' \sim \ga'$. Let us prove that $\be' \sim \ga'$. We have $\al =
\al' + \de_1$, $\be = \be' + \de_1$, $\widetilde\al = \al' + \de_2$,
$\ga = \ga' + \de_2$ for some roots $\al, \be, \widetilde \al, \ga
\in \mr M$ and $\de_1, \de_2 \in \De_+$. By
Lemma~\ref{two_equiv_roots_one_family} we obtain $\widetilde \al =
\al$, whence $\de_1 = \de_2$ and $\be'\sim \ga'$.
\end{proof}

\begin{corollary}\label{more_than_one_element}
Let $A \subset \Psi \backslash \mr M$ be an equivalence class
containing more than one element. Then:

\textup{(a)} for every root $\al' \in A$ there is a unique root $\al
\in \mr M$ such that $\al' \in F(\al)$;

\textup{(b)} the root $\de = \al - \al'$ is the same for all roots
$\al' \in A$;

\textup{(c)} $A + \de \subset \mr M$ and all roots in $A + \de$ are
pairwise equivalent.
\end{corollary}

\begin{proof}
Assertion~(a) follows from Lemma~\ref{two_equiv_roots_one_family},
assertion~(b) from the proof of
Proposition~\ref{equivalence_relation}, assertion~(c) from~(b) and
Lemma~\ref{equivalence}.
\end{proof}

Thus we have an equivalence relation on each of the sets $\mr M$,
$\Psi\backslash \mr M$. We extend it to the whole set $\Psi$ putting
$\al \nsim \be$ for $\al\in \mr M$, $\be\in \Psi\backslash \mr M$ or
$\al\in \Psi\backslash \mr M$, $\be \in \mr M$. Let $\Psi_1, \Psi_2,
\ldots, \Psi_K$ be all equivalence classes of the set $\Psi$ with
respect to relation~$\sim$.

\begin{proposition}\label{psi+de}
\textup{(a)} Suppose that $i, j \in \{1, \ldots, K\}$, $i \ne j$,
roots $\al' \in \Psi_i$ and $\de \in \De_+$ are such that $\al' +
\de \in \Psi_j$. Then $\Psi_i + \de \subset \Psi_j$.

\textup{(b)} Suppose that $i \in \{1, \ldots, K\}$ and roots $\al',
\al'' \in \Psi_i$ are different. Then $\al'' - \al' \notin \De$.

\textup{(c)} Suppose that $i, j \in\{1, \ldots, K\}$, $i \ne j$, and
$|\Psi_i| \ge 2$. Then there is at most one root $\de \in \De_+$
such that $\Psi_i + \de \subset \Psi_j$.
\end{proposition}

\begin{proof}
(a) If $|\Psi_i|=1$ then there is nothing to prove. If
$\Psi_j\subset \mr M$, then the assertion follows from
Corollary~\ref{more_than_one_element}(c). Further we assume that
$|\Psi_i|\ge2$ and $\Psi_j \not\subset \mr M$. Put $\al'' = \al' +
\de$, $\al'' \in \Psi_j$. Let $\al$ be the unique root in $\mr M$
with $\al'\in F(\al)$ (see
Corollary~\ref{more_than_one_element}(a)). Let $\widetilde \al \in
\mr M$ be an arbitrary root such that $\al'' \in F(\widetilde \al)$.
Then $\Supp \al' \subset \Supp \widetilde \al$, whence by
Lemma~\ref{beta_subset_alpha} we obtain $\al' \in F(\widetilde \al)$
and, in view of Corollary~\ref{more_than_one_element}(a),
$\widetilde \al = \al$. Put $\de' = \al - \al' \in \De_+$, $\de'' =
\al - \al'' \in \De_+$. Now, let us take an arbitrary root $\be'\in
\Psi_i$ and show that $\be'' = \be' + \de \in \Psi_j$. Denote by
$\be$ the unique root in $\mr M$ with $\be'\in F(\be)$ (see
Corollary~\ref{more_than_one_element}(a)). From
Corollary~\ref{more_than_one_element}(b,c) it follows that $\be =
\be' + \de'$ and $\al \sim \be$. Further, we have $\de' = \de +
\de''$, whence $|\Supp \de'|\ge 2$. Moreover, $\Supp \de' \subset
\Supp \al \cap \Supp \be$, at that, $\pi(\al),\pi(\be) \in \Supp
\de'$. Therefore for the roots $\al,\be$ possibility $(\mr E2)$ is
realized. Hence the diagram $\Sigma(\Supp \al \cup \Supp \be)$ has
the form shown on Figure~\ref{diagram_difficult} (for some $p, q, r
\ge 1$), $\al = \al_1 + \ldots + \al_p + \ga_0 + \ga_1 + \ldots +
\ga_r$, $\be = \be_1 + \ldots + \be_q + \ga_0 + \ga_1 + \ldots +
\ga_r$, $\pi(\al) = \pi(\be) = \ga_s$ for some $s \in \{0, 1,
\ldots, r\}$. At that, $\de' = \ga_t + \ga_{t+1} + \ldots + \ga_r$,
where $0 \le t \le s$, and $\de'' = \ga_u + \ga_{u+1} + \ldots +
\ga_r$, where $t < u \le s$. It is easy to see that $\be'' = \be -
\de''$ is a root lying in $F(\be)$. From conditions $\be'' + \de'' =
\be$, $\al'' + \de'' = \al$, and $\al \sim \be$ it follows that
$\al'' \sim \be''$ and $\be''\in \Psi_j$, which completes the proof
of~(a).

(b) Assume that $\de_0 = \al'' - \al' \in \De$. Without loss of
generality we may also assume that $\de_0 \in \De_+$. By
Corollary~\ref{more_than_one_element} there are unique roots
$\widetilde \al', \widetilde \al'' \in \mr M$, $\de \in \De_+$ such
that $\al' \in F(\widetilde \al')$, $\al'' \in F(\widetilde \al'')$,
$\al' + \de = \widetilde \al'$, $\al'' + \de = \widetilde \al''$. In
view of Lemma~\ref{beta_subset_alpha} we have $\al' \in F(\widetilde
\al'')$, which contradicts Lemma~\ref{more_than_one_element}(a).

(c) If $\Psi_j \subset \mr M$, then the assertion is true in view of
Corollary~\ref{more_than_one_element}. Further we assume that
$\Psi_j \not\subset \mr M$. Let $\de$ be a root such that $\Psi_i +
\de \subset \Psi_j$. Let us show that $\de$ is uniquely determined
by $i$ and~$j$. By Corollary~\ref{more_than_one_element} there are
uniquely determined indices $k, l \in\{1, \ldots, K\}$ and roots
$\de_i, \de_j \in \De_+$ such that $\Psi_k, \Psi_l \subset \mr M$,
$\Psi_i + \de_i \subset \Psi_k$, and $\Psi_j + \de_j \subset
\Psi_l$. In view of~(a) for every $\al \in \Psi_i$ we have $\al +
\de \in \Psi_j$, $\al + \de_i \in \Psi_k \subset \mr M$, $\al + \de
+ \de_j \in \Psi_l \subset \mr M$, whence by
Lemma~\ref{beta_subset_alpha} we obtain $\al \in F(\al + \de +
\de_j)$. In view of Corollary~\ref{more_than_one_element}(a) we have
$\al + \de_i = \al + \de + \de_j$. Therefore $\de = \de_i - \de_j$
and $\de$ is uniquely determined by~$i, j$.
\end{proof}

\subsection{} \label{existence_proof_u_r}

In this subsection we construct an algebra~$\mf n$, which is going
to be the tangent algebra of the unipotent radical of the desired
solvable spherical subgroup.

We put $\mf u_i = \bigoplus \limits_{\al\in \Psi_i} \mf g_\al$ for
$i = 1, \ldots, K$ and $\mf u_0 = \bigoplus \limits_{\al \notin \Psi
} \mf g_\al$ so that $\mf u = \mf u_0 \oplus \bigoplus\limits_{i =
1}^K \mf u_i$. To each subspace $\mf u_i$, $i = 1, \ldots, K$, we
assign a linear function $\xi_i \colon \mf u_i \to \mb C$ as
follows. First, let $i$ be such that $\Psi_i \subset \mr M$. Then we
may take $\xi_i$ to be an arbitrary linear function such that its
restriction to each root subspace $\mf g_\al$, $\al \in \Psi_i$, is
non-zero. Further, for all $i$ with $\Psi_i \not\subset \mr M$ and
$|\Psi_i|=1$ we take $\xi_i$ to be any non-zero linear function on
the (one-dimensional) space~$\mf u_i$. At last, if $i$ satisfies
$\Psi_i \not\subset \mr M$ and $|\Psi_i| \ge 2$, then we act as
follows. By Corollary~\ref{more_than_one_element} there are a unique
$j \in \{1,\ldots, K\}$ and a unique root $\de \in \De_+$ such that
$\Psi_j \subset \mr M$ and $\Psi_i + \de \subset \Psi_j$. For every
$x \in \mf u_i$ we put $\xi_i(x) = \xi_j([x, e_\de])$. Then $\xi_i$
is a linear function on $\mf u_i$, and its restriction to $\mf
g_\al$ is non-zero for every~$\al \in \Psi_i$.

\begin{lemma}\label{linear_functions_2}
Suppose that $\Psi_i + \de \subset \Psi_j$, where indices $i, j \in
\{1, \ldots, K\}$ are different, and $\de \in \De_+$. Then there is
an element $c_{ij} \in \mb C^\times$ such that $\xi_i(x) = c_{ij}
\xi_j([x, e_\de])$ for all~$x \in \mf u_i$.
\end{lemma}

\begin{proof}
If $|\Psi_i| = 1$ or $\Psi_j\subset \mr M$, then the assertion
follows from the definition of~$\xi_i$. Further we assume that
$|\Psi_i| \ge 2$ and $\Psi_j \not \subset \mr M$. From the proof of
Proposition~\ref{psi+de}(c) it follows that, uniquely determined,
there are an index $k\in \{1,\ldots, K\}$ and roots
$\de_i,\de_j\in\De_+$ such that $\Psi_k\subset \mr M$, $\Psi_i +
\de_i \subset \Psi_k$, $\Psi_j + \de_j \subset \Psi_k$, and $\de +
\de_j = \de_i$. Suppose that $x\in \mf u_i$. Then by definition we
have $\xi_i (x) = \xi_k ([x, e_{\de_i}])$, $\xi_j([x, e_\de]) =
\xi_k([[x, e_\de], e_{\de_j}])$. Applying the Jacobi identity we get
$\xi_j([x, e_\de]) = \xi_k([x, [e_\de, e_{\de_j}]]) + \xi_k([[x,
e_{\de_j}], e_\de])$. Since $[e_\de, e_{\de_j}] = ce_{\de_i}$ for
some $c\in \mb C^\times$, we have $\xi_j([x, e_\de]) = c\xi_k([x,
e_{\de_i}]) + \xi_k([[x, e_{\de_j}], e_\de])$. To complete the
proof, it is sufficient to check that $[x, e_{\de_j}] = 0$. To do
this, it is sufficient to prove that $\al + \de_j \notin \De_+$ for
every~$\al \in \Psi_i$. Assume that $\al + \de_j \in \De_+$ for some
$\al \in \Psi_i$. Then for the root $\al + \de_j + \de \in \Psi_k$
we have the representation $\al + \de_j + \de = (\al + \de_j) + \de$
as the sum of two positive roots. Since $\de \notin \Psi$, we have
$\al + \de_j \in \Psi$. Besides, $\al + \de \in \Psi$. Hence $\Supp
(\al + \de_j + \de) = \Supp (\al + \de) \cup \Supp (\al + \de_j)$, a
contradiction with conditions $\pi(\al + \de_j + \de) \notin \Supp
(\al + \de)$ and $\pi(\al + \de_j + \de) \notin \Supp (\al +
\de_j)$, which hold in view of
Proposition~\ref{associated_simple_root_2}(a).
\end{proof}

For every $i = 1, \ldots, K$ we put $\mf n_i = \{x \in \mf u_i \mid
\xi_i(x)=0\}$. Evidently, $\mf n_i = 0$ for $|\Psi_i|=1$. We now
consider the subspace $\mf n =\mf u_0 \oplus
\bigoplus\limits_{i=1}^K\mf n_i \subset \mf u$.

\begin{proposition}\label{unipotent_radical}
The subspace $\mf n$ is a subalgebra of~$\mf u$.
\end{proposition}

\begin{proof}
Recall that for every root $\al \in \Psi$ and every representation
$\al = \be + \ga$, where $\be, \ga \in \De_+$, exactly one of the
two roots $\be,\ga$ lies in~$\Psi$ (see
Lemma~\ref{sum_of_two_roots_2}). In view of this fact the proof
reduces to verifying the condition $[\mf n_i, \mf g_\de] \subset \mf
n$ for all $i = 1, \ldots, K$ and $\de \notin \Psi$. Let us do that.
If $\al + \de \notin \Psi$ for all $\al \in \Psi_i$, then the
inclusion $[\mf n_i, \mf g_\de] \subset \mf n$ holds automatically.
If $\al + \de \in \Psi_j$ for some $\al \in \Psi_i$ and some $j \in
\{1, \ldots, K\}$ (at that, $i \ne j$ in view of
Proposition~\ref{psi+de}(b)), then by Proposition~\ref{psi+de}(a) we
obtain $\Psi_i + \de \subset \Psi_j$. From the definition of the
subspaces $\mf n_i, \mf n_j$ and Lemma~\ref{linear_functions_2} we
have $[\mf n_i, \mf g_\de]\subset \mf n_j \subset \mf n$.
\end{proof}

\begin{lemma}\label{normalizes}
The torus $S$ normalizes the subalgebra~$\mf n$.
\end{lemma}

\begin{proof}
In view of condition~$(\mr T)$, for any two roots $\al,\be\in \mr M$
with $\al \sim \be$ we have ${\tau(\al) = \tau(\be)}$. Next, suppose
that $\al',\be' \in \Psi \backslash \mr M$ and $\al' \sim \be'$.
Then from the definition of the equivalence relation $\sim$ it
follows that there are roots $\al,\be \in \mr M$ and $\de \in \De_+$
such that $\al = \al' + \de$ and $\be = \be' + \de$. Then
$\tau(\al') = \tau(\al) - \tau(\de) = \tau(\be) - \tau(\de) =
\tau(\be')$. Hence for all $i = 1, \ldots, K$ the subspace $\mf u_i$
is $S$-invariant, and so is~$\mf n_i$. This proves the lemma.
\end{proof}

\subsection{}\label{existence_proof_end}

This subsection is the final stage of the proof of
Theorem~\ref{existence_theorem}. We construct the subgroup $H
\subset G$ and prove that it is spherical in~$G$.

We denote by $N$ the unipotent subgroup in $G$ with tangent
algebra~$\mf n$. We put $H = SN$. From Lemma~\ref{normalizes} it
follows that $H$ is a subgroup in $G$, $H = S \rightthreetimes N$
and $H$ is standardly embedded in~$B$. For $i = 1,\ldots, K$ we put
$\vf_i = \tau(\al) \in \mf X(S)$, where $\al \in \Psi_i$ is an
arbitrary root. From the proof of Lemma~\ref{normalizes} it follows
that the weight $\vf_i$ is well defined.

\begin{proposition}\label{H_is_spherical}
The subgroup $H$ is a connected solvable spherical subgroup in $G$
standardly embedded in $B$. At that, $\Upsilon(H) = (S, \mr M, \pi,
\sim)$.
\end{proposition}

Before we prove this proposition, let us prove several auxiliary
lemmas.

We recall that for every root $\al \in \Psi$ the map $\pi \colon
F(\al) \to \Supp \al$ is a bijection (see
Proposition~\ref{associated_simple_root_2}(b)).

\begin{lemma}\label{eta_2}
Let $\al, \be \in \mr M$ be different roots such that $\al \sim \be$
and $\pi(\al) = \pi(\be)$. Then:

\textup{(a)} there is a unique node $\eta(\al, \be) \in \Supp \al
\backslash \Supp \be$ of the diagram $\Sigma(\Pi)$ joined by an edge
with a node in $\Supp \al \cap \Supp \be$;

\textup{(b)} if a root $\al' \in F(\al)$ is such that $\pi(\al') =
\eta(\al, \be)$, then there is a root $\be' \in F(\be)$ with $\al'
\sim \be'$.
\end{lemma}

\begin{proof}[Proof \textup{repeats that of Lemma~\ref{eta}}]
\end{proof}

\begin{lemma}\label{three_active_roots_2}
Let $\al, \be, \ga \in \mr M$ be pairwise different roots such that
$\al \sim \be = \tau(\ga)$ and $\pi(\al) = \pi(\be) = \pi(\ga)$.
Then either $\Supp \al \cap \Supp \be \subset \Supp \al \cap \Supp
\ga$ or $\Supp \al \cap \Supp \ga \subset \Supp \al \cap \Supp \be$.
\end{lemma}

\begin{proof}
This is proved by the same argument as
Lemma~\ref{three_active_roots}, with replacing the reference to
Proposition~\ref{intersection_equal_weights} by the reference to
condition~$(\mr E)$.
\end{proof}

\begin{lemma} \label{weights_simple_roots_2}
Suppose that $\al, \be \in \mr M$, $\al \ne \be$, and $I = \Supp \al
\cap \Supp \be \ne \varnothing$. Let $\de \in I$ be an arbitrary
root and let $\al' \in F(\al)$, $\be' \in F(\be)$ be such that
$\pi(\al') = \pi(\be') = \de$. Then:

\textup{(a)} if either of $(\mr D1)$ or $(\mr D2)$ is realized for
$\al$, $\be$, then $\al' = \be'$;

\textup{(b)} if either of $(\mr E1)$ or $(\mr E2)$ is realized for
$\al$, $\be$, then $\al' \sim \be'$.

In either case, $\al' \sim \be'$.
\end{lemma}

\begin{proof}
In view of conditions~$(\mr D)$ and $(\mr E)$ for the roots $\al$
and $\be$ one of possibilities $(\mr D1)$, $(\mr D2)$, $(\mr E1)$,
$(\mr E2)$ is realized. Assertion~(a) is obtained by a direct check.
If one of possibilities $(\mr E1)$ or $(\mr E2)$ is realized, then
$\al \sim \be$. In case of $(\mr E1)$ we have $I = \{\de\}$. Then
$\al' = \al$, $\be' = \be$, and $\al' \sim \be'$. At last, in case
of $(\mr E2)$ we have either $\al' = \be'$ or $\al - \al' = \be -
\be'$. In both cases, $\al' \sim \be'$.
\end{proof}

If $\Psi' \subset \Psi$ is an arbitrary nonempty subset, then the
restriction of the equivalence relation $\sim$ from $\Psi$ to
$\Psi'$ is well defined. Therefore we may consider equivalence
classes in~$\Psi'$.

\begin{lemma}\label{estimations_2}
Let $\mr M' \subset \mr M$ be an arbitrary subset. Put $l =
|\bigcup\limits_{\de \in \mr M'}\Supp \de|$, let $k$ be the number
of equivalence classes in the set $\bigcup \limits_{\de \in \mr
M'}F(\de)$, and let $m$ be the number of equivalence classes in the
set~$\mr M'$. Then:

\textup{(a)} $\dim \, \langle \mu - \nu \mid \mu, \nu \in \mr M',
\mu \sim \nu \rangle = |\mr M'| - m$;

\textup{(b)} $l \ge k + |\mr M'| - m$.
\end{lemma}

\begin{proof}
Let us prove both assertions (a), (b) simultaneously by induction
on~$|\mr M'|$.

For $|\mr M'| = 1$ we have $|\mr M'| = m$ and $\langle \mu - \nu
\mid \mu, \nu \in \mr M', \mu \sim \nu \rangle = \{0\}$, therefore
assertion~(a) is true. Assertion~(b) is also true by property~(3) of
roots in~$\mr M$ (see~\S\,\ref{existence_first}).

Assume that assertions~(a) and~(b) are proved for all proper subsets
of the set~$\mr M'$. Let us prove them for~$\mr M'$. Suppose that
$\mr M' = \widetilde {\mr M}' \cup \{\al\}$, where $\al \notin
\widetilde {\mr M}'$. Put $J = (\Supp \al) \backslash
(\bigcup\limits_{\de \in \widetilde {\mr M}'}\Supp \de)$. Put also
$\widetilde l = |\bigcup\limits_{\de \in \widetilde {\mr M}'}\Supp
\de|$, let $\widetilde k$ be the number of equivalence classes in
the set $\bigcup\limits_{\de \in \widetilde {\mr M}'}F(\de)$, and
let $\widetilde m$ be the number of equivalence classes in the
set~$\widetilde {\mr M}'$. Clearly, $|\mr M'| = |\widetilde {\mr
M}'| + 1$ and $l = \widetilde l + |J|$.

Further we consider two cases.

\emph{Case}~1. For every root $\de \in \widetilde {\mr M}'$ we
have~$\al \nsim \de$. Then $m = \widetilde m + 1$ and the subspace
$\langle \mu - \nu \mid \mu, \nu \in \mr M', \mu \sim \nu \rangle$
coincides with the subspace $\langle \mu - \nu \mid \mu, \nu \in
\widetilde {\mr M}', \mu \sim \nu \rangle$ whose dimension equals
$|\widetilde {\mr M}'| - \widetilde m = |\mr M'| - m$ by the
induction hypothesis. Thus~(a) is proved. In order to prove~(b), in
view of the induction hypothesis it suffices to check that $|J| \ge
k - \widetilde k$. This is the case by
Lemma~\ref{weights_simple_roots_2}.

\emph{Case}~2. There is a root $\al_0 \in \widetilde {\mr M}'$ such
that $\al \sim \al_0$. Then we have $m = \widetilde m$. In view of
condition~$(\mr C)$ there is a simple root $\rho \in \Supp \al$ with
$\rho \in J$, therefore $\al - \al_0$ does not lie in the subspace
$\langle \mu - \nu \mid \mu,\nu \in \widetilde {\mr M}', \mu \sim
\nu \rangle$. It is easy to see that the subspace $\langle \mu - \nu
\mid \mu,\nu \in \mr M', \mu \sim \nu \rangle$ coincides with the
subspace $\langle \mu - \nu \mid \mu, \nu \in \widetilde {\mr M}',
\mu \sim \nu \rangle \oplus \langle \al - \al_0 \rangle$ whose
dimension equals $|\widetilde {\mr M}'| - \widetilde m + 1 = |\mr
M'| - m$ in view of the induction hypothesis. Assertion~(a) is
proved. In order to prove~(b), in view of the induction hypothesis
it suffices to check that $|J| \ge k - \widetilde k + 1$. We
consider two subcases.

\emph{Subcase}~2.1. $\pi(\al) \in J$. The required inequality holds
in view of Lemma~\ref{weights_simple_roots_2} and the condition~$\al
\sim \al_0$.

\emph{Subcase}~2.2. $\pi(\al) \notin J$. In this situation there is
a root $\be \in \widetilde {\mr M}'$ such that ${\pi(\al) \in \Supp
\be}$. Without loss of generality we may assume that the set $\Supp
\al \cap \Supp \be$ is maximal with respect to inclusion. Let
$\eta(\al, \be)$ be the root in Lemma~\ref{eta_2}(a). Regard the
root $\al' \in F(\al)$ with $\pi(\al') = \eta(\al, \be)$ and the
root $\be' \in F(\be)$ with~$\be' \sim \al'$ ($\be'$ exists by
Lemma~\ref{eta_2}(b)). Let us prove that $\eta(\al, \be) \in J$.
Assume the converse. Then there are roots $\ga \in \widetilde {\mr
M}'$ and $\ga' \in F(\ga)$ such that $\pi(\ga') = \pi(\al')$. If
$\al' = \ga'$, then $\Supp \al' \subset \Supp \ga$, ${\Supp (\al -
\al')} = \Supp (\be - \be') \subset \Supp \be$, whence $\Supp \al
\subset \Supp \be \cup \Supp \ga$, which contradicts condition~$(\mr
C)$. Therefore $\al' \ne \ga'$ and, by
Lemma~\ref{weights_simple_roots_2}(a,b), for the roots $\al, \ga$
one of possibilities $(\mr E1)$ or $(\mr E2)$ is realized. In
particular, $\al \sim \ga$ and $\pi(\al) = \pi(\ga)$, whence in view
of Lemma~\ref{three_active_roots_2} we get $\Supp \al \cup \Supp \ga
\supsetneqq \Supp \al \cup \Supp \be$, a contradiction with the
choice of~$\be$. Thus we have $\eta(\al, \be) \in J$. Then the
inequality $|J| \ge k - \widetilde k + 1$ holds in view of
Lemma~\ref{weights_simple_roots_2} and the condition~$\al' \sim
\be'$.

Assertion~(b) is proved.
\end{proof}

\begin{proof}[Proof of Proposition~\textup{\ref{H_is_spherical}}]
Only the sphericity of~$H$ needs to be proved since all the other
conditions are fulfilled by construction. Regard the subspace $R =
\langle \bigcup\limits_{\de \in \mr M}\Supp \de \rangle \subset Q$
and denote by $l$ its dimension. Let $m$ be the number of
equivalence classes in~$\mr M$. In view of condition~(4)
(see~\S\,\ref{existence_first}) the image of $R$ under the map
$\tau$ is spanned by the weights $\vf_1, \ldots, \vf_K$. By
Lemma~\ref{estimations_2}(a) the dimension of this image equals $l -
(|\mr M| - m)$. Hence $K \ge l - (|\mr M| - m)$. On the other hand,
$K \le l - (|\mr M| - m)$ by Lemma~\ref{estimations_2}(b).
Therefore, $K = l - (m - n)$ and all weights $\vf_1,\ldots,\vf_K$
are linearly independent. Moreover, by construction for every $i =
1, \ldots, K$ the codimension of the subspace $\mf n_i$ in the space
$\mf u_i$ equals~$1$. Thus, condition~(2) of
Theorem~\ref{solvable_spherical} is satisfied, hence $H$ is
spherical in~$G$.
\end{proof}

The proof of Theorem~\ref{existence_theorem} is completed.

\begin{remark} \label{remark_unipotent_radical}
As we see from the proof of Theorem~\ref{existence_theorem}, up to
conjugation by elements of~$T$, the unipotent radical $N$ of a
connected solvable spherical subgroup $H$ standardly embedded in $B$
is uniquely recovered from the set $\Upsilon_0(H) = (\mr M, \pi,
\sim)$ satisfying conditions $(\mr A)$, $(\mr D)$, $(\mr E)$, $(\mr
C)$.
\end{remark}

\begin{remark}\label{at_least_one_solvable}
For every set $(\mr M, \pi, \sim)$ satisfying conditions $(\mr A)$,
$(\mr D)$, $(\mr E)$, $(\mr C)$ there is at least one connected
solvable spherical subgroup $H \subset G$ standardly embedded in $B$
such that $\Upsilon_0(H) = (\mr M, \pi, \sim)$. Namely, we may
choose $S$ to be the connected component of the identity of the
subgroup in $T$ defined by vanishing of all characters of the form
$\al - \be$, where $\al, \be \in \mr M$ and $\al \sim \be$.
\end{remark}

\section{Classification of connected solvable spherical subgroups\\ up to
conjugation} \label{up_to_conjugacy}

Theorems~\ref{uniqueness_theorem} and~\ref{existence_theorem}
provide a classification of connected solvable spherical subgroups
of $G$ standardly embedded in~$B$, up to conjugation by elements
of~$T$. The aim of this section is to find out when two connected
solvable spherical subgroups in $G$ standardly embedded in $B$ are
conjugate in $G$ and to reveal a relation between the corresponding
sets of combinatorial data.

\subsection{}

The main result of this subsection is
Proposition~\ref{conjugate_spherical}.

Let $H_1, H_2 \subset G$ be two connected solvable subgroups
standardly embedded in~$B$. For $i = 1,2$ let $N_i$ be the unipotent
radical of~$H_i$ and $S_i \subset T$ its maximal torus so that $H_i
= S_i \rightthreetimes N_i$. We put $Z = Z_G(S_1)$. Being the
centralizer of a torus in~$G$, the group $Z$ is reductive and
connected, and its tangent algebra $\mf z$ has the form $\mf z = \mf
t \oplus \bigoplus \limits_{\al \in \De: \tau(\al) = 0} \mf g_\al$,
where $\tau: \mf X(T) \to \mf X(S_1)$ is the restriction of
characters.

\begin{lemma} \label{conjugate_subgroups}
If $H_2 = gH_1g^{-1}$ for some $g\in G$, then $g \in N_2 \cdot
N_G(T) \cdot Z$. In particular, $H_2 = g_0 H_1 g_0^{-1}$ for some
$g_0 \in N_G(T) \cdot Z$.
\end{lemma}

\begin{proof}
Evidently, $N_2 = gN_1g^{-1}$. Next, there is an element $u \in N_2$
such that ${S_2 = g_0S_1g_0^{-1}}$ for $g_0 = ug$. In view of the
Bruhat decomposition of $G$ we have $g_0 = u_1 \si u_2$, where
${u_1, u_2 \in U}$, $\si \in N_G(T)$. Regard an arbitrary element
$s_1 \in S_1$ and put $s_2 = g_0s_1g_0^{-1} \in S_2$. Then $u_1^{-1}
s_2 u_1 \si = \si u_2 s_1 u_2^{-1}$, which may be rewritten as $s_2
v_1 \si = \si s_1 v_2$, where $v_1 = {s_2^{-1} u_1^{-1} s_2 u_1 \in
U}$ and $v_2 = s_1^{-1} u_2 s_1 u_2^{-1} \in U$. Hence $\si v_2
\si^{-1} = (\si s_1^{-1} \si^{-1}) s_2 v_1 \in B$. Since $v_2$ is a
unipotent element, then $\si v_2 \si^{-1} \in U$. Therefore $\si v_2
\si^{-1} v_1^{-1} = (\si s_1^{-1} \si^{-1}) s_2 \in {U \cap T =
\{e\}}$, whence $s_2 = \si s_1 \si^{-1}$ and $v_2 = \si v_1
\si^{-1}$. Thus, $s_1 = \si^{-1} g_0 s_1 g_0^{-1} \si$ for every
element $s_1 \in S_1$, hence $\si^{-1} g_0 \in Z$ and $g_0 \in
N_G(T) \cdot Z$.
\end{proof}

\begin{proposition} \label{conjugate_spherical}
If both of the subgroups $H_1, H_2$ are spherical in $G$ and $H_2 =
g H_1 g^{-1}$ for some $g \in G$, then $g \in N_2 \cdot N_G(T) \cdot
N_1$. In particular, $H_2 = \si H_1 \si^{-1}$ for some $\si \in
N_G(T)$.
\end{proposition}

\begin{proof}
In view of Lemma~\ref{conjugate_subgroups} we may assume that $g = u
\si z$, where $u \in N_2$, $\si_0 \in N_G(T)$, ${z \in Z}$. Regard
the subalgebra $\mf u_0 = \bigoplus \limits_{\al \in \De_+:
\tau(\al) = 0} \mf g_\al$ of the Lie algebra $\mf z$. It is the
tangent algebra of a maximal unipotent subgroup $U_0$ of~$Z$.
Besides, $\mf u_0 \subset \mf h_1$. Since $\Ad(g)\mf u_0 \subset \mf
u$, we have $\Ad(z) \mf u_0 \subset \Ad(\si_0^{-1}) \mf u$. From
this it follows that the projection of the algebra $\Ad(z) \mf u_0$
to the subspace $\mf t \subset \mf z$ is zero. Besides, for every
root $\al \in \De$ the projection of $\Ad(z) \mf u_0$ to one of the
spaces $\mf g_\al$ or $\mf g_{-\al}$ is zero. Counting the
dimensions yields that $\Ad(z)\mf u_0$ is a regular (that is,
normalized by~$T$) subalgebra in $\mf z$ and the subalgebra $\mf t
\oplus \Ad(z) \mf u_0$ is a Borel subalgebra in $\mf z$ containing
the Cartan subalgebra~$\mf t$. Hence there is an element $\si_1 \in
N_Z(T) \subset N_G(T)$ such that $\Ad(\si_1) (\mf t \oplus \Ad(z)
\mf u_0) = \mf t \oplus \mf u_0$ and $\Ad(\si_1z)\mf u_0 = \mf u_0$.
Thus $z \in N_G(T) \cdot N_Z(U_0)$. Since $N_Z(U_0) = T
\rightthreetimes U_0$ and $U_0 \subset N_1$, we have $z \in N_G(T)
\cdot N_1$. From this we finally obtain $g \in N_2 \cdot N_G(T)
\cdot N_1$.
\end{proof}

\subsection{}

In this subsection we introduce the notion of an elementary
transformation and prove the main theorem of this section
(Theorem~\ref{elementary_transformations}). We use the notation
introduced in~\S\,\ref{simple_applications}.

Let $H \subset G$ be a connected solvable spherical subgroup
standardly embedded in~$B$.

\begin{dfn}
An active root $\al$ is called \emph{regular} if the set $\Psi_i$
containing $\al$ consists of one element, that is, $\Psi_i =
\{\al\}$.
\end{dfn}

It is easy to see that an active root $\al$ is regular if and only
if the projection of the subspace $\mf n \subset \mf u$ to $\mf
g_\al$ is zero. It is also clear that the subgroup $H$ is regular
(that is, normalized by~$T$) if and only if all active roots of $H$
are regular.

We denote by $\Psi^{\reg}(H)$ the set of regular active roots
of~$H$. We also put $\Omega(H) = \De_+ \backslash \Psi^{\reg}(H)$.
Clearly, $\al \in \Omega(H)$ if and only if the projection of the
space $\mf n$ to $\mf g_\al$ is non-zero.

\begin{dfn}
Suppose that $\al \in \Psi^{\reg}(H) \cap \Pi$. An \emph{elementary
transformation with center $\al$} (or, simply, an \emph{elementary
transformation}) is a transformation of the form $H \mapsto \si_\al
H \si_\al^{-1}$, where $\si_\al \in N_G(T)$ is a representative of
the reflection~$r_\al$.
\end{dfn}

Since $r_\al(\be) \in \De_+$ for every $\be \in \De_+
\backslash\{\al\}$, the subgroup $\si_\al H \si_\al^{-1}$ is also
standardly embedded in~$B$.

Let $C_0 \subset Q$ be the dominant Weyl chamber, that is, $C_0 =
\{x \in Q \mid (x, \al) \ge 0$ for every $\al \in \Pi$\}. For every
Weyl chamber $C \subset Q$ we denote by $P(C)$ the set of positive
roots with respect to~$C$. Clearly, $P(C_0) = \De_+$.

Now let us study the following question. Suppose we are given a
connected solvable spherical subgroup $H \subset G$ standardly
embedded in~$B$. Let us find all subgroups that are conjugate to $H$
and also standardly embedded in~$B$. Let $H'$ be such a subgroup.
Then by Proposition~\ref{conjugate_spherical} we have $H' = \si H
\si^{-1}$ for some~$\si \in N_G(T)$. Let $w$ be the image of $\si$
in the Weyl group~$W$. We have $w \Omega(H)\subset \De_+ = P(C_0)$,
whence $\Omega(H) \subset P(w^{-1}C_0)$. Conversely, let $C$ be a
Weyl chamber such that $\Omega(H) \subset P(C)$. Denote by $w_C$ the
unique element of $W$ such that $C = w_C^{-1}C_0$. Then, evidently,
the subgroup $H' = \overline w_C H \overline w_C^{-1}$ is standardly
embedded in~$B$.

\begin{lemma}\label{big_cone}
A Weyl chamber $C$ satisfies the condition $\Omega (H) \subset P(C)$
if and only if it is contained in the cone $X(H) = \{x \mid (x, \al)
\ge 0\: \text{for all}\: \al \in \Omega(H)\} \subset Q$.
\end{lemma}

\begin{proof}
This follows from the fact that for a root $\al$ the condition $\al
\in P(C)$ is equivalent to the condition $(\al, x) \ge 0$ for all $x
\in C$.
\end{proof}

Let $H$, $H'$, $\si$, $w$ be as above. Then we have $\Omega(H)
\subset P(w^{-1}C_0)$ and, by Lemma~\ref{big_cone}, $w^{-1}C_0
\subset X(H)$. Being the intersection of several half-spaces, the
cone $X(H)$ is convex. Clearly, $X(H)$ is a union of several Weyl
chambers. Therefore there are Weyl chambers $C_1, C_2, \ldots,
C_{n-1}$ contained in $X(H)$ such that in the sequence $C_0, C_1,
\ldots, C_{n-1}, {C_n = w^{-1}C_0}$ any two successive Weyl chambers
have a common facet. For $i = 1, \ldots, n$ denote by $w_i$ the
reflection with respect to the common facet of the chambers
$C_{i-1}$ and $C_i$, $w_i \in W$, $w_i^2 = e$. Then for $i = 1,
\ldots, n$ we have $C_i = w_iw_{i-1}\ldots w_1 C_0$. Further, for
every $i = 1, \ldots, n$ there is a simple reflection $r_i$
satisfying the condition $w_i = (w_{i-1}w_{i-2} \ldots
w_1)r_i(w_{i-1}w_{i-2} \ldots w_1)^{-1}$. Denote by $\al_i$ the
simple root corresponding to $r_i$. We obtain $C_i = r_1 \ldots
r_{i-1} r_i C_0 = (r_i r_{i - 1} \ldots r_1)^{-1} C_0$, $i = 1,
\ldots, n$. Recall that $C_i \subset X(H)$ for all $i = 1, \ldots,
n$, therefore in view of Lemma~\ref{big_cone} the subgroup $H_i =
(\overline r_i \overline r_{i-1} \ldots \overline r_1)H(\overline
r_i \overline r_{i-1} \ldots \overline r_1)^{-1}$ is standardly
embedded in~$B$. Hence for $i = 1, \ldots, n$ we obtain that $H_i =
\overline r_i H_{i-1}\overline r_i^{-1}$ (we put $H_0 = H$), the
root $\al_i$ is active with respect to the group $H_{i-1}$, and the
transformation $H_{i-1} \mapsto H_i$ is an elementary transformation
with center~$\al_i$. Clearly, $H' = tH_nt^{-1}$ for some~$t \in T$.
Then the chain of elementary transformations $H \mapsto H_1 \mapsto
\ldots \mapsto H_{n-1} \mapsto (t\overline
r_{\al_n})H_{n-1}(t\overline r_{\al_n})^{-1} = H'$ transforms $H$
to~$H'$. Thus we have proved the following theorem.

\begin{theorem}\label{elementary_transformations}
Two connected solvable spherical subgroups $H, H' \subset G$
standardly embedded in $B$ are conjugate in $G$ if and only if $H'$
can be obtained from $H$ by applying a suitable sequence of
elementary transformations.
\end{theorem}

Theorems~\ref{uniqueness_theorem}, \ref{existence_theorem},
and~\ref{elementary_transformations} provide a complete
classification of connected solvable spherical subgroups of $G$ up
to conjugation.

\subsection{}

Now let us find out how the set of combinatorial data of a connected
solvable spherical subgroup is changed under an elementary
transformation. We consider two connected solvable spherical
subgroups $H$ and $H'$ standardly embedded in $B$. Suppose that $H'$
is obtained from $H$ by an elementary transformation with
center~$\al$, where $\al$ is a regular active simple root of $H$
(and also of $H'$). Suppose that $\Upsilon(H) = (S, \mr M, \pi,
\sim)$ and $\Upsilon(H') = (S', \mr M', \pi', \sim')$. We have $H' =
\si_\al H \si_\al^{-1}$ for some representative $\si_\al \in N_G(T)$
of~$r_\al$, whence $S' = \si_\al S \si_\al^{-1}$.

\begin{lemma}\label{M_and_M'}
We have:

\textup{(a)} $\Psi(H') = r_\al(\Psi(H) \backslash \{\al\}) \cup
\{\al\}$;

\textup{(b)} $\pi'(r_\al(\be)) = \pi(\be)$ for $\be \in \Psi(H)
\backslash \{\al\}$;

\textup{(c)} $\mr M' \backslash \{\al\} = r_\al (\mr M \backslash
\{\al\})$;
\end{lemma}

\begin{proof}
Assertion~(a) is obvious. Let us prove~(b). Suppose that $\be \in
\Psi(H) \backslash \{\al\}$. First of all, note that $\pi(\be) \ne
\al$, whence $\pi(\be) \in \Supp r_\al(\be)$. Let $r_\al(\be) =
\be_1 + \be_2$ be an arbitrary representation of the root
$r_\al(\be) \in \Psi(H')$ as a sum of two positive roots with $\be_1
\in \Psi(H')$. It suffices to show that $\pi(\be) \notin \Supp
\be_1$. This holds if $\be_1 = \al$. Further we assume that $\be_1
\ne \al$. Besides, we have $\be_2 \ne \al$ since $\be_2 \notin
\Psi(H')$. Therefore, both roots in the right-hand side of the
equality $\be = r_\al(\be_1) + r_\al(\be_2)$ are positive and
${r_\al(\be_1) \in \Psi(H)}$. Hence ${\pi(\be) \notin \Supp
r_\al(\be_1)}$. Moreover, $\pi(\be) \notin \Supp r_\al(\be_1) \cup
\{\al\}$. Since $\Supp \be_1 \subset {\Supp r_\al(\be_1) \cup
\{\al\}}$, we obtain $\pi(\be) \notin \Supp \be_1$ as desired.

Now let us prove~(c). Suppose that $\be \in \mr M \backslash
\{\al\}$. Assume that the root ${r_\al(\be) \in \Psi(H') \backslash
\{\al\}}$ is not a maximal active root with respect to the
subgroup~$H'$. In this case there are roots $\de \in \Psi(H')
\backslash \{\al\}$ and $\ga \in \De_+ \backslash \Psi(H')$ such
that $r_\al(\be) + \ga = \de$. In particular, $\ga \ne \al$, whence
$r_\al(\ga) \in \De_+$. For the active root $r_\al(\de)$ we obtain
the representation $r_\al(\de) = \be + r_\al(\ga)$ as a sum of two
positive roots, which contradicts the maximality of the active
root~$\be$. Thus, $r_\al(\mr M \backslash \{\al\}) \subset \mr M'
\backslash \{\al\}$. Similarly, $r_\al(\mr M' \backslash \{\al\})
\subset \mr M \backslash \{\al\}$.
\end{proof}

As a consequence of the previous considerations and
Lemma~\ref{M_and_M'} we obtain the following proposition.

\begin{proposition}\label{comb_datas}
The sets of combinatorial data $\Upsilon(H)$ and $\Upsilon(H')$ are
related as follows:

\textup{(1)} $S' = \si_\al S \si_\al^{-1}$;

\textup{(2.1)} if $\al \in \Supp \de$ for some root $\de \in r_\al
(\mr M \backslash \{\al\})$, then:

\textup{(a)} $\mr M' = r_\al(\mr M \backslash \{\al\})$;

\textup{(b)} $\pi'(\be) = \pi(r_\al(\be))$ for every $\be \in \mr
M'$;

\textup{(c)} for all $\be,\ga \in \mr M'$ the relation $\be \sim'
\ga$ holds if and only if $r_\al(\be) \sim r_\al(\ga)$;

\textup{(2.2)} if $\al \notin \Supp \de$ for all $\de \in r_\al(\mr
M \backslash \{\al\})$, then:

\textup{(a)} $\mr M' = r_\al (\mr M \backslash \{\al\}) \cup
\{\al\}$;

\textup{(b)} $\pi'(\be) = \pi(r_\al(\be))$ for every $\be \in \mr M'
\backslash \{\al\}$, $\pi'(\al) = \al$;

\textup{(c)} for all $\be,\ga \in \mr M' \backslash \{\al\}$ the
relation $\be \sim' \ga$ holds if and only if $r_\al(\be) \sim
r_\al(\ga)$; for every $\be \in \mr M' \backslash \{\al\}$ we have
$\be \not\sim' \al$.
\end{proposition}

\subsection{}

Let $H \subset G$ be a connected solvable spherical subgroup
standardly embedded in~$B$. Suppose that $\Upsilon(H) = (S, \mr M,
\pi, \sim)$. In this subsection we find out how one can determine
all regular active simple roots of $H$ given the set $\Upsilon_0(H)
= (\mr M, \pi, \sim)$. The notation used in this subsection is the
same as in~\S\,\ref{simple_applications}.

\begin{proposition}\label{sim_act_reg}
A root $\al \in \Pi$ is a regular active root of\, $H$ in exactly
one of the following two cases:

\textup{(1)} $\al \in \mr M$ and $\be \nsim \al$ for all $\be \in
\mr M \backslash \{\al\}$;

\textup{(2)} $\al \notin \mr M$ and there is a root $\al' \in \mr M$
such that:

\textup{(a)} $\al$ is terminal with respect to~$\Supp \al'$;

\textup{(b)} $\al \ne \pi(\al')$;

\textup{(c)} $\Supp \al' \backslash \{\al\} \not\subset \Supp \be$
for every root $\be \in \mr M \backslash \{\al'\}$.
\end{proposition}

\begin{proof}
Let $\al \in \Pi$ be a regular active root of~$H$. If $\al \in \mr
M$, then, evidently, condition~(1) is fulfilled. Now suppose $\al
\notin \mr M$. Then $\al \in F(\al')$ for some root $\al' \in \mr
M$, at that, $\al \ne \pi(\al')$. In view of
Corollary~\ref{supp&active} the root $\al$ is terminal with respect
to~$\Supp \al'$. Assume that $\Supp \al' \backslash \{\al\} \subset
\Supp \be$ for some root $\be \in \mr M \backslash \{\al'\}$. Then
in view of the condition $\pi(\al') \ne \al$ and
Propositions~\ref{intersection_diff_weights},~\ref{intersection_equal_weights}
we obtain that $\al' \sim \be$ and for $\al'$, $\be$ one of
possibilities $(\mr E1)$ or $(\mr E2)$ is realized. In case $(\mr
E1)$ we have $\Supp \al' = \{\al, \pi(\al')\}$. Since $\al' -
\pi(\al') \in F(\al')$ and $\pi(\al') \notin \Supp(\al' -
\pi(\al'))$, we obtain $\al' - \pi(\al') = \al$. Hence $\tau(\al) =
\tau(\be - \pi(\al'))$ and the root $\al$ is not regular. In case
$(\mr E2)$ the type of the root system $\De_+ \cap \langle \Supp
\al' \cup \Supp \be \rangle$ is $\ms D$ or $\ms E$, whence $\al' -
\al \in \De_+$ and $\be - (\al' - \al) \in F(\be)$. Then $\tau(\al)
= \tau(\be - (\al' - \al))$ and the root $\al$ is not regular. This
contradiction proves that condition~(2) takes place.

Conversely, if condition~(1) holds, then, evidently, $\al$ is a
regular active root. Now assume that condition~(2) holds. By
Corollary~\ref{extreme_root} we obtain $\al \in \Psi$. Assume that
$\al$ is not a regular active root. Then $\tau(\al) = \tau(\ga)$ for
some root $\ga \in \Psi \backslash \{\al\}$. In view of
Proposition~\ref{crucial} we have $\be = \ga + (\al' - \al) \in \mr
M$, whence $\Supp \al' \backslash \{\al\} = \Supp \al' \cap \Supp
\be$, which contradicts condition~(c).
\end{proof}

\begin{remark}
Propositions~\ref{comb_datas} and~\ref{sim_act_reg} together with
Remark~\ref{at_least_one_solvable} allow one to define the notion of
an elementary transformation of a set $(\mr M, \pi, \sim)$
satisfying conditions $(\mr A)$, $(\mr D)$, $(\mr E)$, and~$(\mr
C)$.
\end{remark}

\subsection{} \label{sober_subgroups}

In this subsection we consider an application of the theory
developed above to an important class of connected solvable
spherical subgroups. Namely, we obtain, up to conjugation, a
classification of all connected solvable spherical subgroups in~$G$
having finite index in their normalizer. Following Vinberg
(see~\cite{Vin},~\S\,1.3.4), we use the term
\emph{saturated}\footnote{Another possible term is \emph{sober}.}
for connected spherical subgroups $H \subset G$ with $H = N_G(H)^0$.
Besides, we obtain a classification up to conjugation of all
unipotent subgroups in $G$ that are unipotent radicals of connected
solvable spherical subgroups in~$G$.

\begin{lemma}\label{normalizer}
Let $H \subset G$ be a connected solvable spherical subgroup
standardly embedded in~$B$. Put $N = H \cap U$. Then $N_G(H)^0 =
\widehat S \rightthreetimes N$, where $\widehat S = (N_G(N)\cap
T)^0$.
\end{lemma}

\begin{proof}
Put $S = H \cap T$. Let $\mf n (\mf g)$ be the normalizer of the
algebra $\mf h$ in~$\mf g$. It suffices to prove that $\mf n(\mf h)
= \widehat {\mf s} + \mf n$. Evidently, $\mf n(\mf h) \subset \mf z
+ \mf h$, where $\mf z$ is the centralizer of the algebra $\mf s$
in~$\mf g$ ($\mf z$ is the tangent algebra of the reductive
group~$Z_G(S)$). Put $\mf u_0 = \mf z \cap \mf u$. Since the
subgroup $H$ is spherical, we have $\mf u_0 \subset \mf h$, whence
$\mf z \cap \mf h = \mf s + \mf u_0$. Clearly, $\mf u_0$ is a
maximal nilpotent subalgebra in~$\mf z$, hence its normalizer in
$\mf z$ coincides with the algebra $\mf t + \mf u_0$. From this it
follows that $\mf n(\mf h) \subset \mf t + \mf h$ and therefore $
\mf n(\mf h) = \widehat {\mf s} + \mf n$.
\end{proof}

\begin{corollary}\label{what_is_sober}
Every saturated solvable spherical subgroup $H \subset G$ standardly
embedded in $B$ has the form $H = S \rightthreetimes N$, where $N =
H \cap U$ and $S = (N_G(N) \cap T)^0$. In particular, the torus $S$
is the connected component of the identity of the subgroup in $T$
defined by vanishing of all characters of the form $\al - \be$,
where $\al,\be$ run over all roots in $\mr M(H)$ with $\al \sim
\be$.
\end{corollary}

\begin{corollary}\label{normalizer_is_sober}
For every connected solvable spherical subgroup $H \subset G$ the
subgroup $N_G(H)^0$ is a saturated solvable spherical subgroup
in~$G$.
\end{corollary}

\begin{lemma} \label{sober_is_determined_by_ur}
Up to conjugation, every saturated solvable spherical subgroup
in~$G$ is uniquely determined by its unipotent radical.
\end{lemma}

\begin{proof}
As follows from Lemma~\ref{normalizer}, a maximal torus of a
saturated solvable spherical subgroup $H \subset G$ with unipotent
radical $N$ is a maximal torus in the group~$N_G(N)$. This implies
the assertion of the lemma, since all maximal tori in $N_G(N)$ are
conjugate.
\end{proof}

\begin{corollary}\label{sober_bij_ur}
Let $H \subset G$ be a saturated solvable spherical subgroup and $N$
its unipotent radical. Then the map $H \mapsto N$ is a bijection
between conjugacy classes in $G$ of saturated solvable spherical
subgroups and conjugacy classes in $G$ of unipotent radicals of
connected solvable spherical subgroups.
\end{corollary}

\begin{proof}
The injectivity of this map follows from
Lemma~\ref{sober_is_determined_by_ur}, and the surjectivity follows
from Lemma~\ref{normalizer} and Corollary~\ref{normalizer_is_sober}.
\end{proof}

We denote by $\widetilde \Upsilon_0$ the set of all triples $(\mr M,
\pi, \sim)$, where $\mr M \subset \De_+$ is a subset, $\pi: \mr M
\to \Pi$ is a map, $\sim$ is an equivalence relation on $\mr M$, and
conditions $(\mr A)$, $(\mr D)$, $(\mr E)$, and $(\mr C)$ are
satisfied. We recall (see Remarks~\ref{remark_unipotent_radical}
and~\ref{at_least_one_solvable}) that to each triple $(\mr M, \pi,
\sim) \in \widetilde \Upsilon_0$ there corresponds a unique, up to
conjugation by elements of\,~$T$, subgroup $N = N(\mr M, \pi, \sim)
\subset U$ that is the unipotent radical of a connected solvable
spherical subgroup in $G$ standardly embedded in~$B$.

\begin{proposition}\label{sober}
\textup{(a)} The map $H \mapsto \Upsilon_0(H)$ is a bijection
between the set of saturated solvable spherical subgroups in $G$
standardly embedded in $B$, up to conjugation by elements of\,~$T$,
and the set $\widetilde \Upsilon_0$. Two saturated solvable
spherical subgroups $H_1, H_2 \subset G$ standardly embedded in $B$
are conjugate in $G$ if and only if there is a sequence of
elementary transformations taking the set $\Upsilon_0(H_1)$ to the
set~$\Upsilon_0(H_2)$.

\textup{(b)} The map $(\mr M, \pi, \sim) \mapsto N(\mr M, \pi,
\sim)$ is a bijection between the set $\widetilde \Upsilon_0$ and
the set of unipotent radicals of connected solvable spherical
subgroups in $G$ standardly embedded in~$B$, up to conjugation by
elements of~$T$. Two subgroups ${N(\mr M, \pi, \sim)}$, $N(\mr M',
\pi', \sim')$ are conjugate in $G$ if and only if there is a
sequence of elementary transformations taking the set $(\mr M, \pi,
\sim)$ to the set~$(\mr M', \pi', \sim')$.

In particular, the set of conjugacy classes in $G$ of saturated
solvable spherical subgroups, as well as the set of conjugacy
classes in $G$ of unipotent radicals of connected solvable spherical
subgroups, is finite.
\end{proposition}

\begin{proof}
(a) The injectivity of this map follows from
Remark~\ref{remark_unipotent_radical},
Corollary~\ref{what_is_sober}, and Theorem~\ref{uniqueness_theorem}.
Let us prove the surjectivity. Suppose that ${(\mr M, \pi, \sim)}
\in \widetilde \Upsilon_0$. Consider the subtorus $S \subset T$ that
is the connected component of the identity of the subgroup in~$T$
defined by vanishing of all characters of the form $\al - \be$,
where $\al,\be$ run over all roots in $\mr M$ with~$\al \sim \be$.
Then $S$ satisfies condition $(\mr T)$ (at that, $S$ is the largest
subtorus in~$T$ satisfying this condition), whence by
Theorem~\ref{existence_theorem}, up to conjugation by elements
of~$T$, there is a unique connected solvable spherical subgroup $H$
standardly embedded in~$B$ with $\Upsilon(H) = (S, \mr M, \pi,
\sim)$. Let $N \subset U$ be the unipotent radical of~$H$. Then,
evidently, $S = N_G(N) \cap T$ and by Corollary~\ref{what_is_sober}
the subgroup $H$ is saturated. The second part of assertion~(a)
follows from the first one and
Theorem~\ref{elementary_transformations}.

(b) Let $H \subset G$ be a saturated solvable spherical subgroup
standardly embedded in~$B$ and suppose that $\Upsilon_0(H) = (\mr M,
\pi, \sim)$. Then, up to conjugation by elements of~$T$, the
subgroup $N = N(\mr M, \pi, \sim)$ is the unipotent radical of~$H$.
Now the desired assertion follows from~(a) and
Corollary~\ref{sober_bij_ur}.
\end{proof}

\section{Simplification of the set of combinatorial data corresponding to
a connected solvable spherical subgroup} \label{simplification}

This section is devoted to problems concerned with `simplification'
of the set of combinatorial data $\Upsilon(H)$ corresponding to a
connected solvable spherical subgroup $H\subset G$ standardly
embedded in~$B$.

\subsection{}

In this subsection we show that every conjugacy class in $G$ of
connected solvable spherical subgroups contains a subgroup $H$
standardly embedded in $B$ such that the set $\Upsilon(H)$ satisfies
stronger conditions than those appearing in
Theorem~\ref{uniqueness_theorem}. Thereby the set $\Upsilon(H)$ is
in a sense simpler than a set of the general form.

Until the end of this subsection, $H$ stands for a connected
solvable spherical subgroup standardly embedded in~$B$.

\begin{dfn}
An active root $\al$ is called \emph{typical} if $\al = \sum
\limits_{\de \in \Supp \al} \de$. An active root $\al$ is called
\emph{non-typical} if $\al$ is not typical.
\end{dfn}

It is easy to see that an active root $\al$ is typical if and only
if it is contained in row~1 of Table~\ref{table_active_roots}.

\begin{lemma}\label{star_roots}
Let $\al$ be a non-typical maximal active root. Then the simple root
$\de \in F(\al) \cap \Supp \al$ marked by an asterisk in
Table~\textup{\ref{table_non-typical_active_roots}} is a regular
active root. \textup{(}The notation in
Table~\textup{\ref{table_non-typical_active_roots}} is the same as
in Table~\textup{\ref{table_active_roots}}.\textup{)}
\end{lemma}

\begin{table}[h]

\caption{}\label{table_non-typical_active_roots}

\begin{center}

\begin{tabular}{|c|c|}

\hline

Type of $\De(\al)$ & $\al$\\

\hline

$\ms B_n$ & $\al_1 + \al_2 + \ldots + \al_{n-1} + 2\al_n^*$\\

\hline

$\ms C_n$ & $2\al_1^* + 2\al_2 + \ldots + 2\al_{n-1} + \al_n$\\

\hline

$\ms F_4$ & $2\al_1^* + 2\al_2 + \al_3 + \al_4$\\

\hline

$\ms G_2$ & $2\al_1^* + \al_2$\\

\hline

$\ms G_2$ & $3\al_1^* + \al_2$\\

\hline
\end{tabular}

\end{center}

\end{table}

\begin{proof}
Assume that the active root $\de$ is not regular. Then there exists
an active root $\de' \ne \de$ such that $\tau(\de') = \tau(\de)$.
Since $\al - \de \in \De_+$, by Proposition~\ref{crucial} we obtain
$\be = \de' + (\al - \de) \in \mr M$. Meanwhile, $\Supp (\al - \de)
= \Supp \al$, therefore $\Supp \al \subset \Supp \be$, which is
impossible in view of maximality of the active roots $\al$
and~$\be$.
\end{proof}

\begin{proposition}\label{get_rid_of_ntyp_roots}
There is an element $w \in W$ such that $\overline w H \overline
w^{-1} \subset B$ and all maximal active roots of the subgroup
$\overline w H \overline w^{-1}$ are typical.
\end{proposition}

\begin{proof}
Put $m = m(H)= \sum \hgt \de$, where $\de$ runs over all non-typical
maximal active roots of~$H$. Let us prove the assertion by induction
on~$m$. For $m = 0$ there is nothing to prove. Assume that the
assertion is proved for all $m < k$, where $k
> 0$, and consider the case $m = k$. Let $\al$ be a non-typical
maximal active root. Consider the indecomposable component $\De^0$
of $\De$ containing $\al$ and denote by $\Pi^0$ its set of simple
roots, $\Pi^0 = \Pi \cap \De^0$. Assume that $\Pi^0 = \{\al_1,
\ldots, \al_n\}$, where $n = |\Pi^0|$. Put also $\De^0_+ = \De_+
\cap \De^0$. Let $\de \in \Supp \al$ be the simple root marked by an
asterisk in Table~\ref{table_non-typical_active_roots}. By
Lemma~\ref{star_roots} the root $\de$ is a regular active root.
Regard the subgroup $H' = \overline r_\de H \overline r_\de^{-1}$
and show that $m(H') < m(H)$. Note that by Lemma~\ref{M_and_M'}(c)
all non-typical maximal active roots of $H'$ are contained in the
set~$r_\de(\mr M(H) \backslash \{\de\})$. Since $r_\de(\be) = \be$
for all roots $\be \in \mr M(H) \backslash \{\al\}$ orthogonal
to~$\de$, it suffices to show that the remaining maximal active
roots of $H$ (including~$\al$) make a positive contribution to the
difference $m(H) - m(H')$. Further we consider all possibilities
for~$\De^0$ and~$\al$.

$1^\circ$. $\De^0$ if of type~$\ms B_n$. From
Theorem~\ref{active_roots} it follows that every non-typical active
root contained in $\De^0_+$ contains the root $\al_n$ in its
support. Hence $\al$ is the unique non-typical maximal active root
contained in~$\De^0_+$. We have $\al = \al_l + \al_{l+1} + \ldots +
\al_{n-1} + 2\al_n$, where $1 \le l \le n - 1$. Then $\de = \al_n$.
Further, $r_\de(\al) = \al_l + \al_{l+1} + \ldots + \al_{n-1} \in
\mr M(H')$ is a typical active root, therefore the contribution of
$\al$ to $m(H) - m(H')$ is $\hgt \al$. If a root $\be \in \mr M(H)
\backslash \{\al\}$ is not orthogonal to $\de$, then by
Propositions~\ref{intersection_diff_weights}
and~\ref{intersection_equal_weights} we obtain that $l = n - 1$,
$\be = \al_p + \al_{p + 1} + \ldots + \al_{n - 1}$, where $1 \le p
\le n - 2$, and $\pi(\be) = \pi(\al) = \al_{n-1}$. In other words,
for the roots $\al$ and $\ga$ possibility $(\mr E1)$ is realized,
whence $2\al_n = \al - \al_{n-1} \in F(\al) \subset \De_+$, which is
not the case. Therefore $m(H) - m(H') = \hgt \al > 0$.

$2^\circ$. $\De^0$ is of type~$\ms C_n$. From
Theorem~\ref{active_roots} it follows that every non-typical active
root contained in $\De^0_+$ contains the root $\al_n$ in its
support. Hence $\al$ is the unique non-typical maximal active root
containted in~$\De^0_+$. We have $\al = 2\al_l + 2\al_{l+1} + \ldots
+ 2\al_{n-1} + \al_n$, where $1 \le l \le n - 1$. Then $\de =
\al_l$. We have $r_{\de}(\al) = 2\al_{l + 1} + \ldots + 2\al_{n - 1}
+ \al_n$, whence $\hgt \al - \hgt(r_\de(\al)) = 2$. If a root $\be
\in \mr M(H) \backslash \{\al\}$ is not orthogonal to~$\de$, then by
Propositions~\ref{intersection_diff_weights}
and~\ref{intersection_equal_weights} we obtain that $\Supp \be
\subset \{\al_1, \ldots, \al_l\}$. It follows that each of the roots
$\be$ and $r_\de(\be)$ is typical. Therefore $m(H) - m(H') = 2
> 0$.

$3^\circ$. $\De^0$ is of type~$\ms F_4$. By
Theorem~\ref{active_roots} four cases are possible. Let us consider
them separately.

\emph{Case}~1. $\al = 2\al_2 + \al_3$. Then $\de = \al_2$. We have
$r_\de(\al) = \al_3$, whence the contribution of $\al$ to $m(H) -
m(H')$ is $\hgt \al = 3$. If a root $\be \in \mr M(H) \backslash
\{\al\}$ is not orthogonal to~$\de$, then by
Propositions~\ref{intersection_diff_weights}
and~\ref{intersection_equal_weights} we obtain that either $\Supp
\be \subset \{\al_1, \al_2\}$ or $\be = \al_3 + \al_4$. In the first
case each of the roots $\be$ and $r_\de(\be)$ is typical. In the
second case for the roots $\al$ and $\be$ possibility $(\mr E1)$ is
realized, whence $2\al_2 = \al - \al_3 \in \Psi \subset \De_+$,
which is not the case. Therefore $m(H) - m(H') = 3
> 0$.

\emph{Case}~2. $\al = 2\al_2 + \al_3 + \al_4$. Then $\de = \al_2$.
Since $r_\de(\al) = \al_3 + \al_4$ is a typical root, the
contribution of $\al$ in $m(H) - m(H')$ is $\hgt \al = 4$. If a root
$\be \in \mr M(H) \backslash \{\al\}$ is not orthogonal to~$\de$,
then by Propositions~\ref{intersection_diff_weights}
and~\ref{intersection_equal_weights} we obtain that $\Supp \be
\subset \{\al_1, \al_2\}$, whence each of the roots $\be$ and
$r_\de(\be)$ is typical. Therefore $m(H) - m(H') = 4 > 0$.

\emph{Case}~3. $\al = 2\al_1 + 2\al_2 + \al_3$. Then $\de = \al_1$.
We have $r_\de(\al) = 2\al_2 + \al_3$, whence $\hgt \al - \hgt
(r_\de(\al)) = 2$. From Propositions~\ref{intersection_diff_weights}
and~\ref{intersection_equal_weights} it follows that all roots in
$\mr M(H) \backslash \{\al\}$ are orthogonal to~$\de$. Therefore
$m(H) - m(H') = 2 > 0$.

\emph{Case}~4. $\al = 2\al_1 + 2\al_2 + \al_3 + \al_4$. Then $\de =
\al_1$. We have $r_\de(\al) = 2\al_2 + \al_3 + \al_4$, whence $\hgt
\al - \hgt (r_\de(\al)) = 2$. Clearly, all roots in $\mr M(H)
\backslash \{\al\}$ are orthogonal to~$\de$. Therefore $m(H) - m(H')
= 2
> 0$.

$4^\circ$. $\De^0$ is of type~$\ms G_2$. By
Theorem~\ref{active_roots} we have $\al = 2\al_1 + \al_2$ or $ \al =
3\al_1 + \al_2$. In both cases, $\de = \al_1$. For $\al = 2\al_1 +
\al_2$ and $\al = 3\al_1 + \al_2$ we have $r_\de(\al) = \al_1 +
\al_2$ and $r_\de(\al) = \al_2$, respectively, therefore the
contribution of $\al$ to $m(H) - m(H')$ is $\hgt \al$. Clearly, all
roots in $\mr M(H) \backslash \{\al\}$ are orthogonal to~$\de$.
Therefore $m(H) - m(H') = \hgt \al \ge 3 > 0$.

Thus we have obtained that $m(H) > m(H')$. By the induction
hypothesis there is an element $w' \in w$ such that $H'' = \overline
w'H' \overline w'^{-1} \subset B$ and all maximal active roots of
the subgroup $H''$ are typical. Then it is easy to see that for $w =
w'r_\de$ and some $t \in T$ we have $\overline wH \overline w^{-1} =
tH''t^{-1}$, which completes the proof.
\end{proof}

\begin{remark}
If all maximal active roots of $H$ are typical, then all active
roots of $H$ are typical.
\end{remark}

\begin{remark}
If all roots in $\mr M(H)$ are typical, then for two different roots
in $\mr M(H)$ the condition~$(\mr E1)$ is equivalent to the
following condition:

$(\mr E1')$ $\Supp \al \cap \Supp \be = \{\ga\}$, where $\ga =
\pi(\al) = \pi(\be)$, and the root $\ga$ is terminal with respect to
both $\Supp \al$ and $\Supp \be$.
\end{remark}

The next step of the simplification of the set of combinatorial data
corresponding to a connected solvable spherical subgroup standardly
embedded in $B$ is the following proposition.

\begin{proposition} \label{simple_intersection}
Suppose that all maximal active roots of $H$ are typical. Then there
is an element $w \in W$ such that $H' = \overline w H \overline
w^{-1} \subset B$, all maximal active roots of $H'$ are typical, and
for any two maximal active roots $\al, \be$ of $H$ one of
possibilities $(\mr D0)$, $(\mr E1')$, or $(\mr E2')$ is realized,
where $(\mr E2')$ is as follows:

$(\mr E2')$ the diagram $\Sigma(\Supp \al \cup \Supp \be)$ has the
form shown on Figure~\ref{diagram_difficult} \textup{(}for some ${p,
q, r \ge 1}$\textup{)}, $\al = \al_1 + \ldots + \al_p + \ga_0 +
\ga_1 + \ldots + \ga_r$, $\be = \be_1 + \ldots + \be_q + \ga_0 +
\ga_1 + \ldots + \ga_r$, and $\pi(\al) = \pi(\be) = \ga_r$.
\end{proposition}

Before we prove this proposition, let us prove several auxiliary
lemmas.

\begin{lemma} \label{regular_simple_root}
Suppose that a simple root $\de$ is active and is contained in
supports of at least two different maximal active roots. Then $\de$
is a regular active root.
\end{lemma}

\begin{proof}
Let $\al, \be$ be different maximal active roots such that $\de \in
\Supp \al \cap \Supp \be$. Assume that there is an active root $\de'
\ne \de$ with $\tau(\de') = \tau(\de)$. Then by
Proposition~\ref{crucial} we obtain $\al' = \de' + (\al - \de) \in
\mr M$ and $\be' = \de' + (\be - \de) \in \mr M$, at that, $\al' \ne
\be'$. Since $\de' \ne \de$, we have $\al' \ne \al$ and $\be' \ne
\be$. Then the linear dependence $\al + \be' = \be + \al'$
contradicts Lemma~\ref{Psi_maximal_elements} even in the case when
some of the roots $\al, \be, \al', \be'$ coincide.
\end{proof}

\begin{lemma} \label{Weyl_transformation}
Suppose that $\al$ is a typical maximal active root, $|\Supp \al|\ge
2$, and the root system $\De(\al)$ is not of type~$\ms G_2$. Let
$\de \in \Supp \al$ be a regular active root. Then:

\textup{(a)} if in the diagram $\Supp \al$ the node $\de$ is
incident to a double edge with the arrow directed to~$\de$, then
$r_\de(\al) = \al$;

\textup{(b)} in all other cases we have $r_\de(\al) = \al - \de$.
\end{lemma}

\begin{proof}
By corollary~\ref{supp&active} the root $\de$ is terminal with
respect to $\Supp \al$. The remaining part of the proof is obtained
by a direct check.
\end{proof}

\begin{lemma} \label{auxiliary_simple_intersection}
Suppose that $\al, \be$ are different maximal active roots and an
active simple root $\de$ is such that $\de \in \Supp \al \cap \Supp
\be$. Let $\al_0$ denote the \textup{(}unique\textup{)} node of the
diagram $\Sigma(\Supp \al)$ joined by an edge with the node~$\de$.
Then for every maximal active root $\ga \ne \al$ with $\al_0 \in
\Supp \ga$ we have $\de \in \Supp \ga$.
\end{lemma}

\begin{proof}
Assume that a maximal active root $\ga \ne \al$ satisfies $\al_0 \in
\Supp \ga$ and $\de \notin \Supp \ga$. In view of
Corollary~\ref{supp&active} the root $\de$ is terminal with respect
to $\Supp \al$. If $\al_0$ is also terminal with respect to $\Supp
\al$, then $|\Supp \al| = 2$ and we have a contradiction with
Proposition~\ref{covering_of_support}. Now assume that $\al_0$ is
not terminal with respect to $\Supp \al$. Then the set $\Supp \al
\cap \Supp \ga$ contains at least two elements. Since $\de \notin
\Supp \ga$, in view of Propositions~\ref{intersection_diff_weights}
and~\ref{intersection_equal_weights} for the roots $\al$ and $\ga$
possibility $(\mr E2)$ is realized, at that, $\Supp \al \cap \Supp
\ga = \Supp \al \backslash \{\de\}$. Then $\Supp \al \subset \Supp
\be \cup \Supp \ga$, and again we obtain a contradiction with
Proposition~\ref{covering_of_support}.
\end{proof}

For any two different maximal active roots $\al, \be$ of $H$ we
introduce a quantity $f(\al, \be) = f(H, \al, \be)$ as follows. If
the set $\Supp \al \cap \Supp \be$ contains an active simple root,
then we put $f(\al, \be) = |\Supp \al \cap \Supp \be|$. Otherwise we
put $f(\al, \be) = 0$.

\begin{lemma}\label{orthogonal}
Suppose that a regular active simple root $\de$ and a maximal active
root $\al$ are such that in the diagram $\Sigma(\Pi)$ the node $\de$
is joined by an edge with none of the nodes in the set $\Supp \al$.
Then for every maximal active root $\be \notin \{\al, \de\}$ we have
$f(H, \al, \be) = f(\overline r_\de H \overline r_\de^{-1},
r_\de(\al), r_\de(\be))$.
\end{lemma}

\begin{proof}
It follows from the hypothesis that the root $\de$ is not contained
in the set $\Supp \al$ and is orthogonal to all roots contained in
it, whence $r_\de(\al) = \al$. Then for every maximal active root
$\be$ we have $\Supp \al \cap \Supp \be = \Supp r_\de(\al) \cap
\Supp r_\de(\be)$. To complete the proof, it remains to observe that
for a root $\de' \in \Supp \al \cap \Supp \be$ the conditions $\de'
\in \Psi(H)$ and $\de' \in \Psi(\overline r_\de H \overline
r_\de^{-1})$ are equivalent.
\end{proof}

\begin{proof}[Proof of
Proposition~\textup{\ref{simple_intersection}}] Put $f(H) = \sum
f(H, \al, \be)$, where the sum is taken over all (unordered) pairs
of different maximal active roots of~$H$. Let us prove the assertion
by induction on~$f(H)$. If $f(H) = 0$, then for every two different
roots $\al, \be \in \mr M(H)$ the set $\Supp \al \cap \Supp \be$
contains no active roots, which in view of
Propositions~\ref{intersection_diff_weights}
and~\ref{intersection_equal_weights} means that for $\al, \be$ one
of possibilities $(\mr D0)$, $(\mr E1')$, or $(\mr E2')$ is
realized. Assume that for some $k>0$ the assertion is proved for all
subgroups $H$ with $f(H) < k$ and consider the case $f(H) = k$. Let
$\de$ be an active simple root contained in supports of at least two
different maximal active roots of~$H$. Then by
Lemma~\ref{regular_simple_root} the root $\de$ is regular. Regard
the subgroup $H_0 = \overline r_\de H \overline r_\de^{-1}$. Let us
prove that $f(H_0) < f(H)$ and all roots in the set $\mr M(H_0)$ are
typical. Let $\De^0$ denote the indecomposable component of the root
system $\De$ containing~$\de$. Clearly, $\De^0$ is not of type~$\ms
G_2$. In view of Lemma~\ref{M_and_M'}(c) it is easy to see that a
non-zero contribution to $f(H_0)$ can only be made by pairs of roots
in $\mr M(H_0)$ that are contained in $r_\de(\mr M(H))$. Note the
following: if for a root $\al \in \mr M(H)$ the node $\de$ of the
diagram $\Sigma(\De^0)$ is joined by an edge with none of nodes in
$\Supp \al$, then $r_\de(\al) = \al$, which implies that the root
$r_\de(\al)$ is typical. Further, if different roots $\al,\be \in
\mr M(H)$ are such that $\de \in \Supp \al \cap \Supp \be$, then in
view of Propositions~\ref{intersection_diff_weights}
and~\ref{intersection_equal_weights} for $\al, \be$ one of
possibilities $(\mr D1)$, $(\mr D2)$, or $(\mr E2)$ is realized. (At
that, in the latter case possibility $(\mr E2')$ is not realized.)
In view of Lemma~\ref{Weyl_transformation} and the condition that in
the diagram $\Sigma(\Pi)$ the node $\de$ is incident to at most one
double edge, in all the cases we obtain $f(H_0, r_\de(\al),
r_\de(\be)) = f(H, \al ,\be) - 1$. Moreover, each of the roots
$r_\de(\al)$ and $r_\de(\be)$ is typical. In view of
Lemma~\ref{auxiliary_simple_intersection}, the further consideration
is divided into two cases.

\emph{Case}~1. For every root $\al \in \mr M(H)$ with $\de \notin
\Supp \al$ in the diagram $\Sigma(\Pi)$ the node $\de$ is joined by
an edge with none of nodes in~$\Supp \al$. Then for every such root
$\al$ by Lemma~\ref{orthogonal} we obtain that for every root $\be
\in \mr M(H) \backslash \{\al\}$ the equality $f(H, \al, \be) =
f(H_0, r_\de(\al), r_\de(\be))$ holds.

\emph{Case}~2. The diagram $\Sigma(\Pi \cap \De^0)$ has the form
shown on Figure~\ref{diagram_difficult} (for some $p, q, r \ge 1$),
$\de = \ga_{r'}$, where $0 \le r' \le r - 1$, and the set $\mr M(H)$
contains the roots $\al = \al_1 + \ldots + \al_{p'} + \ga_0 + \ga_1
+ \ldots + \ga_{r'}$, $\be = \be_1 + \ldots + \be_{q'} + \ga_0 +
\ga_1 + \ldots + \ga_{r'}$, $\ga = \ga_{r'+1} + \ga_{r'+2} + \ldots
+ \ga_{r''}$ for some $p', q', r''$, where $1 \le p' \le p$, $1 \le
q' \le q$, $r'+1 \le r'' \le r$. In view of
Propositions~\ref{intersection_diff_weights},~\ref{intersection_equal_weights},
and~\ref{covering_of_support} none of the maximal active roots of
$H$ different from $\ga$ contains the root $\ga_{r' + 1}$ in its
support. From this and Lemma~\ref{auxiliary_simple_intersection} it
follows that for every root $\al' \in \mr M(H)$ different from $\al,
\be, \ga$ the root $\de$ is orthogonal to all roots in $\Supp \al'$.
Then by Lemma~\ref{orthogonal} we obtain that for every root $\be'
\in \mr M(H) \backslash \{\al'\}$ the equality $f(H, \al', \be') =
f(H_0, r_\de(\al'), r_\de(\be'))$ holds. Further, it is not hard to
see that $f(H, \al, \ga) = f(H_0, r_\de(\al), r_\de(\ga)) = 0$,
$f(H, \be, \ga) = f(H_0, r_\de(\be), r_\de(\ga)) = 0$, and each of
roots $r_\de(\al)$, $r_\de(\be)$, $r_\de(\ga)$ is typical.

In both cases we have obtained that $f(H_0) < f(H)$ and all roots in
$\mr M(H_0)$ are typical. Then by the induction hypothesis there is
an element $w' \in W$ such that the subgroup $H' = \overline w' H_0
\overline w'^{-1}$ is standardly embedded in~$B$, all roots in $\mr
M(H')$ are typical, and for any two different roots $\al, \be \in
\mr M(H')$ one of possibilities $(\mr D0)$, $(\mr E1')$, or $(\mr
E2')$ is realized. Then $w = w'r_\de$ is the desired element.
\end{proof}

\begin{dfn}
The set $\Upsilon(H) = (S, \mr M, \pi, \sim)$, as well as its subset
$\Upsilon_0(H) = (\mr M, \pi, \sim)$, is called \emph{reduced} if
the following conditions are fulfilled:

$(\mr A')$ $\al = \sum\limits_{\de \in \Supp \al} \de$ for every
$\al \in \mr M$, and $\pi(\al) \in \Supp \al$;

$(\mr D')$ if $\al,\be \in \mr M$ and $\al \nsim \be$, then $\Supp
\al \cap \Supp \be = \varnothing$;

$(\mr E')$ if $\al, \be \in \mr M$ and $\al \sim \be$, then for the
roots $\al, \be$ one of possibilities $(\mr D0)$, $(\mr E1')$, or
$(\mr E2')$ is realized.
\end{dfn}

Propositions~\ref{get_rid_of_ntyp_roots}
and~\ref{simple_intersection} imply the following theorem.

\begin{theorem}\label{simplified}
Every connected solvable spherical subgroup $H\subset G$ is
conjugate to a subgroup $H'$ standardly embedded in $B$ such that
the set $\Upsilon(H')$ is reduced.
\end{theorem}

\subsection{}

In this subsection we prove the following theorem.

\begin{theorem}\label{simplified_conjugate}
Let $H, H' \subset G$ be two connected solvable spherical subgroups
standardly embedded in $B$ such that the sets $\Upsilon(H)$ and
$\Upsilon(H')$ are reduced. Suppose that $H$ and $H'$ are conjugate
in~$G$. Then there is a sequence of elementary transformations
taking $H$ to $H'$ such that for every intermediate subgroup
$\widetilde H$ the set $\Upsilon(\widetilde H)$ is also reduced.
\end{theorem}

Before we prove this theorem, let us prove an auxiliary lemma.

\begin{lemma}\label{node_connected_with_active}
Suppose that $\al \in \Psi(H)$, $|\Supp \al| \ge 2$, and $\de \in
F(\al) \cap \Pi$. Regard \textup{(}the unique\textup{)} node $\de'$
in the diagram $\Sigma(\Supp \al)$ connected by an edge with~$\de$.
Then $\de' \notin \Psi(H)$.
\end{lemma}

\begin{proof}
Assume that $\de' \in \Psi(H')$. Then in view of
Corollary~\ref{supp&active} both roots $\de, \de'$ are terminal with
respect to $\Supp \al$, whence $\Supp \al = \{\de, \de'\}$. In this
situation the roots $\de, \de'$ cannot be active for $H$
simultaneously, a contradiction.
\end{proof}

\begin{proof}[Proof of Theorem~\textup{\ref{simplified_conjugate}}]
Without loss of generality we may assume that the root system $\De$
is indecomposable. If $\De$ is of type~$\ms G_2$, then the assertion
is easily verified by a case-by-case consideration of all possible
sets $(\mr M, \pi, \sim)$ (which are 9 in total) and all sequences
of elementary transformations. Further we assume that $\De$ is of
type different from~$\ms G_2$.

Regard the shortest sequence of elementary transformations taking
$H$ to $H'$, and show that it satisfies the desired property. Let
$\de_1, \de_2, \ldots, \de_k$ be the successive centers of
elementary transformations of this sequence. In view of minimality
of the length of the sequence, for all $i = 1,\ldots, k-1$ we have
$\de_i \ne \de_{i-1}$. Put $H_0 = H$ and for all $i = 1, \ldots, k$
denote by $H_i$ the $i$th intermediate subgroup: $H_i =
\si_iH_{i-1}\si_i^{-1}$, where $\si_i \in N_G(T)$ is some
representative of the reflection $r_{\de_i}$ and $H_k = H'$. Denote
by $\mr M_i$ the set of maximal active roots of~$H_i$, $i = 0,
\ldots, k$. The further argument consists of two steps.

\emph{Step}~1. Let us prove that for all $i = 1, \ldots, k-1$ the
set $\mr M_i$ consists of typical roots. Assume the converse: for
some $j \in \{1, \ldots, k - 1\}$ the set $\mr M_j$ contains a
non-typical root~$\al$, whereas for all $i < j$ the set $\mr M_i$
consists of typical roots. Without loss of generality we may assume
that $j$ is maximal among all shortest sequences of elementary
transformations taking $H$ to~$H'$. It is not hard to see that in
the diagram $\Sigma(\Pi)$ the node $\de_j$ is incident to a double
edge with the arrow directed to~$\de_j$. Let us show that the nodes
$\de_j$ and $\de_{j+1}$ are joined by an edge. Indeed, otherwise the
elementary transformations with centers in $\de_j$ and $\de_{j+1}$
commute, therefore their interchange yields a new sequence such that
$\mr M_i$ can contain a non-typical root only for $i \ge j + 1$, a
contradiction with the choice of~$j$. Further we consider all
possibilities for~$\De$.

$1^\circ$. $\De$ is of type~$\ms B_n$. Then $\al = \al_l + \al_{l+1}
+ \ldots + 2\al_n$, where $1 \le l \le n-1$. We have $\de_{j+1} =
\al_{n-1}$, which is impossible since the root $\al_{n-1}$ cannot be
active.

$2^\circ$. $\De$ is of type~$\ms C_n$, $n \ge 3$. Then $\al =
2\al_{n-1} + \al_n$. The root $\al_n$ is not active, therefore
$\de_{j+1} = \al_{n-2}$. Let us prove by induction that $\de_{j+l} =
\al_{n-l-1}$ for all $l = 1, \ldots, n-2$. For $l = 1$ this is true.
The passage from $l-1$ to $l$ is performed as follows. Clearly, the
set $\mr M_{j + l - 1}$ contains the root $2\al_{n-l} + 2\al_{n-l+1}
+ \ldots + 2\al_{n-1} + \al_n$. Therefore none of the roots
$2\al_{n-l+1}, \ldots, \al_{n-1}, \al_n$ can be active (with respect
to $H_{j + l - 1}$). If $\de_{j+l} \ne \al_{n-l-1}$, then in view of
the condition $\de_{j+l} \ne \de_{j+l-l}$ we obtain that the root
$\de_{j+1}$ is orthogonal to each of the roots $\al_{n-l},
\al_{n-l+1}, \ldots, \al_{n-1}$. Hence in our sequence of elementary
transformations, interchanging successively neighboring elements, we
can move the $(j+l)$th elementary transformation to the $j$th place.
As a result we obtain a new sequence such that the set $\mr M_i$ can
contain a non-typical root only for $i \ge j + 1$, which contradicts
the maximality of~$j$. Thus we have $\de_{j+n-2} = \al_1$, $\mr
M_{j+n-2} = \{2\al_1 + \ldots + 2\al_{n-1} + \al_n\}$, whence the
root $\de_{j+n-1}$ can be none of simple roots, which is impossible.

$3^\circ$. $\De$ is of type~$\ms F_4$. Then $\al = 2\al_2 + \al_3$
or $\al = 2\al_2 + \al_3 + \al_4$. In both cases the root $\al_3$ is
not active, whence $\de_{j+1} = \al_1$. Then $\de_{j+2} = \al_4$ and
in our sequence of elementary transformations, interchanging
successively neighboring elements, we can move the ${(j+2)}$nd
elementary transformation to the $j$th place. The new sequence
possesses the property that the set $\mr M_i$ can contain
non-typical roots only for $i \ge j+1$, a contradiction with the
maximality of~$j$.

Thus, in each of the three cases we have obtained a contradiction.
Hence for all $i = 1, \ldots, k-1$ the set $\mr M_i$ consists of
typical roots.

\emph{Step}~2. Let us prove that for all $i = 1, \ldots, k-1$ for
any two different roots in $\mr M_i$ one of possibilities $(\mr
D0)$, $(\mr E1')$, or $(\mr E2')$ is realized. For this it suffices
to prove that for all $i = 1, \ldots, k-1$ every active simple root
of $H_i$ is contained in the support of exactly one root in~$\mr
M_i$. Assume the converse: for some $j \in \{1,\ldots, k-1\}$ there
is a root $\de \in \Psi(H_j) \cap \Pi$ contained in the supports of
two different roots $\al,\be \in \mr M_j$, whereas for all $i < j$
there is no root with an analogous property. Without loss of
generality we may assume that $j$ is maximal among all shortest
sequences of elementary transformations taking $H$ to $H'$.

Let us show that $\de_j = \de$. Assume that $\de_j = \nu \ne \de$.
Then $r_\nu(\al), r_\nu(\be) \in \mr M_{j-1}$ by
Lemma~\ref{M_and_M'}(c). We have $\nu \ne \de$ and $\nu \ne \al -
\de$, therefore both roots $r_\nu(\de)$ and $r_\nu(\al - \de)$ are
positive, whence $r_\nu(\de) \in F(r_\nu(\al))$. Similarly,
$r_\nu(\de) \in F(r_\nu(\be))$. Then any root in the (nonempty) set
$F(r_\nu(\de)) \cap \Pi$ is contained in supports of the maximal
active roots $r_\nu(\al), r_\nu(\be)$ of~$H_{j-1}$, and we have a
contradiction with the choice of~$j$.

Thus $\de_j = \de$. Repeating the argument analogous to that in
\textit{Step}~1, we obtain that in the diagram $\Sigma(\Pi)$ the
nodes $\de_j$ and $\de_{j+1}$ are connected by an edge. Further we
consider two cases.

\emph{Case}~1. Every node of the diagram $\Sigma(\Pi)$ connected by
an edge with $\de$ is contained in the support of a root $\al \in
\mr M_j$ such that $\de \in F(\al)$. In this case in view of
Lemma~\ref{node_connected_with_active} we obtain a contradiction.

\emph{Case}~2. The diagram $\Sigma(\Pi)$ has the form shown on
Figure~\ref{diagram_difficult} (for some $p, q, r \ge 1$), $\de =
\ga_{r'}$, where $0 \le r' \le r - 1$, and the set $\mr M_j$
contains the roots $\al = \al_1 + \ldots + \al_{p'} + \ga_0 + \ga_1
+ \ldots + \ga_{r'}$ and $\be = \be_1 + \ldots + \be_{q'} + \ga_0 +
\ga_1 + \ldots + \ga_{r'}$, where $1 \le p' \le p$, $1 \le q' \le
q$. In view of Lemma~\ref{node_connected_with_active} we obtain that
$\de_{j+1} = \ga_{r'+1}$. Applying inductive argument analogous to
that in case~$2^\circ$ of \textit{Step}~1 we obtain that $\de_{j+i}
= \ga_{r'+i}$ for all $i = 1, \ldots, r - r'$. Further, the root
$\de_{j+r-r'+1}$ cannot lie in the set $\Supp \al \cup \Supp \be$,
therefore it is orthogonal to each of the roots $\ga_{r'},
\ga_{r'+1}, \ldots, \ga_r$. Hence in our sequence of elementary
transformations, interchanging successively neighboring elements, we
can move the $(j+r-r'+1)$st elementary transformation to the $j$th
place. As a result we obtain a new sequence such that for all $i \le
j$ every active simple root of $H_i$ is contained in the support of
exactly one root in~$\mr M_i$. We have obtained a contradiction with
the maximality of~$j$.

Thus in both cases we have come to a contradiction, which completes
the proof.
\end{proof}

\subsection{}

Let $H \subset G$ be a connected solvable spherical subgroup
standardly embedded in $B$ such that the set $\Upsilon(H)$ is
reduced. Regard an elementary transformation $H \mapsto H' = \si_\de
H \si_\de^{-1}$, where $\de \in \Psi(H) \cap \Pi$ is a regular
active simple root and $\si_\de \in N_G(T)$ is a representative of
the reflection~$r_\de$. In this subsection we find out conditions on
$\de$ under which the set $\Upsilon_0(H')$ is reduced and different
from the set~$\Upsilon_0(H)$.

We note that the set $\Upsilon(H')$ is reduced if and only if the
following two conditions are fulfilled:

(1) every root in $\mr M(H')$ is typical;

(2) for every two different roots $\al, \be \in \mr M(H')$ the set
$\Supp \al \cap \Supp \be$ contains no active simple roots of~$H'$.

To perform further considerations, we need the following notation:

$\mr M_0(H, \de)$ is the set of roots $\al \in \mr M(H) \backslash
\{\de\}$ such that in the diagram $\Sigma(\Pi)$ the node $\de$ is
joined by an edge with none of nodes in~$\Supp \al$;

$\mr M_1(H, \de)$ is the set of roots $\al \in \mr M(H) \backslash
\{\de\}$ with the following property: $\de \notin \Supp \al$ and
there is a (unique) root $\ga \in \Supp \al$ such that in the
diagram $\Sigma(\Pi)$ the nodes $\de$ and $\ga$ are joined by an
edge;

$\mr M_{11}(H, \de)$ is the set of roots $\al \in \mr M_1(H, \de)$
such that in the diagram $\Sigma(\Pi)$ the nodes $\de$ and $\ga$ are
joined by a triple edge with the arrow directed to~$\de$;

$\mr M_{12}(H, \de)$ is the set of roots $\al \in \mr M_1(H, \de)$
such that in the diagram $\Sigma(\Pi)$ the nodes $\de$ and $\ga$ are
joined by a double edge with the arrow directed to~$\de$;

$\mr M_{13}(H, \de) = \mr M_1(H, \de) \backslash (\mr M_{11}(H,\de)
\cup \mr M_{12}(H, \de))$;

$\mr M_2(H, \de)$ is the set of roots $\al \in \mr M(H) \backslash
\{\de\}$ such that $\de \in \Supp \al$;

$\mr M_{21}(H, \de)$ is the set of roots $\al \in \mr M_2(H, \de)$
with the following property: in the diagram $\Sigma(\Supp \al)$ the
node $\de$ is incident to a triple edge with the arrow directed
to~$\de$;

$\mr M_{22}(H, \de)$ is the set of roots $\al \in \mr M_2(H, \de)$
such that in the diagram $\Sigma(\Supp \al)$ the node $\de$ is
incident to a double edge with the arrow directed to~$\de$;

$\mr M_{23}(H, \de) = \mr M_2(H, \de) \backslash (\mr M_{21}(H,\de)
\cup \mr M_{22}(H, \de))$.

We note that the following disjoint unions take place: $\mr M(H) =
\mr M_0(H, \de) \cup \mr M_1(H, \de) \cup \mr M_2(H, \de)$, $\mr
M_1(H,\de) = \mr M_{11}(H, \de) \cup \mr M_{12}(H, \de) \cup \mr
M_{13}(H, \de)$, $\mr M_2(H, \de) = \mr M_{21}(H, \de) \cup \mr
M_{22}(H, \de) \cup \mr M_{23}(H, \de)$.

\begin{proposition}
\textup{(a)} The set $\Upsilon_0(H')$ coincides with $\Upsilon_0(H)$
if and only if $|\mr M_1(H, \de)| + |\mr M_{21}(H, \de)| + |\mr
M_{23}(H, \de)| = 0$;

\textup{(b)} the set $\Upsilon_0(H')$ contains a non-typical root if
and only if $|\mr M_{11}(H, \de)| + |\mr M_{12}(H, \de)| + |\mr
M_{21}(H, \de)| \ge 1$;

\textup{(c)} there is a root in $\Psi(H')$ contained in the supports
of at least two roots in $\mr M(H')$ if and only if $|\mr M_1(H,
\de)| + |\mr M_{22}(H, \de)| \ge 2$.
\end{proposition}

\begin{proof}
(a) The equality $|\mr M_1(H, \de)| + |\mr M_{21}(H, \de)| + |\mr
M_{23}(H, \de)| = 0$ is equivalent to the condition that the root
$\de$ is orthogonal to all roots in the set $\mr M(H) \backslash
\{\de\}$. If the latter condition is fulfilled, then, evidently,
$\Upsilon_0(H') = \Upsilon_0(H)$. Conversely, suppose that
$\Upsilon_0(H') = \Upsilon_0(H)$ and there exist a root $\al \in \mr
M(H) \backslash \{\de\}$ that is not orthogonal to~$\de$. Then we
have $\al, r_\de(\al) \in \mr M(H)$, whence the support of one of
these roots is contained in the support of the other, which is
impossible by Corollary~\ref{subfamilies}(c).

(b) The desired result is obtained by a direct check.

(c) In view of Lemma~\ref{orthogonal} the proof reduces to a
consideration of roots in the sets $\mr M_1(H, \de)$ and $\mr M_2(H,
\de)$. The remaining part of the proof is obtained by a direct
check.
\end{proof}

\begin{corollary} \label{reduced_set}
The set $\Upsilon_0(H')$ is reduced and different from the set
$\Upsilon_0(H)$ if and only if the following conditions are
satisfied:

\textup{(1)} $|\mr M_{13}(H, \de)| + |\mr M_{23}(H, \de)| \ge 1$;

\textup{(2)} $|\mr M_{11}(H, \de)| + |\mr M_{12}(H, \de)| + |\mr
M_{21}(H, \de)| = 0$;

\textup{(3)} $|\mr M_{13}(H, \de)| + |\mr M_{22}(H, \de)| \le 1$.
\end{corollary}

\section{Examples}\label{examples}

In this section, as an illustration of the theory developed in the
present paper, for all simple groups $G$ of rank at most $4$ we
list, up to conjugation, all saturated solvable spherical subgroups
in~$G$, that is, solvable spherical subgroups $H \subset G$ such
that ${H = N_G(H)^0}$ (see \S\,\ref{sober_subgroups}). Thereby we
also list, up to conjugation, all unipotent subgroups in $G$ that
are unipotent radicals of connected solvable spherical subgroups
in~$G$ (see Corollary~\ref{sober_bij_ur} and
Proposition~\ref{sober}).

For a fixed group $G$ we proceed in two steps.

\emph{Step}~1. We list all reduced sets $(\mr M, \pi, \sim)$ (where
$\mr M \subset \De_+$ is a subset, $\pi: \mr M \to \Pi$ is a map,
$\sim$ is an equivalence relation on~$\mr M$). We recall that
reduced sets are characterized by satisfying conditions $(\mr A')$,
$(\mr D')$, $(\mr E')$, and $(\mr C)$ (see Proposition~\ref{sober}
and Theorem~\ref{simplified}).

\emph{Step}~2. For each set $(\mr M, \pi, \sim)$ in the list
obtained at \emph{Step}~1, we indicate all other sets in this list
that are obtained from the set $(\mr M, \pi, \sim)$ by applying
finite sequences of elementary transformations (see
Theorem~\ref{simplified_conjugate}, as well as
Propositions~\ref{sim_act_reg},~\ref{comb_datas} and
Corollary~\ref{reduced_set}).

We note that the procedure described at \emph{Step}~1 depends only
on the \emph{underlying graph} of the diagram $\Sigma(\Pi)$, that
is, the graph obtained from $\Sigma(\Pi)$ by replacing each multiple
edge by a non-oriented simple edge. In this connection, for all
groups $G$ with the same underlying graph of the diagram
$\Sigma(\Pi)$ it is convenient to perform \emph{Step}~1
simultaneously.

The procedure of enumeration of all reduced sets for a given group
$G$ can be shortened as follows. Let ${(\mr M, \pi, \sim)}$ be a
reduced set. Put $\Supp \mr M = \bigcup \limits_{\de \in \mr M}
\Supp \de$. Note that the set $\Supp \mr M$ is invariant under an
elementary transformation. Besides, the procedure of enumeration of
all reduced sets ${(\mr M, \pi, \sim)}$ with a given set $\Supp \mr
M$ depends only on the diagram $\Sigma(\Supp \mr M)$. In this
connection, for a given group $G$ the initial problem reduces to the
following one: for every subset $\Pi' \subset \Pi$ perform
\emph{Steps}~1,2, at that, at \emph{Step}~1 list only reduced sets
$(\mr M, \pi, \sim)$ with $\Supp \mr M = \Pi'$. The advantage of
this approach is that for any two subsets $\Pi', \Pi'' \subset \Pi$
with $\Sigma(\Pi') \simeq \Sigma(\Pi'')$ the procedure of
enumeration of all reduced sets $(\mr M, \pi, \sim)$ with $\Supp \mr
M = \Pi'$ and $\Supp \mr M = \Pi''$ can be performed simultaneously.

Let $(\mr M, \pi, \sim)$ be a reduced set. Regard the corresponding
saturated solvable spherical subgroup $H \subset G$ standardly
embedded in~$B$. Denote by $c(S)$ the codimension in $T$ of the
maximal torus $S = H \cap T$ of $H$ and by $c(N)$ the codimension in
$U$ of the unipotent radical $N = H \cap U$ of $H$. We have $c(S) =
|\mr M| - \mu$, where $\mu$ is the number of equivalence classes in
the set $\mr M$ and $c(N)$ equals the number of equivalence classes
in the set $\Psi$, see \S\,\ref{existence_proof_Psi}. Note that the
equality $c(S) + c(N) = |\Supp \mr M|$ takes place, whence in the
case $\Supp \mr M = \Pi$ we get $c(S) + c(N) = \rk G$. In
particular, in the latter case we have $\dim H = \dim U$, whence $H$
is a spherical subgroup in $G$ of minimal dimension.

We denote by $d_0(G)$ the number of conjugacy classes of saturated
solvable spherical subgroups in $G$ corresponding to reduced sets
${(\mr M, \pi, \sim)}$ with $\Supp \mr M = \Pi$.

Now, for all groups $G$ such that the diagrams $\Sigma(\Pi)$ are
subdiagrams of connected Dynkin diagrams of rank at most~$4$, we
list all reduced sets ${(\mr M, \pi, \sim)}$ with $\Supp \mr M =
\Pi$. At that, we formally include into consideration the case $\rk
G = 0$ corresponding to the situation when $\mr M = \varnothing$. In
each case we also indicate the values of $c(S)$ and~$c(N)$. Besides,
we indicate the values of $d_0(G)$ for all groups $G$ under
consideration.

$\rk G = 0$. We have $\mr M = \varnothing$, $c(S) = c(N) = 0$,
$d_0(\varnothing) = 1$.

$\rk G = 1$. We have $\mr M = \De_+$, $c(S) = 0$, $c(N) = 1$,
$d_0(\ms A_1) = 1$.

$\rk G = 2$, underlying graph of type $\ms A_1 \times \ms A_1$. We
have $\mr M = \Pi = \{\al_1, \al_2\}$ and either $\al_1 \not\sim
\al_2$ (then $c(S) = 0$, $c(N)=2 $), or $\al_1 \sim \al_2$ (then
$c(S) = c(N) = 1$); $d_0(\ms A_1 \times \ms A_1) = 2$.

$\rk G = 2$, underlying graph of type~$\ms A_2$. The results are
presented in Table~\ref{table_rank_2}.

$\rk G = 3$, underlying graph of type $\ms A_1 \times \ms A_1 \times
\ms A_1$. The results are presented in Table~\ref{table_A1A1A1}.

$\rk G = 3$, underlying graph of type $\ms A_1 \times \ms A_2$. The
results are presented in Table~\ref{table_A1A2}.

$\rk G = 3$, underlying graph of type~$\ms A_3$. The results are
presented in Table~\ref{table_rank_3}.

$\rk G = 4$, underlying graph of type~$\ms A_4$. The results are
presented in Table~\ref{table_rank_4A}.

$\rk G = 4$, underlying graph of type~$\ms D_4$. The results are
presented in Table~\ref{table_rank_4D}.

Having known the values $d_0(G)$ for all groups $G$ considered
above, for every simple group $G$ of rank at most~$4$ we now can
determine the value $d(G)$, which is equal to the number of
conjugacy classes in $G$ of saturated solvable spherical subgroups
as well as to the number of conjugacy classes in $G$ of unipotent
radicals of connected solvable spherical subgroups. The
corresponding computations and
Table~\ref{table_number_of_solvable_subgroups} summarizing them are
given below.

$d(\ms A_1) = d_0(\varnothing) + d_0(\ms A_1) = 1 + 1 = 2$.

$d(\ms A_2) = d_0(\varnothing) + 2d_0(\ms A_1) + d_0(\ms A_2) = 1 +
2 \cdot 1 + 2 = 5$.

$d(\ms B_2) =  d_0(\varnothing) + 2d_0(\ms A_1) + d_0(\ms B_2) = 1 +
2 \cdot 1 + 3 = 6$.

$d(\ms G_2) =  d_0(\varnothing) + 2d_0(\ms A_1) + d_0(\ms G_2) = 1 +
2 \cdot 1 + 3 = 6$.

$d(\ms A_3) = d_0(\varnothing) + 3d_0(\ms A_1) + d_0(\ms A_1 \times
\ms A_1) + 2d_0(\ms A_2) + d_0(\ms A_3) = 1 + 3\cdot1 + 2 + 2\cdot2
+8 = 18$.

$d(\ms B_3) = d_0(\varnothing) + 3d_0(\ms A_1) + d_0(\ms A_1 \times
\ms A_1) + d_0(\ms A_2) + d_0(\ms B_2) + d_0(\ms B_3) = 1 + 3\cdot1
+ 2 + 2 + 3 + 11 = 22$.

$d(\ms C_3) = d_0(\varnothing) + 3d_0(\ms A_1) + d_0(\ms A_1 \times
\ms A_1) + d_0(\ms A_2) + d_0(\ms B_2) + d_0(\ms C_3) = 1 + 3\cdot1
+ 2 + 2 + 3 + 10 = 21$.

$d(\ms A_4) = d_0(\varnothing) + 4d_0(\ms A_1) + 3d_0(\ms A_1 \times
\ms A_1) + 3d_0(\ms A_2) + 2d_0(\ms A_1 \times \ms A_2) + 2d_0(\ms
A_3) + d_0(\ms A_4) = 1 + 4\cdot1 + 3\cdot2 + 3\cdot2 + 2\cdot5 +
2\cdot8 + 31 = 74$.

$d(\ms B_4) = d_0(\varnothing) + 4d_0(\ms A_1) + 3d_0(\ms A_1 \times
\ms A_1) + 2d_0(\ms A_2) + d_0(\ms B_2) + d_0(\ms A_1 \times \ms
A_2) + d_0({\ms A_1 \times \ms B_2}) + d_0(\ms A_3) + d_0(\ms B_3) +
d_0(\ms B_4) = 1 + 4\cdot1 + 3\cdot2 + 2\cdot2 + 3 + 5 + 7 + 8 + 11
+ 42 = 91$.

$d(\ms C_4) = d_0(\varnothing) + 4d_0(\ms A_1) + 3d_0(\ms A_1 \times
\ms A_1) + 2d_0(\ms A_2) + d_0(\ms B_2) + d_0(\ms A_1 \times \ms
A_2) + d_0({\ms A_1 \times \ms B_2}) + d_0(\ms A_3) + d_0(\ms C_3) +
d_0(\ms C_4) = 1 + 4\cdot1 + 3\cdot2 + 2\cdot2 + 3 + 5 + 7 + 8 + 10
+ 38 = 86$.

$d(\ms D_4) = d_0(\varnothing) + 4d_0(\ms A_1) + 3d_0(\ms A_1 \times
\ms A_1) + 3d_0(\ms A_2) + d_0(\ms A_1 \times \ms A_1 \times \ms
A_1) + 3d_0(\ms A_3) + d_0(\ms D_4) = 1 + 4\cdot1 + 3\cdot2 +
3\cdot2 + 5 + 3\cdot8 + 40 = 86$.

$d(\ms F_4) = d_0(\varnothing) + 4d_0(\ms A_1) + 3d_0(\ms A_1 \times
\ms A_1) + 2d_0(\ms A_2) + d_0(\ms B_2) + 2d_0(\ms A_1 \times \ms
A_2) + d_0(\ms B_3) + d_0(\ms C_3) + d_0(\ms F_4) = 1 + 4\cdot1 +
3\cdot2 + 2\cdot2 + 3 + 2\cdot5 + 11 + 10 + 38 = 87$.

\begin{table}[h]

\caption{}\label{table_number_of_solvable_subgroups}

\begin{center}

\begin{tabular}{|c|c|c|c|c|c|c|c|c|c|c|c|c|}

\hline

Type of $G$ & $\ms A_1$ & $\ms A_2$ & $\ms B_2$ & $\ms G_2$ & $\ms
A_3$ & $\ms B_3$ & $\ms C_3$ & $\ms A_4$ & $\ms B_4$ & $\ms C_4$ &
$\ms D_4$ &
$\ms F_4$ \\

\hline

$d(G)$ & $2$ & $5$ & $6$ & $6$ & $18$ & $22$ & $21$ & $74$ & $91$ &
$86$ &
$86$ & $87$\\

\hline

\end{tabular}

\end{center}

\end{table}

We now describe the notation used in
Tables~\ref{table_rank_2}--\ref{table_rank_4D}. In case $\rk G = n$
($2 \le n \le 4$) we suppose that $\Pi = \{\al_1, \al_2, \ldots,
\al_n\}$. In the column `$(\mr M, \pi)$' we indicate all pairs
$(\al, \pi(\al))$, where $\al \in \mr M$. At that, $(i_1i_2\ldots
i_k, i_j)$ denotes the pair $(\al_{i_1} + \al_{i_2} + \ldots +
\al_{i_k}, \al_{i_j})$. In the column `$\sim$' we indicate
equivalence classes in $\mr M$ containing more than one element. At
that, $i_1\ldots i_k{\sim}j_1\ldots j_l$ denotes the relation
$\al_{i_1} + \ldots + \al_{i_k} \sim \al_{j_1} + \ldots +
\al_{j_l}$. In each of the columns headed by type of $G$, we
indicate all reduced sets $(\mr M, \pi, \sim)$ that are obtained
from a given set $(\mr M, \pi, \sim)$ by applying finite sequences
of elementary transformations in~$G$. Namely, in this column $m(i)$
(resp. $m$) denotes that the given set $(\mr M, \pi, \sim)$ is
transformed to the set $(\mr M, \pi, \sim)$ in row $m$ of the same
table under the elementary transformation with center~$\al_i$ (resp.
under an appropriate sequence of elementary transformations of
length more than~$1$). The columns `$c(S)$' and `$c(N)$' contain the
corresponding values. In the last row of each table we indicate the
values of $d_0(G)$ for each $G$ considered in this table.

\renewcommand{\arraystretch}{1.1}
\renewcommand{\tabcolsep}{0.5pt}

\begin{table}[b]

\begin{center}
\caption{} \label{table_rank_2}

\begin{tabular}{|c|c|c|c|c|c|c|c|}
\hline

No. & $(\mr M, \pi)$ & $\sim$ & $\ms A_2$ & $\ms B_2$ & $\ms G_2$ & $c(S)$ & $c(N)$ \\

\hline

1 & $(12,1)$ & & $2,3(2)$ & & $3(2)$ & $0$ & $2$ \\

\hline

2 & $(12,2)$ & & $1,3(1)$ & $3(1)$ & & $0$ & $2$ \\

\hline

3 & $(1,1), (2,2)$ & & $1(2),2(1)$ & $2(1)$ & $1(2)$ & $0$ & $2$  \\

\hline

4 & $(1,1),(2,2)$ & $1{\sim}2$ & & & & $1$ & $1$ \\

\hline

\multicolumn{3}{|c|}{$d_0(G)$} & 2 & 3 & 3 & \multicolumn{2}{|c}{}\\

\cline{1-6}

\end{tabular}

\end{center}

\end{table}

\begin{table}[p]

\begin{center}
\caption{} \label{table_A1A1A1}

{\small

\begin{tabular}{|c|c|c|c|c|c|}
\hline

No. & $(\mr M, \pi)$ & ${\sim}$ & $\ms A_1 \times \ms A_1 \times \ms A_1$ & $c(S)$ & $c(N)$ \\

\hline

1 & $(1,1),(2,2),(3,3)$ & & & $0$ & $3$ \\

\hline

2 & $(1,1),(2,2),(3,3)$ & $1{\sim}2$ & & $1$ & $2$ \\

\hline

3 & $(1,1),(2,2),(3,3)$ & $1{\sim}3$ & & $1$ & $2$ \\

\hline

4 & $(1,1),(2,2),(3,3)$ & $2{\sim}3$ & & $1$ & $2$ \\

\hline

5 & $(1,1),(2,2),(3,3)$ & $1{\sim}2{\sim}3$ & & $2$ & $1$ \\

\hline

\multicolumn{3}{|c|}{$d_0(G)$} & 5 & \multicolumn{2}{|c}{}\\

\cline{1-4}

\end{tabular}

}

\end{center}

\end{table}

\begin{table}[p]

\begin{center}
\caption{} \label{table_A1A2}

{\small

\begin{tabular}{|c|c|c|c|c|c|c|}
\hline

No. & $(\mr M, \pi)$ & ${\sim}$ & $\ms A_1 \times \ms A_2$ & $\ms A_1 \times \ms B_2$ & $c(S)$ & $c(N)$ \\

\hline

1 & $(1,1),(23,2)$ & & $3,5(3)$ &  & $0$ & $3$ \\

\hline

2 & $(1,1),(23,2)$ & $1{\sim}23$ & $6(3)$ & & $1$ & $2$ \\

\hline

3 & $(1,1),(23,3)$ & & $1,5(2)$ & $5(2)$ & $0$ & $3$ \\

\hline

4 & $(1,1),(23,3)$ & $1{\sim}23$ & $7(2)$ & $7(2)$ & $1$ & $2$ \\

\hline

5 & $(1,1),(2,2),(3,3)$ & & $1(3),3(2)$ & $3(2)$ & $0$ & $3$ \\

\hline

6 & $(1,1),(2,2),(3,3)$ & $1{\sim}2$ & $2(3)$ & & $1$ & $2$ \\

\hline

7 & $(1,1),(2,2),(3,3)$ & $1{\sim}3$ & $4(2)$ & $4(2)$ & $1$ & $2$ \\

\hline

8 & $(1,1),(2,2),(3,3)$ & $2{\sim}3$ & & & $1$ & $2$ \\

\hline

9 & $(1,1),(2,2),(3,3)$ & $1{\sim}2{\sim}3$ & & & $2$ & $1$ \\

\hline

\multicolumn{3}{|c|}{$d_0(G)$} & 5 & 7 & \multicolumn{2}{|c}{}\\

\cline{1-5}

\end{tabular}

}

\end{center}

\end{table}

\begin{table}[p]

\begin{center}
\caption{} \label{table_rank_3}

{\small

\begin{tabular}{|c|c|c|c|c|c|c|c|}
\hline

No. & $(\mr M, \pi)$ & ${\sim}$ & $\ms A_3$ & $\ms B_3$ & $\ms C_3$ & $c(S)$ & $c(N)$ \\

\hline

1 & $(123,1)$ & & $3,6,8(3)$ & & $8(3)$ & $0$ & $3$ \\

\hline

2 & $(123,2)$ & & $4(1),10(3),13$ & $4(1)$ & $4(1),10(3),13$ & $0$ & $3$ \\

\hline

3 & $(123,3)$ & & $1,6(1),8$ & $6(1),8$ & $6(1)$ & $0$ & $3$ \\

\hline

4 & $(1,1),(23,2)$ & & $2(1),10,13(3)$ & $2(1)$ & $2(1),10,13(3)$ & $0$ & $3$ \\

\hline

5 & $(1,1),(23,2)$ & $1{\sim}23$ & $14(3)$ & & $14(3)$ & $1$ & $2$ \\

\hline

6 & $(1,1),(23,3)$ & & $1,3(1),8(2)$ & $3(1),8(2)$ & $3(1)$ & $0$ & $3$ \\

\hline

7 & $(1,1),(23,3)$ & $1{\sim}23$ & $9(2)$ & $9(2)$ & & $1$ & $2$ \\

\hline

8 & $(12,1),(3,3)$ & & $1(3),3,6(2)$ & $3,6(2)$ & $1(3)$ & $0$ & $3$ \\

\hline

9 & $(12,1),(3,3)$ & $12{\sim}3$ & $7(2)$ & $7(2)$ & & $1$ & $2$ \\

\hline

10 & $(12,2),(3,3)$ & & $2(3),4,13(1)$ & $13(1)$ & $2(3),4,13(1)$ & $0$ & $3$ \\

\hline

11 & $(12,2),(3,3)$ & $12{\sim}3$ & $16(1)$ & $16(1)$ & $16(1)$ & $1$ & $2$ \\

\hline

12 & $(12,2),(23,2)$ & $12{\sim}23$ & & & & $1$ & $2$ \\

\hline

13 & $(1,1),(2,2),(3,3)$ & & $2,4(3),10(1)$ & $10(1)$ & $2,4(3),10(1)$ & $0$ & $3$ \\

\hline

14 & $(1,1),(2,2),(3,3)$ & $1{\sim}2$ & $5(3)$ & & $5(3)$ & $1$ & $2$ \\

\hline

15 & $(1,1),(2,2),(3,3)$ & $1{\sim}3$ & & & & $1$ & $2$ \\

\hline

16 & $(1,1),(2,2),(3,3)$ & $2{\sim}3$ & $11(1)$ & $11(1)$ & $11(1)$ & $1$ & $2$ \\

\hline

17 & $(1,1),(2,2),(3,3)$ & $1{\sim}2{\sim}3$ & & & & $2$ & $1$ \\

\hline

\multicolumn{3}{|c|}{$d_0(G)$} & 8 & 11 & 10 & \multicolumn{2}{|c}{}\\

\cline{1-6}

\end{tabular}

}

\end{center}

\end{table}

\newpage

{\footnotesize

\begin{longtable}{|c|c|c|c|c|c|c|c|c|}

\caption{} \label{table_rank_4A} \\

\hline

No. & $(\mr M, \pi)$ & ${\sim}$ & $\ms A_4$ & $\ms B_4$ & $\ms C_4$ & $\ms F_4$ & $c(S)$ & $c(N)$ \\

\hline \endfirsthead

\caption{(continuation)} \\

\hline

No. & $(\mr M, \pi)$ & ${\sim}$ & $\ms A_4$ & $\ms B_4$ & $\ms C_4$ & $\ms F_4$ & $c(S)$ & $c(N)$ \\

\hline \endhead

1 & $(1234,1)$ &  & $4,9,11(4),19$ &  & $11(4)$ & $11(4),19$ & $0$ & $4$ \\

\hline

2 & $(1234,2)$ &  & \begin{tabular}{c}$5(1),13(4),$\\$23,37,52$\end{tabular} & $5(1)$ & $5(1),13(4),37$ & \begin{tabular}{c}$5(1),13(4),$\\$23,37,52$\end{tabular} & $0$ & $4$ \\

\hline

3 & $(1234,3)$ &  & \begin{tabular}{c}$7(1),15(4),$\\$17,27,42$\end{tabular} & $7(1),17$ & \begin{tabular}{c}$7(1),15(4),$\\$17,27,42$\end{tabular} & $7(1),15(4),42$ & $0$ & $4$ \\

\hline

4 & $(1234,4)$ &  & $1,9(1),11,19$ & $9(1),11,19$ & $9(1),19$ & $9(1)$ & $0$ & $4$ \\

\hline

5 & $(1,1),(234,2)$ &  & \begin{tabular}{c}$2(1),13,23,$\\$37(4),52$\end{tabular} & $2(1)$ & $2(1),13,37(4)$ & \begin{tabular}{c}$2(1),13,23,$\\$37(4),42$\end{tabular} & $0$ & $4$ \\

\hline

6 & $(1,1),(234,2)$ & $1{\sim}234$ & $38(4),53$ &  & $38(4)$ & $38(4),53$ & $1$ & $3$ \\

\hline

7 & $(1,1),(234,3)$ &  & \begin{tabular}{c}$3(1),15,17(2),$\\$27,42(4)$\end{tabular} & $3(1),17(2)$ & \begin{tabular}{c}$3(1),15,17(2),$\\$27,42(4)$\end{tabular} & $3(1),15,42(4)$ & $0$ & $4$ \\

\hline

8 & $(1,1),(234,3)$ & $1{\sim}234$ & $18(2),28,43(4)$ & $18(2)$ & $18(2),28,43(4)$ & $43(4)$ & $1$ & $3$ \\

\hline

9 & $(1,1),(234,4)$ &  & $1,4(1),11,19(2)$ & $4(1),11,19(2)$ & $4(1),19(2)$ & $4(1)$ & $0$ & $4$ \\

\hline

10 & $(1,1),(234,4)$ & $1{\sim}234$ & $12,20(2)$ & $12,20(2)$ & $20(2)$ &  & $1$ & $3$ \\

\hline

11 & $(123,1),(4,4)$ &  & $1(4),4,9,19(3)$ & $4,9,19(3)$ & $1(4)$ & $1(4),19(3)$ & $0$ & $4$ \\

\hline

12 & $(123,1),(4,4)$ & $123{\sim}4$ & $10,20(3)$ & $10,20(3)$ &  & $20(3)$ & $1$ & $3$ \\

\hline

13 & $(123,2),(4,4)$ &  & \begin{tabular}{c}$2(4),5,23(3),$\\$37(1),52$\end{tabular} & $23(3),37(1),52$ & $2(4),5,37(1)$ & \begin{tabular}{c}$2(4),5,23(3),$\\$37(1),52$\end{tabular} & $0$ & $4$ \\

\hline

14 & $(123,2),(4,4)$ & $123{\sim}4$ & $24(3),40(1),55$ & $24(3),40(1),55$ & $40(1)$ & $24(3),40(1),55$ & $1$ & $3$ \\

\hline

15 & $(123,3),(4,4)$ &  & \begin{tabular}{c}$3(4),7,17,$\\$27,42(1)$\end{tabular} & $27,42(1)$ & \begin{tabular}{c}$3(4),7,17,$\\$27,42(1)$\end{tabular} & $3(4),7,42(1)$ & $0$ & $4$ \\

\hline

16 & $(123,3),(4,4)$ & $123{\sim}4$ & $30,45(1)$ & $30,45(1)$ & $30,45(1)$ & $45(1)$ & $1$ & $3$ \\

\hline

17 & $(12,1),(34,3)$ &  & \begin{tabular}{c}$3,7(2),15,$\\$27(4),42$\end{tabular} & $3,7(2)$ & \begin{tabular}{c}$3,7(2),15,$\\$27(4),42$\end{tabular} & $27(4)$ & $0$ & $4$ \\

\hline

18 & $(12,1),(34,3)$ & $12{\sim}34$ & $8(2),28(4),43$ & $8(2)$ & $8(2),28(4),43$ & $28(4)$ & $1$ & $3$ \\

\hline

19 & $(12,1),(34,4)$ &  & $1,4,9(2),11(3)$ & $4,9(2),11(3)$ & $4,9(2)$ & $1,11(3)$ & $0$ & $4$ \\

\hline

20 & $(12,1),(34,4)$ & $12{\sim}34$ & $10(2),12(3)$ & $10(2),12(3)$ & $10(2)$ & $12(3)$ & $1$ & $3$ \\

\hline

21 & $(12,2),(34,3)$ &  & $32(4),47(1),61$ & $47(1)$ & $32(4),47(1),61$ & $32(4),47(1),61$ & $0$ & $4$ \\

\hline

22 & $(12,2),(34,3)$ & $12{\sim}34$ & $33(4),50(1),65$ & $50(1)$ & $33(4),50(1),65$ & $33(4),50(1),65$ & $1$ & $3$ \\

\hline

23 & $(12,2),(34,4)$ &  & \begin{tabular}{c}$2,5,13(3),$\\$37,52(1)$\end{tabular} & $13(3),37,52(1)$ & $52(1)$ & \begin{tabular}{c}$2,5,13(3),$\\$37,52(1)$\end{tabular} & $0$ & $4$ \\

\hline

24 & $(12,2),(34,4)$ & $12{\sim}34$ & $14(3),40,55(1)$ & $14(3),40,55(1)$ & $55(1)$ & $14(3),40,55(1)$ & $1$ & $3$ \\

\hline

25 & $(12,2),(234,2)$ & $12{\sim}234$ & $57(4)$ &  & $57(4)$ & $57(4)$ & $1$ & $3$ \\

\hline

26 & $(123,3),(34,3)$ & $123{\sim}34$ & $59(1)$ & $59(1)$ & $59(1)$ & $59(1)$ & $1$ & $3$ \\

\hline

27 & $(12,1),(3,3),(4,4)$ &  & \begin{tabular}{c}$3,7,15,$\\$17(4),42(2)$\end{tabular} & $15,42(2)$ & \begin{tabular}{c}$3,7,15,$\\$17(4),42(2)$\end{tabular} & $17(4)$ & $0$ & $4$ \\

\hline

28 & $(12,1),(3,3),(4,4)$ & $12{\sim}3$ & $8,18(4),43(2)$ & $43(2)$ & $8,18(4),43(2)$ & $18(4)$ & $1$ & $3$ \\

\hline

29 & $(12,1),(3,3),(4,4)$ & $12{\sim}4$ & $44(2)$ & $44(2)$ & $44(2)$ &  & $1$ & $3$ \\

\hline

30 & $(12,1),(3,3),(4,4)$ & $3{\sim}4$ & $16,45(2)$ & $16,45(2)$ & $16,45(2)$ &  & $1$ & $3$ \\

\hline

31 & $(12,1),(3,3),(4,4)$ & $12{\sim}3{\sim}4$ & $46(2)$ & $46(2)$ & $46(2)$ &  & $2$ & $2$ \\

\hline

32 & $(12,2),(3,3),(4,4)$ &  & $21(4),47,61(1)$ & $61(1)$ & $21(4),47,61(1)$ & $21(4),47,61(1)$ & $0$ & $4$ \\

\hline

33 & $(12,2),(3,3),(4,4)$ & $12{\sim}3$ & $22(4),50,65(1)$ & $65(1)$ & $22(4),50,65(1)$ & $22(4),50,65(1)$ & $1$ & $3$ \\

\hline

34 & $(12,2),(3,3),(4,4)$ & $12{\sim}4$ & $66(1)$ & $66(1)$ & $66(1)$ & $66(1)$ & $1$ & $3$ \\

\hline

35 & $(12,2),(3,3),(4,4)$ & $3{\sim}4$ & $67(1)$ & $67(1)$ & $67(1)$ & $67(1)$ & $1$ & $3$ \\

\hline

36 & $(12,2),(3,3),(4,4)$ & $12{\sim}3{\sim}4$ & $71(1)$ & $71(1)$ & $71(1)$ & $71(1)$ & $2$ & $2$ \\

\hline

37 & $(1,1),(23,2),(4,4)$ &  & \begin{tabular}{c}$2,5(4),13(1),$\\$23,52(3)$\end{tabular} & $13(1),23,52(3)$ & $2,5(4),13(1)$ & \begin{tabular}{c}$2,5(4),13(1),$\\$23,52(3)$\end{tabular} & $0$ & $4$ \\

\hline

38 & $(1,1),(23,2),(4,4)$ & $1{\sim}23$ & $6(4),53(3)$ &  $53(3)$ &  $6(4)$ &  $6(4),53(3)$ & $1$ & $3$ \\

\hline

39 & $(1,1),(23,2),(4,4)$ & $1{\sim}4$ & $54(3)$ & $54(3)$ &  & $54(3)$ & $1$ & $3$ \\

\hline

40 & $(1,1),(23,2),(4,4)$ & $23{\sim}4$ & $14(1),24,55(3)$ & $14(1),24,55(3)$ & $14(1)$ & $14(1),24,55(3)$ & $1$ & $3$ \\

\hline

41 & $(1,1),(23,2),(4,4)$ & $1{\sim}23{\sim}4$ & $56(3)$ & $56(3)$ &  & $56(3)$ & $2$ & $2$ \\

\hline

42 & $(1,1),(23,3),(4,4)$ &  & \begin{tabular}{c}$3,7(4),15(1),$\\$17,27(2)$\end{tabular} & $15(1),27(2)$ & \begin{tabular}{c}$3,7(4),15(1),$\\$17,27(2)$\end{tabular} & $3,7(4),15(1)$ & $0$ & $4$ \\

\hline

43 & $(1,1),(23,3),(4,4)$ & $1{\sim}23$ & $8(4),18,28(2)$ & $28(2)$ & $8(4),18,28(2)$ & $8(4)$ & $1$ & $3$ \\

\hline

44 & $(1,1),(23,3),(4,4)$ & $1{\sim}4$ & $29(2)$ & $29(2)$ & $29(2)$ &  & $1$ & $3$ \\

\hline

45 & $(1,1),(23,3),(4,4)$ & $23{\sim}4$ & $16(1),30(2)$ & $16(1),30(2)$ & $16(1),30(2)$ & $16(1)$ & $1$ & $3$ \\

\hline

46 & $(1,1),(23,3),(4,4)$ & $1{\sim}23{\sim}4$ & $31(2)$ & $31(2)$ & $31(2)$ &  & $2$ & $2$ \\

\hline

47 & $(1,1),(2,2),(34,3)$ &  & $21(1),32,61(4)$ & $21(1)$ & $21(1),32,61(4)$ & $21(1),32,61(4)$ & $0$ & $4$ \\

\hline

48 & $(1,1),(2,2),(34,3)$ & $1{\sim}2$ & $62(4)$ &  & $62(4)$ & $62(4)$ & $1$ & $3$ \\

\hline

49 & $(1,1),(2,2),(34,3)$ & $1{\sim}34$ & $63(4)$ &  & $63(4)$ & $63(4)$ & $1$ & $3$ \\

\hline

50 & $(1,1),(2,2),(34,3)$ & $2{\sim}34$ & $22(1),33,65(4)$ & $22(1)$ & $22(1),33,65(4)$ & $22(1),33,65(4)$ & $1$ & $3$ \\

\hline

51 & $(1,1),(2,2),(34,3)$ & $1{\sim}2{\sim}34$ & $68(4)$ &  & $68(4)$ & $68(4)$ & $2$ & $2$ \\

\hline

52 & $(1,1),(2,2),(34,4)$ &  & \begin{tabular}{c}$2,5,13,$\\$23(1),37(3)$\end{tabular} & $13,23(1),37(3)$ & $23(1)$ & \begin{tabular}{c}$2,5,13,$\\$23(1),37(3)$\end{tabular} & $0$ & $4$ \\

\hline

53 & $(1,1),(2,2),(34,4)$ & $1{\sim}2$ & $6,38(3)$ & $38(3)$ &  & $6,38(3)$ & $1$ & $3$ \\

\hline

54 & $(1,1),(2,2),(34,4)$ & $1{\sim}34$ & $39(3)$ & $39(3)$ &  & $39(3)$ & $1$ & $3$ \\

\hline

55 & $(1,1),(2,2),(34,4)$ & $2{\sim}34$ & $14,24(1),40(3)$ & $14,24(1),40(3)$ & $24(1)$ & $14,24(1),40(3)$ & $1$ & $3$ \\

\hline

56 & $(1,1),(2,2),(34,4)$ & $1{\sim}2{\sim}34$ & $41(3)$ & $41(3)$ &  & $41(3)$ & $2$ & $2$ \\

\hline

57 & $(12,2),(23,2),(4,4)$ & $12{\sim}23$ & $25(4)$ &  & $25(4)$ & $25(4)$ & $1$ & $3$ \\

\hline

58 & $(12,2),(23,2),(4,4)$ & $12{\sim}23{\sim}4$ &  &  &  &  & $2$ & $2$ \\

\hline

59 & $(1,1),(23,3),(34,3)$ & $23{\sim}34$ & $26(1)$ & $26(1)$ & $26(1)$ & $26(1)$ & $1$ & $3$ \\

\hline

60 & $(1,1),(23,3),(34,3)$ & $1{\sim}23{\sim}34$ &  &  &  &  & $2$ & $2$ \\

\hline

61 & $(1,1),(2,2),(3,3),(4,4)$ & & $21,32(1),47(4)$ & $32(1)$ & $21,32(1),47(4)$ & $21,32(1),47(4)$ & $0$ & $4$ \\

\hline

62 & $(1,1),(2,2),(3,3),(4,4)$ & $1{\sim}2$ & $48(4)$ &  & $48(4)$ & $48(4)$ & $1$ & $3$ \\

\hline

63 & $(1,1),(2,2),(3,3),(4,4)$ & $1{\sim}3$ & $49(4)$ &  & $49(4)$ & $49(4)$ & $1$ & $3$ \\

\hline

64 & $(1,1),(2,2),(3,3),(4,4)$ & $1{\sim}4$ &  &  &  &  & $1$ & $3$ \\

\hline

65 & $(1,1),(2,2),(3,3),(4,4)$ & $2{\sim}3$ & $22,33(1),50(4)$ & $33(1)$ & $22,33(1),50(4)$ & $22,33(1),50(4)$ & $1$ & $3$ \\

\hline

66 & $(1,1),(2,2),(3,3),(4,4)$ & $2{\sim}4$ & $34(1)$ & $34(1)$ & $34(1)$ & $34(1)$ & $1$ & $3$ \\

\hline

67 & $(1,1),(2,2),(3,3),(4,4)$ & $3{\sim}4$ & $35(1)$ & $35(1)$ & $35(1)$ & $35(1)$ & $1$ & $3$ \\

\hline

68 & $(1,1),(2,2),(3,3),(4,4)$ & $1{\sim}2{\sim}3$ & $51(4)$ &  & $51(4)$ & $51(4)$ & $2$ & $2$ \\

\hline

69 & $(1,1),(2,2),(3,3),(4,4)$ & $1{\sim}2{\sim}4$ &  &  &  &  & $2$ & $2$ \\

\hline

70 & $(1,1),(2,2),(3,3),(4,4)$ & $1{\sim}3{\sim}4$ &  &  &  &  & $2$ & $2$ \\

\hline

71 & $(1,1),(2,2),(3,3),(4,4)$ & $2{\sim}3{\sim}4$ & $36(1)$ & $36(1)$ & $36(1)$ & $36(1)$ & $2$ & $2$ \\

\hline

72 & $(1,1),(2,2),(3,3),(4,4)$ & $1{\sim}2,3{\sim}4$ &  &  &  &  & $2$ & $2$ \\

\hline

73 & $(1,1),(2,2),(3,3),(4,4)$ & $1{\sim}3,2{\sim}4$ &  &  &  &  & $2$ & $2$ \\

\hline

74 & $(1,1),(2,2),(3,3),(4,4)$ & $1{\sim}4,2{\sim}3$ &  &  &  &  & $2$ & $2$ \\

\hline

75 & $(1,1),(2,2),(3,3),(4,4)$ & $1{\sim}2{\sim}3{\sim}4$ &  &  &  &  & $3$ & $1$ \\

\hline

\multicolumn{3}{|c|}{$d_0(G)$} & 31 & 42 & 38 & 38 & \multicolumn{2}{|c}{}\\

\cline{1-7}

\end{longtable}

}

{\footnotesize

\begin{longtable}{|c|c|c|c|c|c|}

\caption{} \label{table_rank_4D} \\

\hline

No. & $(\mr M, \pi)$ & ${{\sim}}$ & $\ms D_4$  & $c(S)$ & $c(N)$ \\

\hline \endfirsthead

\caption{(continuation)} \\

\hline

No. & $(\mr M, \pi)$ & ${{\sim}}$ & $\ms D_4$  & $c(S)$ & $c(N)$ \\

\hline \endhead

1 & $(1234,1)$ &  & $11(3),17(4),26$ & $0$ & $4$ \\

\hline

2 & $(1234,2)$ &  & $5(1),13(3),19(4),31,36,46,63$ & $0$ & $4$ \\

\hline

3 & $(1234,3)$ &  & $7(1),21(4),41$ & $0$ & $4$ \\

\hline

4 & $(1234,4)$ &  & $9(1),15(3),51$ & $0$ & $4$ \\

\hline

5 & $(1,1),(234,2)$ &  & $2(1),13,19,31,36(4),46(3),63$ & $0$ & $4$ \\

\hline

6 & $(1,1),(234,2)$ & $1{\sim}234$ & $37(4),47(3),64$ & $1$ & $3$ \\

\hline

7 & $(1,1),(234,3)$ &  & $3(1),21,41(4)$ & $0$ & $4$ \\

\hline

8 & $(1,1),(234,3)$ & $1{\sim}234$ & $42(4)$ & $1$ & $3$ \\

\hline

9 & $(1,1),(234,4)$ &  & $4(1),15,51(3)$ & $0$ & $4$ \\

\hline

10 & $(1,1),(234,4)$ & $1{\sim}234$ & $52(3)$ & $1$ & $3$ \\

\hline

11 & $(124,1),(3,3)$ &  & $1(3),17,26(4)$ & $0$ & $4$ \\

\hline

12 & $(124,1),(3,3)$ & $124{\sim}3$ & $27(4)$ & $1$ & $3$ \\

\hline

13 & $(124,2),(3,3)$ &  & $2(3),5,19,31(4),36,46(1),63$ & $0$ & $4$ \\

\hline

14 & $(124,2),(3,3)$ & $124{\sim}3$ & $32(4),49(1),67$ & $1$ & $3$ \\

\hline

15 & $(124,4),(3,3)$ &  & $4(3),9,51(1)$ & $0$ & $4$ \\

\hline

16 & $(124,4),(3,3)$ & $124{\sim}3$ & $54(1)$ & $1$ & $3$ \\

\hline

17 & $(123,1),(4,4)$ &  & $1(4),11,26(3)$ & $0$ & $4$ \\

\hline

18 & $(123,1),(4,4)$ & $123{\sim}4$ & $28(3)$ & $1$ & $3$ \\

\hline

19 & $(123,2),(4,4)$ &  & $2(4),5,13,31(3),36(1),46,63$ & $0$ & $4$ \\

\hline

20 & $(123,2),(4,4)$ & $123{\sim}4$ & $33(3),39(1),68$ & $1$ & $3$ \\

\hline

21 & $(123,3),(4,4)$ &  & $3(4),7,41(1)$ & $0$ & $4$ \\

\hline

22 & $(123,3),(4,4)$ & $123{\sim}4$ & $44(1)$ & $1$ & $3$ \\

\hline

23 & $(123,3),(234,3)$ & $123{\sim}234$ &  & $1$ & $3$ \\

\hline

24 & $(123,1),(124,1)$ & $123{\sim}124$ &  & $1$ & $3$ \\

\hline

25 & $(234,4),(124,4)$ & $234{\sim}124$ &  & $1$ & $3$ \\

\hline

26 & $(12,1),(3,3),(4,4)$ &  & $1,11(4),17(3)$ & $0$ & $4$ \\

\hline

27 & $(12,1),(3,3),(4,4)$ & $12{\sim}3$ & $12(4)$ & $1$ & $3$ \\

\hline

28 & $(12,1),(3,3),(4,4)$ & $12{\sim}4$ & $18(3)$ & $1$ & $3$ \\

\hline

29 & $(12,1),(3,3),(4,4)$ & $3{\sim}4$ &  & $1$ & $3$ \\

\hline

30 & $(12,1),(3,3),(4,4)$ & $12{\sim}3{\sim}4$ &  & $2$ & $2$ \\

\hline

31 & $(12,2),(3,3),(4,4)$ &  & $2,5,13(4),19(3),36,46,63(1)$ & $0$ & $4$ \\

\hline

32 & $(12,2),(3,3),(4,4)$ & $12{\sim}3$ & $14(4),49,67(1)$ & $1$ & $3$ \\

\hline

33 & $(12,2),(3,3),(4,4)$ & $12{\sim}4$ & $20(3),39,68(1)$ & $1$ & $3$ \\

\hline

34 & $(12,2),(3,3),(4,4)$ & $3{\sim}4$ & $69(1)$ & $1$ & $3$ \\

\hline

35 & $(12,2),(3,3),(4,4)$ & $12{\sim}3{\sim}4$ & $73(1)$ & $2$ & $2$ \\

\hline

36 & $(1,1),(23,2),(4,4)$ &  & $2,5(4),13,19(1),31,46,63(3)$ & $0$ & $4$ \\

\hline

37 & $(1,1),(23,2),(4,4)$ & $1{\sim}23$ & $6(4),47,64(3)$ & $1$ & $3$ \\

\hline

38 & $(1,1),(23,2),(4,4)$ & $1{\sim}4$ & $66(3)$ & $1$ & $3$ \\

\hline

39 & $(1,1),(23,2),(4,4)$ & $23{\sim}4$ & $20(1),33,68(3)$ & $1$ & $3$ \\

\hline

40 & $(1,1),(23,2),(4,4)$ & $1{\sim}23{\sim}4$ & $71(3)$ & $2$ & $2$ \\

\hline

41 & $(1,1),(23,3),(4,4)$ &  & $3,7(4),21(1)$ & $0$ & $4$ \\

\hline

42 & $(1,1),(23,3),(4,4)$ & $1{\sim}23$ & $8(4)$ & $1$ & $3$ \\

\hline

43 & $(1,1),(23,3),(4,4)$ & $1{\sim}4$ &  & $1$ & $3$ \\

\hline

44 & $(1,1),(23,3),(4,4)$ & $23{\sim}4$ & $22(1)$ & $1$ & $3$ \\

\hline

45 & $(1,1),(23,3),(4,4)$ & $1{\sim}23{\sim}4$ &  & $2$ & $2$ \\

\hline

46 & $(1,1),(24,2),(3,3)$ &  & $2,5(3),13(1),19,31,36,63(4)$ & $0$ & $4$ \\

\hline

47 & $(1,1),(24,2),(3,3)$ & $1{\sim}24$ & $6(3),37,64(4)$ & $1$ & $3$ \\

\hline

48 & $(1,1),(24,2),(3,3)$ & $1{\sim}3$ & $65(4)$ & $1$ & $3$ \\

\hline

49 & $(1,1),(24,2),(3,3)$ & $24{\sim}3$ & $14(1),32,67(4)$ & $1$ & $3$ \\

\hline

50 & $(1,1),(24,2),(3,3)$ & $1{\sim}24{\sim}3$ & $70(4)$ & $2$ & $2$ \\

\hline

51 & $(1,1),(24,4),(3,3)$ &  & $4,9(3),15(1)$ & $0$ & $4$ \\

\hline

52 & $(1,1),(24,4),(3,3)$ & $1{\sim}24$ & $10(3)$ & $1$ & $3$ \\

\hline

53 & $(1,1),(24,4),(3,3)$ & $1{\sim}3$ &  & $1$ & $3$ \\

\hline

54 & $(1,1),(24,4),(3,3)$ & $24{\sim}3$ & $16(1)$ & $1$ & $3$ \\

\hline

55 & $(1,1),(24,4),(3,3)$ & $1{\sim}24{\sim}3$ &  & $2$ & $2$ \\

\hline

56 & $(1,1),(23,2),(24,2)$ & $23{\sim}24$ &  & $1$ & $3$ \\

\hline

57 & $(1,1),(23,2),(24,2)$ & $1{\sim}23{\sim}24$ &  & $2$ & $2$ \\

\hline

58 & $(12,2),(24,2),(3,3)$ & $12{\sim}24$ &  & $1$ & $3$ \\

\hline

59 & $(12,2),(24,2),(3,3)$ & $12{\sim}24{\sim}3$ &  & $2$ & $2$ \\

\hline

60 & $(12,2),(23,2),(4,4)$ & $12{\sim}23$ &  & $1$ & $3$ \\

\hline

61 & $(12,2),(23,2),(4,4)$ & $12{\sim}23{\sim}4$ &  & $2$ & $2$ \\

\hline

62 & $(12,2),(23,2),(24,2)$ & $12{\sim}23{\sim}24$ &  & $2$ & $2$ \\

\hline

63 & $(1,1),(2,2),(3,3),(4,4)$ &  & $2,5,13,19,31(1),36(3),46(4)$ & $0$ & $4$ \\

\hline

64 & $(1,1),(2,2),(3,3),(4,4)$ & $1{\sim}2$ & $6,37(3),47(4)$ & $1$ & $3$ \\

\hline

65 & $(1,1),(2,2),(3,3),(4,4)$ & $1{\sim}3$ & $48(4)$ & $1$ & $3$ \\

\hline

66 & $(1,1),(2,2),(3,3),(4,4)$ & $1{\sim}4$ & $38(3)$ & $1$ & $3$ \\

\hline

67 & $(1,1),(2,2),(3,3),(4,4)$ & $2{\sim}3$ & $14,32(1),49(4)$ & $1$ & $3$ \\

\hline

68 & $(1,1),(2,2),(3,3),(4,4)$ & $2{\sim}4$ & $20,33(1),39(3)$ & $1$ & $3$ \\

\hline

69 & $(1,1),(2,2),(3,3),(4,4)$ & $3{\sim}4$ & $34(1)$ & $1$ & $3$ \\

\hline

70 & $(1,1),(2,2),(3,3),(4,4)$ & $1{\sim}2{\sim}3$ & $50(4)$ & $2$ & $2$ \\

\hline

71 & $(1,1),(2,2),(3,3),(4,4)$ & $1{\sim}2{\sim}4$ & $40(3)$ & $2$ & $2$ \\

\hline

72 & $(1,1),(2,2),(3,3),(4,4)$ & $1{\sim}3{\sim}4$ &  & $2$ & $2$ \\

\hline

73 & $(1,1),(2,2),(3,3),(4,4)$ & $2{\sim}3{\sim}4$ & $35(1)$ & $2$ & $2$ \\

\hline

74 & $(1,1),(2,2),(3,3),(4,4)$ & $1{\sim}2,3{\sim}4$ &  & $2$ & $2$ \\

\hline

75 & $(1,1),(2,2),(3,3),(4,4)$ & $1{\sim}3,2{\sim}4$ &  & $2$ & $2$ \\

\hline

76 & $(1,1),(2,2),(3,3),(4,4)$ & $1{\sim}4,2{\sim}3$ &  & $2$ & $2$ \\

\hline

77 & $(1,1),(2,2),(3,3),(4,4)$ & $1{\sim}2{\sim}3{\sim}4$ &  & $3$ & $1$ \\

\hline

\multicolumn{3}{|c|}{$d_0(G)$} & 40 & \multicolumn{2}{|c}{}\\

\cline{1-4}

\end{longtable}

}

\newpage

\section*{Acknowledgements}

The author expresses his gratitude to Natalia Gorfinkel for numerous
conversations that stimulated the development of the active root
theory, and also to Ernest~B.~Vinberg for fruitful discussions.


\begin{thebibliography}{1000}

\setlength{\itemsep}{-\parsep}

\bibitem[Avd]{Avd}
R.~Avdeev, \textit{On solvable spherical subgroups of semisimple
algebraic groups}, Oberwolfach Reports \textbf{7}~(2010), no.~2,
1105--1108; see also
\href{http://arxiv.org/abs/1006.4459}{\texttt{arXiv:math.GR/1006.4459}}.

\bibitem[BP]{BP}
P.~Bravi, G.~Pezzini, \textit{Wonderful varieties of type $B$ and
$C$}, preprint~(2009), see
\href{http://arxiv.org/abs/0909.3771}{\texttt{arXiv:math.AG/0909.3771}}.

\bibitem[Bri]{Br}
M.~Brion, \textit{Classification des espaces homog\`enes
sph\'eriques}, Compositio Math. \textbf{63}~(1987), no.~2, 189--208.

\bibitem[C]{CF}
S.~Cupit-Foutou, \textit{Wonderful varieties: a geometrical
realization}, preprint~(2009), see
\href{http://arxiv.org/abs/0907.2852}{\texttt{arXiv:math.AG/0907.2852}}.

\bibitem[Kr]{Kr} M.~Kr\"amer, \textit{Sph\"arische Untergruppen in kompakten
zusammenh\"angenden Liegruppen}, Compositio Math.
\textbf{38}~(1979), no.~2, 129--153.

\bibitem[Lu1]{Lu93}
D.~Luna, \textit{Sous-groupes sph{\'e}riques r{\'e}solubles},
Pr{\'e}publication de l'Institut Fourier no.~241, 1993.

\bibitem[Lu2]{Lu01}
D.~Luna, \textit{Vari\'et\'es sph\'eriques de type $A$}, IH\'ES
Publ. Math. \textbf{94}~(2001), 161--226.

\bibitem[Mik]{Mi}
I.\,V.~Mikityuk, \textit{On the integrability of invariant
Hamiltonian systems with homogeneous configuration spaces}, Math.
USSR-Sb. \textbf{57}~(1987), no.~2, 527--546; Russian original:
И.\,В.~Микитюк, \textit{Об интегрируемости инвариантных
гамильтоновых систем с однородными конфигурационными
пространствами}, Мат. сб. \textbf{129(171)}~(1986), №\,4, 514--534.

\bibitem[Mon]{Mon}
P.-L.~Montagard, \textit{Une nouvelle propri\'et\'e de stabilit\'e
du pl\'ethysme}, Comment. Math. Helvetici. \textbf{71}~(1996),
475--505.

\bibitem[OV]{VO}
A.\,L.~Onishchik and E.\,B.~Vinberg, \textit{Lie groups and
algebraic groups}, Springer Ser. Soviet Math., Springer-Verlag,
Berlin 1990; Russian original (2nd edition): Э.\,Б.~Винберг,
А.\,Л.~Онищик, \textit{Семинар по группам Ли и алгебраическим
группам}, Москва, УРСС, 1995.

\bibitem[Vin]{Vin}
E.\,B.~Vinberg, \textit{Commutative homogeneous spaces and
co-isotropic symplectic actions}, Russian Math. Surveys
\textbf{56}~(2001), no.~1, 1--60; Russian original: Э.\,Б.~Винберг,
\textit{Коммутативные однородные пространства и коизотропные
симплектические действия}, Успехи мат. наук \textbf{56}~(2001),
№\,1, 3--62.

\bibitem[Yak]{Yak}
O.\,S.~Yakimova, \textit{Weakly symmetric spaces of semisimple Lie
groups}, Moscow Univ. Math. Bull. \textbf{57}~(2002), no.~2, 37--40;
Russian original: О.\,С.~Якимова, \textit{Слабо симметрические
пространства полупростых групп Ли}, Вестник Моск. ун-та. Сер.~1,
Матем. Мех., 2002, №\,2, 57--60.

\end{thebibliography}
\end{document}